\documentclass[a4paper,10pt]{article}
\usepackage{amsmath}
\usepackage[a4paper,left=2cm,right=2cm,top=2cm,bottom=3cm]{geometry}
\usepackage{amsthm}
\usepackage{bm}
\usepackage{amssymb}
\usepackage[utf8]{inputenc}
\usepackage{authblk}
\usepackage{graphicx}
\usepackage{color}
\usepackage{tikz}
\usepackage{pgfplots}
\usepackage{algorithm2e}
\usepackage{tabularx}
\usepackage{url}
\usepackage{xparse}
\usepackage{multirow}
\usepackage{dsfont}
\pgfplotsset{compat=1.13}
\usepackage{mathtools}
\usepackage[outline]{contour}
\usepackage[multiple]{footmisc}

\theoremstyle{plain}
\newtheorem{thrm}{Theorem}[section]
\newtheorem{lmm}[thrm]{Lemma}
\newtheorem{crllr}[thrm]{Corollary}
\newtheorem{prpstn}[thrm]{Proposition}

\theoremstyle{definition}
\newtheorem{dfntn}[thrm]{Definition}

\newtheorem{xmpl}[thrm]{Example}

\newtheorem{rmrk}[thrm]{Remark}

\theoremstyle{plain}

\newcommand{\Id}{\mathds{1}}
\newcommand{\Ide}{\text{Id}}

\newcommand{\R}{\mathbb{R}}
\newcommand{\N}{\mathbb{N}}
\newcommand{\BW}{\mathbb{R}_{+, \text{sym}}^{d\times d}}

\newcommand{\BWt}{\mathbb{R}_{+, \text{sym}}^{2\times 2}}
\newcommand{\BWd}{\mathbb{R}_{+, \text{dia}}^{d\times d}}
\DeclareMathOperator*{\argmin}{\operatorname{argmin}}
\newcommand{\tr}{\operatorname{tr}}

\newcommand{\diag}{\operatorname{diag}}

\renewcommand{\d}{\,\mathrm{d}}
\newcommand{\bigO}{\mathcal{O}}
\newcommand{\Wass}{\mathcal{W}}
\newcommand{\WassEps}{\mathbf{W}_\epsilon}
\newcommand{\Prob}{\mathcal{P}}
\newcommand{\ProbW}{\mathcal{P}_2}
\newcommand{\ProbAC}{\mathcal{P}_2^{ac}}
\newcommand{\ProbWGd}{\mathcal{P}^{G,d}_2}

\newcommand{\std}{\operatorname{std}}
\newcommand{\cov}{\operatorname{cov}}
\newcommand{\avg}{\operatorname{mean}}

\newcommand{\flowmap}{\mathcal{T}}

\newcommand{\domain}{\Omega}

\newcommand{\PathEnergy}{\mathbf{E}}
\newcommand{\PathEnergyt}{\hat{\mathbf{E}}}
\newcommand{\SplineEnergy}{\mathbf{F}}
\newcommand{\SplineEnergyt}{\hat{\mathbf{F}}}
\newcommand{\SplineEnergyG}{\mathbf{F}_G}

\newcommand{\pathenergy}{\mathcal{E}}
\newcommand{\pathenergyt}{\hat{\mathcal{E}}}
\newcommand{\splineenergy}{\mathcal{F}}
\newcommand{\splineenergyt}{\hat{\mathcal{F}}}

\newcommand{\discreteDomain}{\Omega_{\scriptscriptstyle{MN}}}

\newcommand{\Bary}{\textup{Bar}}
\newcommand{\BaryEps}{\mathbf{Bar}^\epsilon}

\definecolor{blue}{rgb}{0.137255,0,0.862745}
\definecolor{red}{rgb}{0.862745,0.137255,0}
\definecolor{myOrange}{RGB}{255, 169, 87}
\definecolor{myGreen}{RGB}{180, 255, 162}
\definecolor{myGrey}{RGB}{187, 187, 187}
\definecolor{myDarkGrey}{RGB}{100, ,100, 100}

\makeatletter
\DeclareRobustCommand\onedot{\futurelet\@let@token\@onedot}
\def\@onedot{\ifx\@let@token.\else.\null\fi\xspace}

\def\etal{\emph{et al}\onedot}
\makeatother

\allowdisplaybreaks
\numberwithin{equation}{section}

\newcommand*\samethanks[1][\value{footnote}]{\footnotemark[#1]}
\begin{document}

\title{Approximation of Splines in Wasserstein Spaces}
\author[,1]{Jorge Justiniano\thanks{J.J. and M.R. are funded by the Deutsche Forschungsgemeinschaft (DFG, German Research Foundation) – Project-ID 211504053 – SFB 1060}} \author[,1]{Martin Rumpf\samethanks}
\author[,2]{Matthias Erbar\thanks{M.E. is funded by the Deutsche Forschungsgemeinschaft (DFG, German Research Foundation) – Project-ID 317210226 – SFB 1283}}

\affil[1]{Institute for Numerical Simulation, University of Bonn, Endenicher Allee 60, 53115 Bonn, Germany \authorcr
  \tt jorge.justiniano@ins.uni-bonn.de, martin.rumpf@ins.uni-bonn.de}
\affil[2]{Fakult\"at f\"ur Mathematik, Universit\"at Bielefeld,
              Postfach 100131, 33501 Bielefeld, Germany \authorcr
  \tt matthias.erbar@math.uni-bielefeld.de}

\date{\today}

\maketitle
\begin{abstract} 
This paper investigates a time discrete variational model for splines in Wasserstein spaces to interpolate probability measures.
Cubic splines in Euclidean space are known to minimize the 
integrated squared acceleration subject to a set of interpolation constraints. 
As generalization on the space of probability measures 
the integral of the squared acceleration is considered as a spline energy and regularized by addition of the usual action functional. 
Both energies are then discretized in time using local Wasserstein-2 distances and the generalized Wasserstein barycenter.
The existence of time discrete regularized splines for given interpolation conditions is established. On the subspace of Gaussian distributions, the spline interpolation problem is solved explicitly and consistency in the discrete to continuous limit is shown.
The computation of time discrete splines is implemented numerically, based on entropy regularization and the Sinkhorn algorithm. 
A variant of Nesterov's accelerated gradient descent algorithm is applied for the minimization of the fully discrete functional.
A variety of numerical examples demonstrate the robustness of the approach and show striking characteristics of the method.
As a particular application the spline interpolation for synthesized textures is presented.
\end{abstract}
\section{Introduction}\label{sec:intro}
In this paper we will study a time discrete variational model to compute spline paths in the space of probability measures equipped with the Wasserstein-2 metric. The spline paths are defined as measure-valued paths minimizing a spline energy subject to interpolation constraints and boundary conditions.

In the last decade, higher-order interpolation methods attracted a lot of attention in time-sequence interpolation or regression in the context of data analysis. Applications are for instance in computer graphics, computer vision, or medical imaging. The objects to be interpolated are usually considered as shapes 
in some infinite dimensional manifold equipped with an application dependent Riemannian metric.
One approach is to consider  a spline energy functional as a second order extension of the first order
path energy on the Riemannian manifold. Given a set of objects -- from now on called key frames --  at disjoint times 
a spline curve is then defined as a minimizer of the spline energy subject to the key frame interpolation constraint. 

In Euclidean space, cubic splines
$x: [0,1] \to \R^d$ are known to be minimizers of the integral of the squared acceleration $\int_0^1 \vert \ddot x \vert^2 \d t$ due to a famous result by de Boor \cite{dB63}. Our proposed method can be seen as a generalization of de Boor's result to the Wasserstein space in discrete time. To see this, we use a simple rectangular quadrature rule to replace the integral, and the second order central difference to approximate the acceleration of a curve, to obtain
\begin{align*}
\int_0^1\vert\ddot x\vert^2\d t\approx 4K^3\sum_{k=1}^{K-1}\left\vert x_k-\frac{x_{k-1}+x_{k+1}}{2}\right\vert^2.
\end{align*}
In our case, we are interested in discrete measure-valued curves. Hence, it is natural to replace the Euclidean norm $\vert\cdot\vert$ with the Wasserstein $L^2$ distance between probability measures, and use a notion of barycenter between measures $\mu_{k-1}$ and $\mu_{k+1}$, denoted as $\Bary(\mu_{k-1},\mu_{k+1})$, instead of the middle point $\frac{x_{k-1}+x_{k+1}}{2}$. The proposed discrete spline functional studied in this paper will therefore be given by
\begin{align}
4K^3\sum_{k=1}^{K-1}\Wass^2(\mu_k, \Bary(\mu_{k-1},\mu_{k+1})).
\end{align}

Noakes \etal \cite{NoHePa89} generalized de Boor's result in a finite dimensional Riemannian context, introducing Riemannian cubic splines as stationary paths of the integrated squared covariant derivative of the velocity field of a path. To define continuous splines on the space of probability measures as investigated here, a similar geometric structure is required. Fortunately, the celebrated Benamou-Brenier formula \cite{BeBr00} allows for a formal endowment of a Riemannian structure on the Wasserstein space $\ProbW(\R^d)$, as originally described in \cite{Ot01}. Beyond this scope, however, a second order analysis of paths in Wasserstein spaces is required. Such an analysis has been developed by Gigli in \cite{Gi12}.
This approach in particular allows to define the acceleration of a curve of measures as the covariant derivative of its velocity.  Energy splines (E-splines) are then defined as minimizers of the total squared acceleration, see \eqref{def:spline_energy_cont}
for more details.  This functional, however, is not well-behaved as it is computationally intractable, and, unlike the usual action functional \eqref{eq:path_energy}, not convex. Thus, some relaxations thereof have been recently proposed:

Both Benamou \etal in \cite{BeGaVi19} and Chen \etal in \cite{ChCoGe18} independently introduced the concept denoted now as P-splines, short for path splines. 
P-splines are $\R^d$-valued stochastic processes $(X_t)_{t\in[0,1]}$ defined on some underlying probability space $(\domain, \mathbb{P})$ that solve the following minimization problem
\begin{equation}\label{eq:P-splines}
\min_{(X_t)_t}\int_0^1\int_\domain\Vert \ddot{X}_t\Vert^2 \d\mathbb{P}\d t,
\end{equation}
subject to $I\geq 2$ given marginal constraints $X_{\overline{t}_i}\sim\overline{\mu}_{i}$ for prescribed times $0= \overline{t}_1<\ldots<\overline{t}_I= 1$, and given probability measures $\overline{\mu}_{i}$ for all $i=1,\ldots,I$. 

Numerically, this computationally demanding task is solved via a relaxation based on multi-marginal optimal transport with quadratic cost and entropic regularization.
A drawback of this method is that a solution of \eqref{eq:P-splines} might fail to be deterministic, i.e. there is no guarantee that a Monge map $\phi_t:\R^d\rightarrow\R^d$ exists, such that $X_t=\phi_t(X_0)$, even if the marginal constraints are regular, see Chewi et al.~\cite{ChClGo+20}. 
In fact, any solution $(X_t)_t$ of \eqref{eq:P-splines} has spline trajectories $t\mapsto X_t(\omega)$ for $\mathbb{P}$-almost all $\omega\in\domain$, which formally follows from rewriting \eqref{eq:P-splines}:
\begin{align}\nonumber
\min_{{(X_t)_t} \atop {X_{t_i}\sim\overline{\mu}_i}} \int_0^1\int_\domain\Vert \ddot{X}_t\Vert^2 \d\mathbb{P}\d t 
&=
\min_{{Q\in \mathcal P(\R^{d\cdot I}) }\atop{ (\pi_i)_\#Q=\overline{\mu}_i}} \int_{\R^{d\cdot I}} \min_{t\mapsto y_t \atop y_{t_i}=x_i} \int_0^1\|\ddot {y}_t\|^2\d t \d Q(x_1,\dots, x_I)\\
&= \min_{{Q\in \mathcal P(\R^{d\cdot I}) }\atop{ (\pi_i)_\#Q=\overline{\mu}_i}} \int_{\R^{d\cdot I}}\int_0^1\Vert \ddot{s}_t\Vert^2\d t
\d Q(x_1,\dots, x_I),
\end{align}
where $t\mapsto s_t$ is the classical Euclidean spline interpolating the points $(t_i,x_i)$ and $\pi_i:\R^{d\cdot I}\rightarrow\R^d$ is the projection onto the $i$-th $d$-sized batch of coordinates, i.e. $\pi_i(y_1,\ldots,y_{d\cdot I})=(y_{d\cdot i+1},\ldots,y_{d\cdot i+d})$. Hence, any solution $(X_t)_t$ of \eqref{eq:P-splines} has spline trajectories $t\mapsto X_t(\omega)$ for $\mathbb{P}$-almost all $\omega\in\domain$ (see Fig. 1, top left).  

A different approach introduced by Chewi et al.~\cite{ChClGo+20} to construct measure-valued splines remedies this shortcoming and introduces so called transport splines (T-splines), where one studies the smooth interpolation of probability measures in the optimal transport context using a particle flow approach. To this end, samples are drawn from one distribution to be interpolated, usually $X_0\sim\mu_0$. These samples are then pushed by the Monge maps $T_i$ between consecutive prescribed distributions $\overline{\mu}_i, \overline{\mu}_{i+1}$ and the resulting chains of points $(T_i\circ\ldots\circ T_0\circ X_0)(\omega)$ at the prescribed times $\overline{t}_i$, for $i=1,\ldots, I,$ are interpolated using classical cubic spline interpolation (see Fig. 1 top middle and bottom left). In \cite{ChClGo+20}, a relation to energy splines is investigated for Gaussian distributions in the one dimensional case. This method enjoys computational advantages.

From a theoretical point of view, both the P-spline and T-spline approaches are based on the Lagrangian perspective of optimal transport. Hence, instead of directly minimizing probability measures, they work with stochastic processes $X_t$ that have $C^2$ sample paths and laws $\mu_t$. This flow perspective may be a more natural choice for some applications, since one is able to easily track particle trajectories in continuous time. In contrast, E-splines can be seen as working on the Eulerian perspective of optimal transport, as we are able to track densities and particle velocities passing through any fixed time and spatial position. In the aforementioned papers,  algorithms are given to compute sample trajectories of P-splines and T-splines. In this work, we propose a consistent variational time approximation of E-splines, and devise an algorithm on how to construct them. We shall prove that this approximation is consistent with the Riemannian geometry of the Wasserstein space in the Gaussian case. Moreover, we are able to construct very simple counterexamples in $1$D, where E-splines differ from P-splines and/or T-splines, and in both cases our approach properly and exactly (up to machine accuracy) matches the theoretical value of the E-spline, see Fig. \ref{fig:GaussianCounterexamples}. 
\begin{figure*}[htb]
\resizebox{\linewidth}{!}{
\tikzstyle{frame} = [line width=1.8pt, draw=red,inner sep=0.01em]
\begin{tikzpicture}
\begin{scope}[scale=1.0]

\begin{scope}[shift={(0,-11.02)}]
\node[anchor=south west] at (0.,0.) 
{\includegraphics[width=0.3\linewidth]{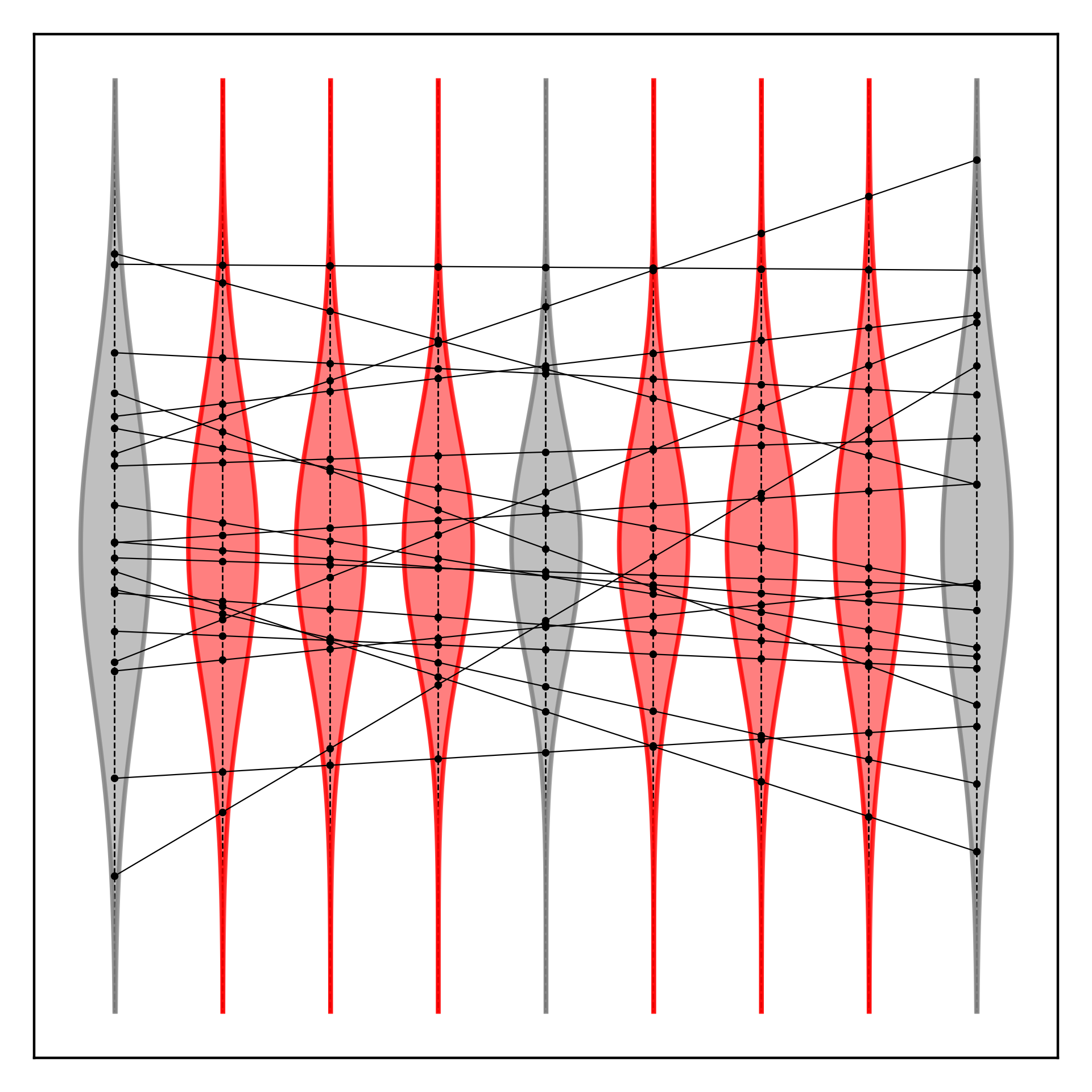}};
\node[anchor=south west] at (5.22,0.) 
{\includegraphics[width=0.3\linewidth]{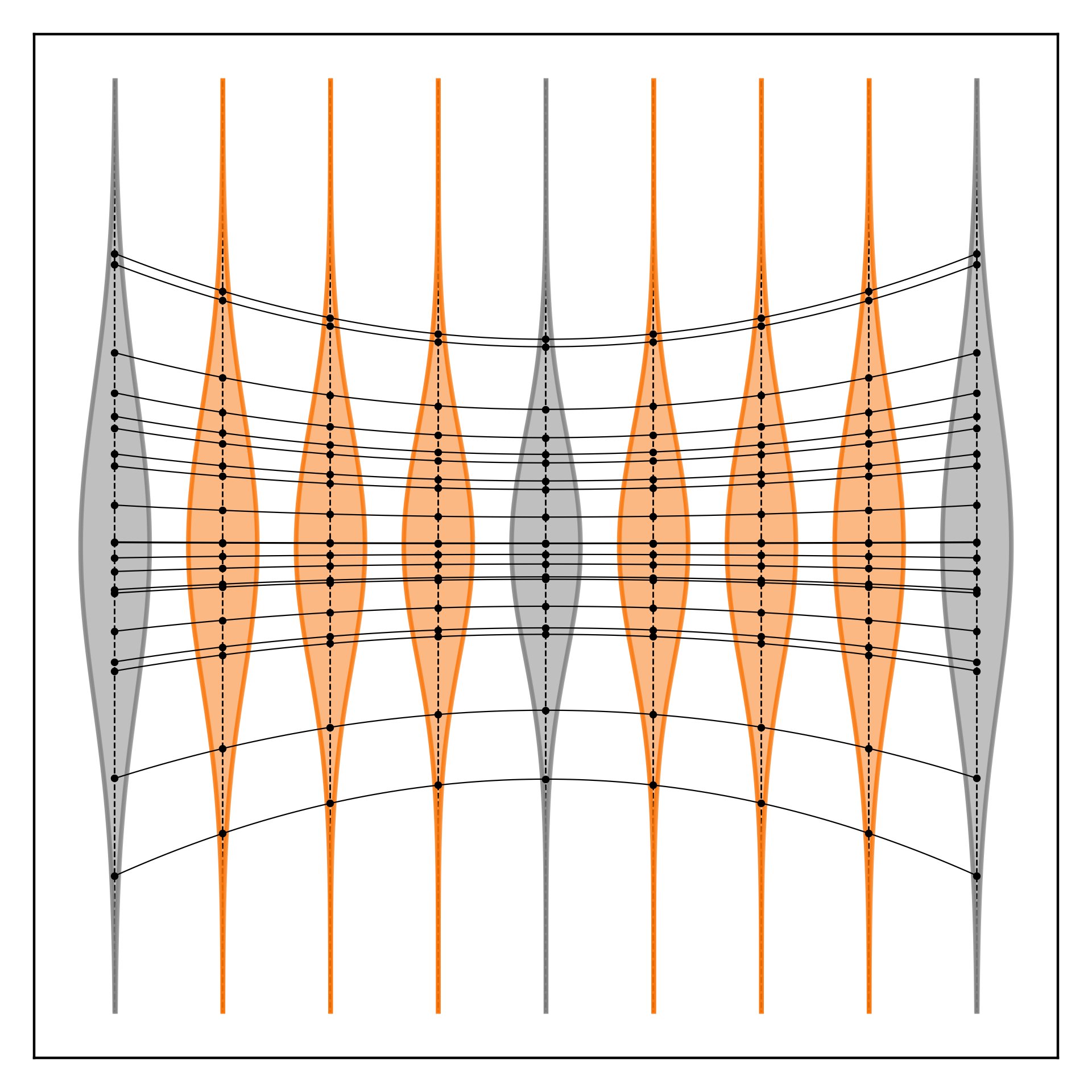}};
\node[anchor=south west] at (10.44,0.)
{\includegraphics[width=0.3\linewidth]{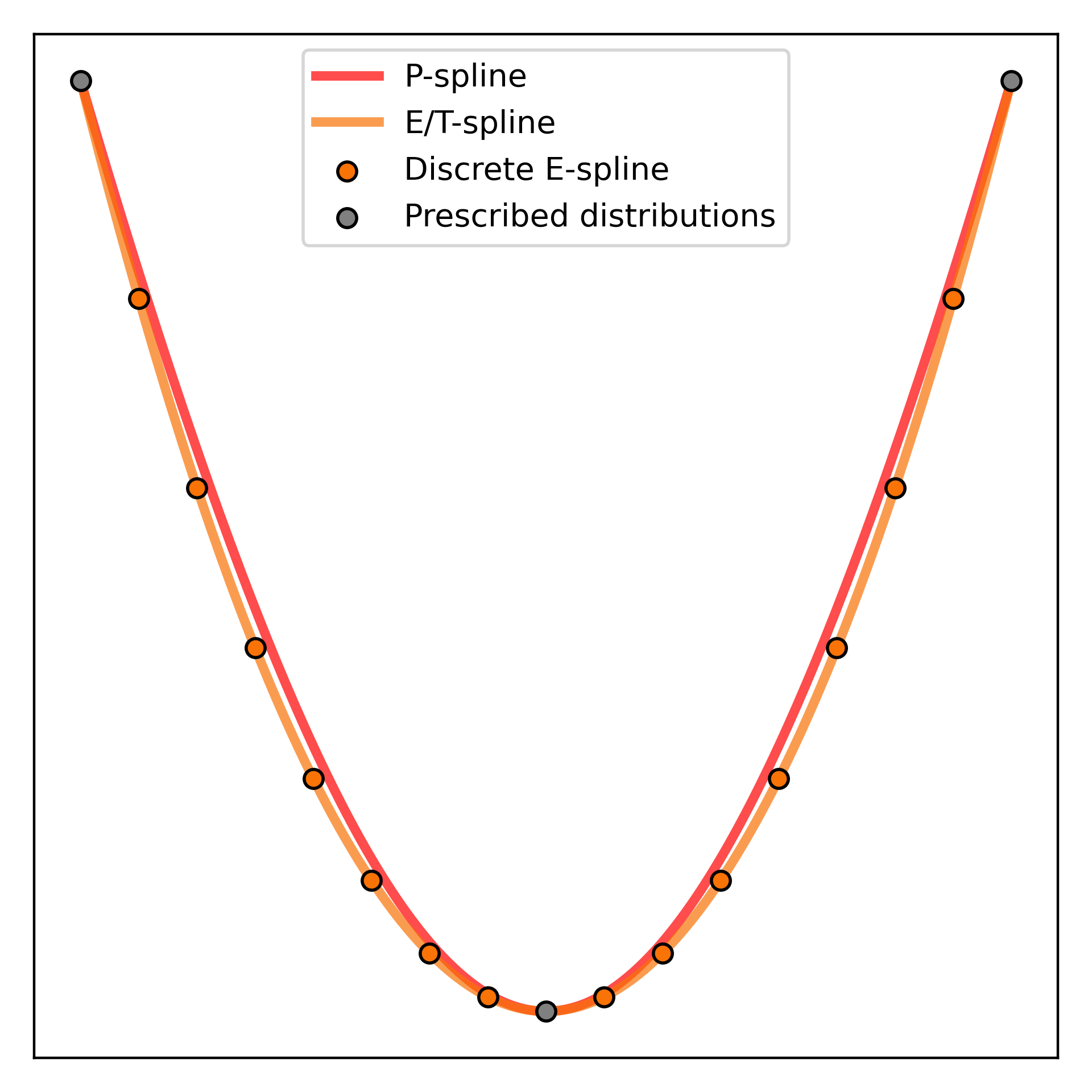}};
\node at (0.1, 0.1) {\fontsize{2}{2.5}\selectfont $t$};
\node at (0.68, 0.1) {\fontsize{2}{2.5}\selectfont $0$};
\node at (1.68, 0.1) {\fontsize{2}{2.5}\selectfont $0.25$};
\node at (2.68, 0.1) {\fontsize{2}{2.5}\selectfont $0.5$};
\node at (3.69, 0.1) {\fontsize{2}{2.5}\selectfont $0.75$};
\node at (4.7, 0.1) {\fontsize{2}{2.5}\selectfont $1$};
\node at (2.68, -0.25) {\tiny (a) P-spline with sample trajectories};
\begin{scope}[shift={(5.22,0.)}]
\node at (0.1, 0.1) {\fontsize{2}{2.5}\selectfont $t$};
\node at (0.68, 0.1) {\fontsize{2}{2.5}\selectfont $0$};
\node at (1.68, 0.1) {\fontsize{2}{2.5}\selectfont $0.25$};
\node at (2.68, 0.1) {\fontsize{2}{2.5}\selectfont $0.5$};
\node at (3.69, 0.1) {\fontsize{2}{2.5}\selectfont $0.75$};
\node at (4.7, 0.1) {\fontsize{2}{2.5}\selectfont $1$};
\node at (2.68, -0.25) {\tiny (b) E/T-spline with sample trajectories};
\end{scope}
\begin{scope}[scale=1.1, shift={(9.24,0.)}]
\node at (0.4, 0.1) {\fontsize{2}{2.5}\selectfont $t$};
\node at (0.68, 0.1) {\fontsize{2}{2.5}\selectfont $0$};
\node at (1.68, 0.1) {\fontsize{2}{2.5}\selectfont $0.25$};
\node at (2.68, 0.1) {\fontsize{2}{2.5}\selectfont $0.5$};
\node at (3.69, 0.1) {\fontsize{2}{2.5}\selectfont $0.75$};
\node at (4.7, 0.1) {\fontsize{2}{2.5}\selectfont $1$};
\node at (5.1, 0.25) {\fontsize{2}{2.5}\selectfont $0.7$};
\node at (5.1, 1.63) {\fontsize{2}{2.5}\selectfont $0.8$};
\node at (5.1, 3.015) {\fontsize{2}{2.5}\selectfont $0.9$};
\node at (5.1, 4.4) {\fontsize{2}{2.5}\selectfont $1$};
\node at (5.1, 4.7) {\fontsize{2}{2.5}\selectfont $\sigma$};
\node at (2.68, -0.25) {\tiny (c) Standard deviations of (a), (b)};
\end{scope}
\end{scope}

\begin{scope}[shift={(0,-17.04)},scale=1.0]
\node[anchor=south west] at (0.,0.) 
{\includegraphics[width=0.3\linewidth]{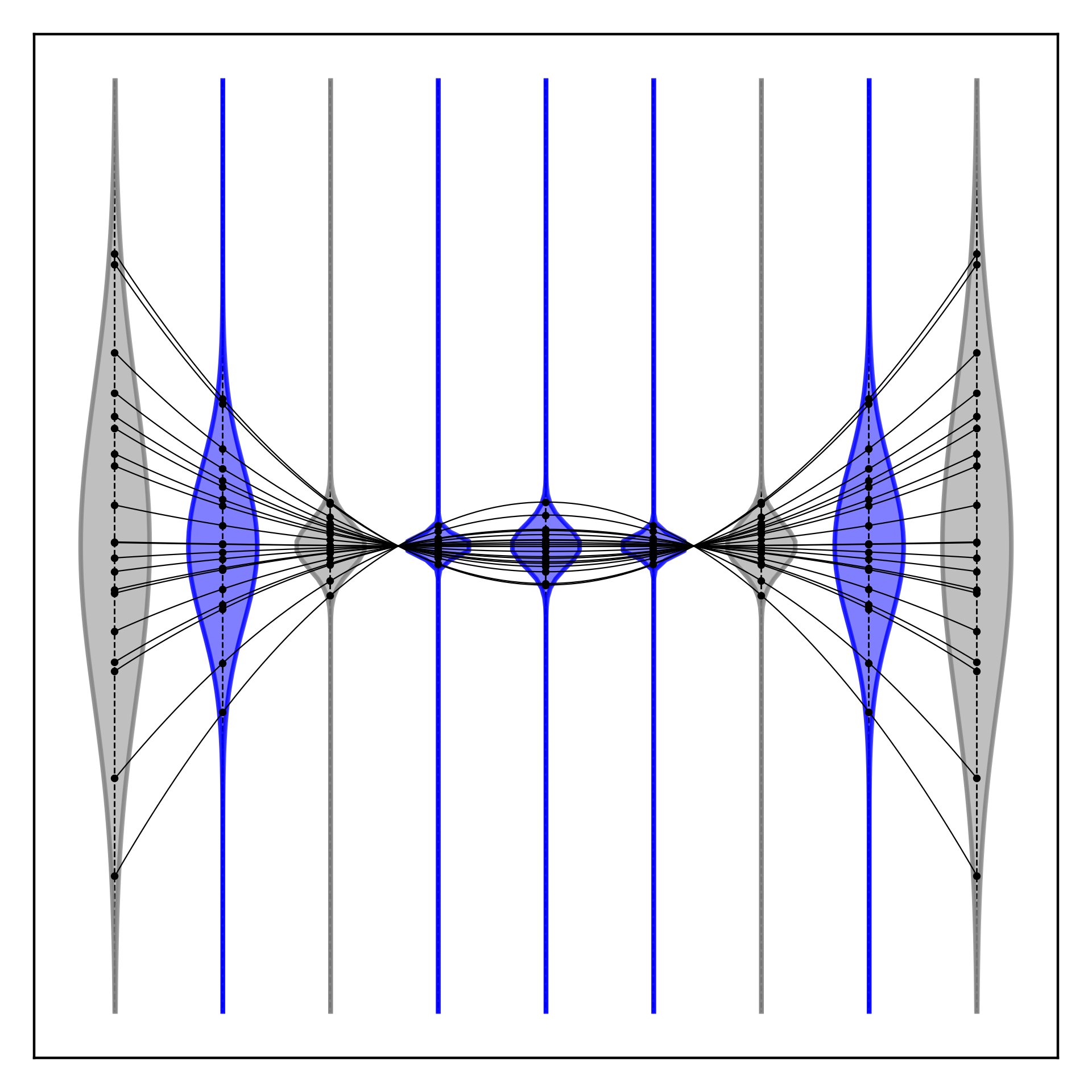}};
\node[anchor=south west] at (5.22,0.) 
{\includegraphics[width=0.3\linewidth]{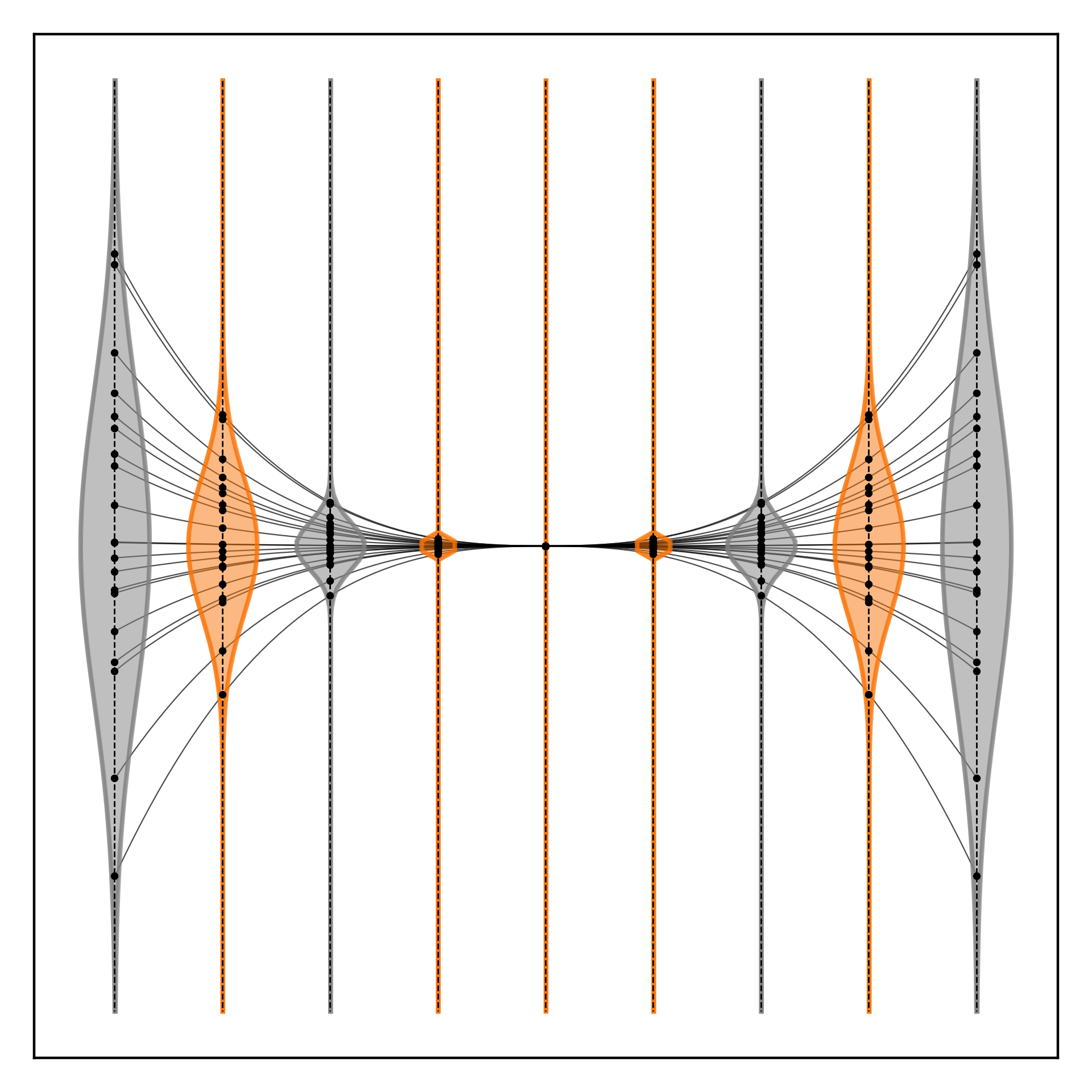}};

\node[anchor=south west] at (10.44,0.)
{\includegraphics[width=0.3\linewidth]{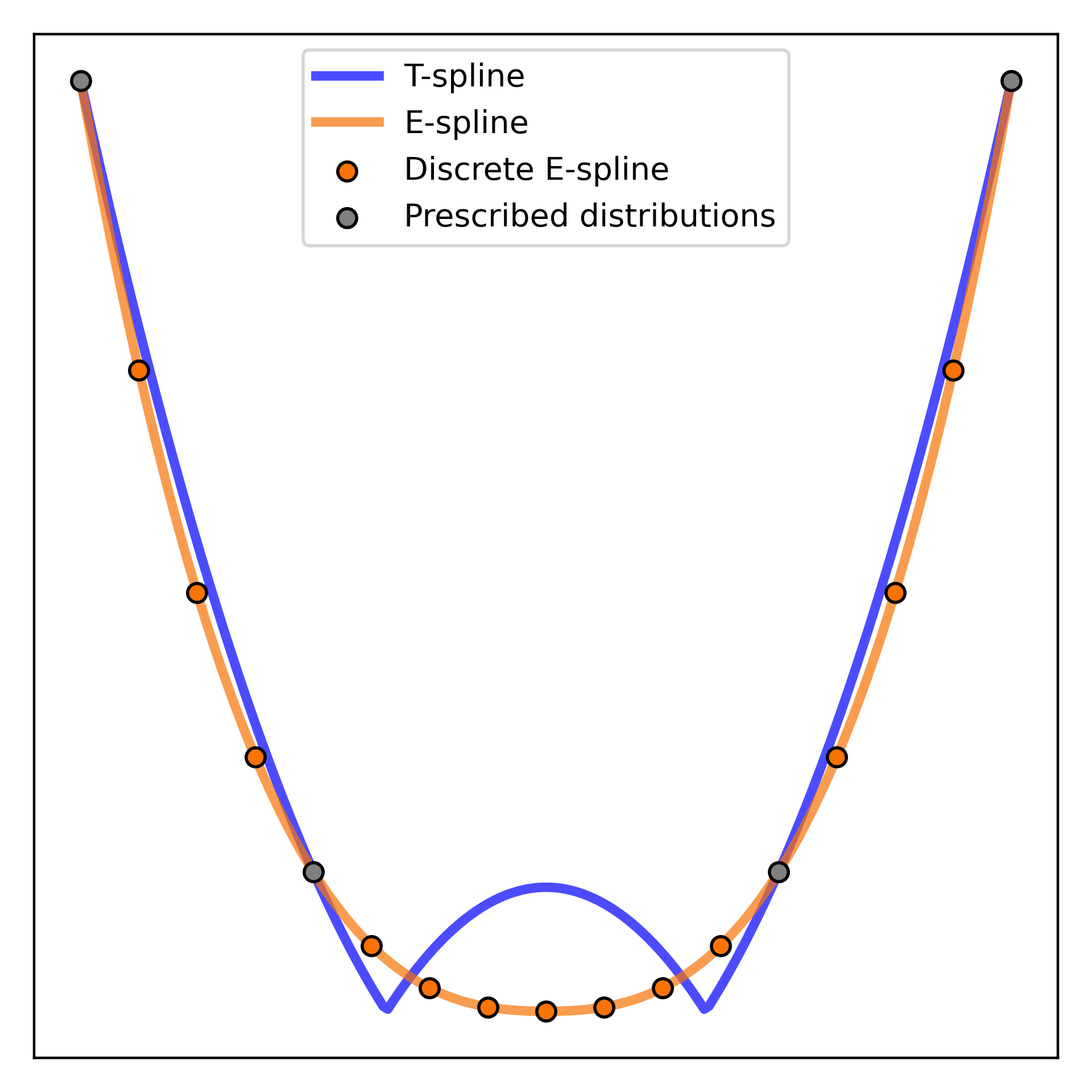}};
\node at (0.1, 0.1) {\fontsize{2}{2.5}\selectfont $t$};
\node at (0.68, 0.1) {\fontsize{2}{2.5}\selectfont $0$};
\node at (1.68, 0.1) {\fontsize{2}{2.5}\selectfont $0.25$};
\node at (2.68, 0.1) {\fontsize{2}{2.5}\selectfont $0.5$};
\node at (3.69, 0.1) {\fontsize{2}{2.5}\selectfont $0.75$};
\node at (4.7, 0.1) {\fontsize{2}{2.5}\selectfont $1$};
\node at (2.68, -0.25) {\tiny (d) T-spline with sample trajectories};
\begin{scope}[shift={(5.22,0.)}]
\node at (0.1, 0.1) {\fontsize{2}{2.5}\selectfont $t$};
\node at (0.68, 0.1) {\fontsize{2}{2.5}\selectfont $0$};
\node at (1.68, 0.1) {\fontsize{2}{2.5}\selectfont $0.25$};
\node at (2.68, 0.1) {\fontsize{2}{2.5}\selectfont $0.5$};
\node at (3.69, 0.1) {\fontsize{2}{2.5}\selectfont $0.75$};
\node at (4.7, 0.1) {\fontsize{2}{2.5}\selectfont $1$};
\node at (2.68, -0.25) {\tiny (e) E-spline with sample trajectories};
\end{scope}
\begin{scope}[scale=1.1, shift={(9.24,0.)}]
\node at (0.4, 0.1) {\fontsize{2}{2.5}\selectfont $t$};
\node at (0.68, 0.1) {\fontsize{2}{2.5}\selectfont $0$};
\node at (1.68, 0.1) {\fontsize{2}{2.5}\selectfont $0.25$};
\node at (2.68, 0.1) {\fontsize{2}{2.5}\selectfont $0.5$};
\node at (3.69, 0.1) {\fontsize{2}{2.5}\selectfont $0.75$};
\node at (4.7, 0.1) {\fontsize{2}{2.5}\selectfont $1$};
\node at (5.1, 0.4) {\fontsize{2}{2.5}\selectfont $0$};
\node at (5.1, 1.2) {\fontsize{2}{2.5}\selectfont $0.2$};
\node at (5.1, 2.) {\fontsize{2}{2.5}\selectfont $0.4$};
\node at (5.1, 2.8) {\fontsize{2}{2.5}\selectfont $0.6$};
\node at (5.1, 3.6) {\fontsize{2}{2.5}\selectfont $0.8$};
\node at (5.1, 4.4) {\fontsize{2}{2.5}\selectfont $1$};
\node at (5.1, 4.7) {\fontsize{2}{2.5}\selectfont $\sigma$};
\node at (2.68, -0.25) {\tiny (f) Standard deviations of (d), (e)};
\end{scope}
\end{scope}
\end{scope}
\end{tikzpicture}
}
\caption{A comparison of the different spline models sampled at nine equidistant times 
in 1D for Gaussian probability distributions as interpolation constraints depicted in grey. Sampled random variables are drawn as black dots, and their optimal sample trajectories are depicted by the connecting black curves:
top left: continuous P-spline (red); top middle: continuous E-spline/T-spline (orange) sampled at nine equidistant times; top right: standard deviations for both the P-spline (red) and E-spline/T-spline (orange). Orange dots represent the discrete values obtained with our method; bottom left: continuous T-spline (blue) sampled at nine equidistant times; bottom middle: continuous E-spline (orange) sampled at nine equidistant times; bottom right: standard deviations for both the T-spline (blue) and E-spline (orange). Orange dots represent the discrete values obtained with our  method.}
\label{fig:GaussianCounterexamples}
\end{figure*}
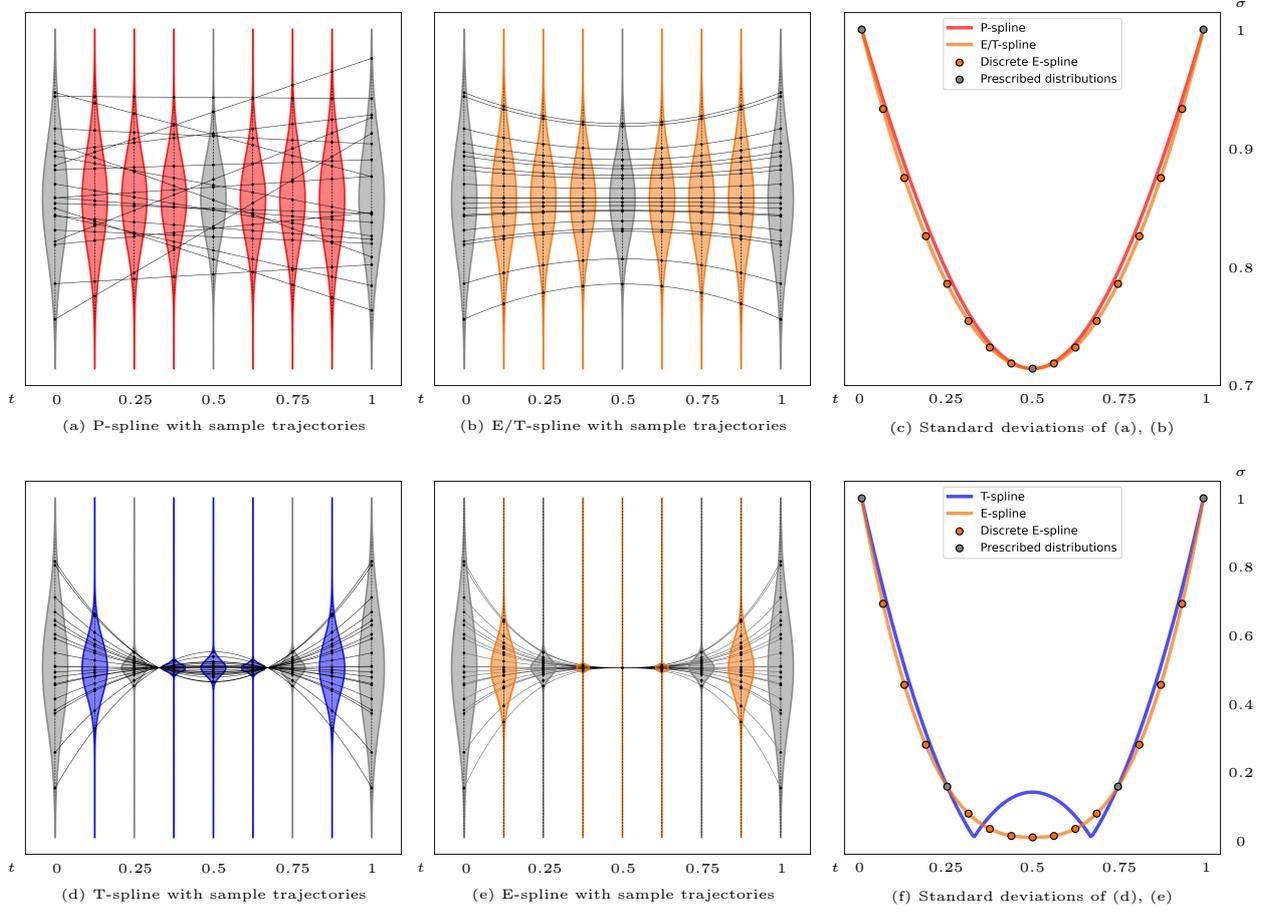 
Nowadays, there is a variety of spline approaches in non-linear spaces.
Trouv\'e and Vialard~\cite{TrVi12} investigated a second-order shape functional in landmark space based on a taylored optimal control approach.  Singh et al.~\cite{SiNiVi15} introduced an optimal control method involving a functional which measures the motion acceleration in a flow of  diffeomorphisms ansatz for image regression. Tahraoui and Vialard~\cite{TaVi19} consider a second-order variational model on the group of diffeomorphisms involving the Eulerian acceleration in the context of diffeomorphic flow.
They proposed a relaxed model leading to a Fisher-Rao functional, as a convex functional on the space of measures. 
Vialard \cite{Vi20} showed the existence of a minimizer of the Riemannian acceleration energy on the group of diffeomorphisms endowed with a right-invariant Sobolev metric of high order. 

For smooth temporal interpolation of data distributions
Chen and Karlsson \cite{ChKa18} studied an optimal control problem subject to the transport problem and interpolation conditions as constraints. 
Thereby, they in particular consider the transport of Gaussian distributions. 
In his thesis, Julien Clancy \cite{Cl21} compared different spline approaches in the space of probability measures. 
He investigated entropy regularization and extended the approach to the spline interpolation for unbalanced measures.

An alternative higher order approximation approach has been presented by Karimi and Georgiou in \cite{KaGe21}. 
They considered a regression problem for polynomial measure-valued curves and a probability law on such curves to 
approximate distributional snapshots. This approach can be viewed as a least-squares regression in Wasserstein space where a multi-marginal optimal transport formulation leads to a linear program and the Sinkhorn algorithm allows for the efficient computation in a entropy-regularized set up. Zhang and Noakes \cite{ZhNo19} investigated Riemannian cubic splines in the manifold of symmetric positive definite matrices using Lie algebra calculus and the Riemannian geometry on the space of Gaussian densities (the Bures-Wasserstein manifold \cite{Bu69}, \cite{FoKi16}, \cite{BhJa17}) induced by the Wasserstein distance \cite{MaMo18}.

Moreover, Rajkovi\'c \etal \cite{JuRaRu23} studied the spline interpolation of images where the underlying shape space reflects the metamorphosis model. The model separates in a physically intuitive way the Eulerian flow acceleration and the second material derivative of the image intensity.
The resulting model is not Riemannian in the sense that 
splines are minimizers of the squared covariant derivative of the path velocity as in \cite{NoHePa89,TaVi19,Vi20}. 
In fact, the covariant derivative of the path velocity in the Riemannian metric would lead to an interwoven model of the different types of acceleration.
A rigorous convergence analysis in terms of Mosco convergence~\cite{Mo69}, a stronger variant of $\Gamma$-convergence, of the time discrete to the time continuous metamorphosis model is presented in \cite{JuRaRu23}. This way also the existence of minimizers of the continuous spline energy is established.

Our time discretization of energy splines will rely on
a general theory for a variational time discretization of splines on Riemannian manifolds that has been proposed in \cite{HeRuWi18}. The core ingredients of the general spline discretization proposed therein are a functional $\Wass$ which approximates the squared Riemannian distance of two nearby objects on the manifold.
This approach has been applied in \cite{HeRuSc16} in computer graphics to the smooth interpolation of triangulated surfaces using the concept of discrete thin shells. In the context of probability measures the local functional $\Wass$ will be the squared Wasserstein distance and the approximate average will be the Wasserstein barycenter. 

\paragraph{Organization. }
This paper is organized as follows. 
In Section~\ref{sec:review} we will briefly review the Wasserstein distance between probability measures, the Riemannian perspective on Wasserstein spaces and the flow formulation of optimal transport. 
In Section~\ref{sec:splines_wass_spaces} the time continuous spline energy is derived using the Riemannian perspective and a variational time discretization of the continuous spline energy is introduced.
Section~\ref{sec:gaussian} expands on the special case of Gaussian distributions, proving consistency of the discrete functional with the continuous one.  Then, in Section~\ref{sec:mosco_convergence} we prove convergence of temporally extended discrete Wasserstein spline energies to time-continuous ones in the sense of Mosco for Gaussian distributions with diagonal covariance matrices.
Section~\ref{sec:fully_discrete} explains the fully discrete scheme which relies on the Sinkhorn algorithm and shows how to set up suitable variants of accelerated gradient algorithms \cite{Ne05} to numerically solve for a spline interpolation given a set of key frames. Moreover, experimental results of the application of this algorithm on probability measures are presented. 
Finally, Section~\ref{sec:textures} experimentally demonstrates the versatility of our approach by applying the spline interpolation to image and texture interpolation.

\section{Background}\label{sec:review}
In this section, we briefly review the classical theory of optimal transport (OT), and the Riemannian structure of the Wasserstein space induced by this OT metric. 

\subsection{Review of Optimal Transport}\label{ssec:reviewOT}
Let $\domain$ be a Polish space (separable, completely metrizable) that additionally satisfies the Heine-Borel property, i.e. its compact sets are exactly the closed and bounded ones. Moreover, we introduce the set of probability measures $\Prob(\domain)$ on $\domain$. The subset of probability measures $\mu$ with finite second moment, i.e. $\int_\domain d^2(x_0,x)\d\mu<\infty,$ for some (and any) $x_0\in\domain$ and a fixed metric $d(\cdot,\cdot)$ that completely metrizes $\domain$ will be denoted as $\ProbW(\domain)$. For two probability measures $\mu,\nu\in\Prob(\domain)$, we shall denote with $U(\mu,\nu)$ the set of couplings between them, that is, the set of (probability) measures $\Pi\in\Prob(\domain^2)$ with $\Pi(A\times\domain)=\mu(A)$ and $\Pi(\domain\times A)=\nu(A)$ for all Borel sets $A$ in $\domain$. 
For $\mu,\nu\in\Prob(\domain)$, the set $U_o(\mu,\nu)$ is the set of all couplings $\Pi$ between $\mu$ and $\nu$ that minimize $\int_{\domain^2}d^2(x,y)\d\Pi(x,y),$ i.e. the set of optimal couplings for the cost $d^2(\cdot,\cdot)$.
\begin{dfntn}[Wasserstein distance]
The squared ($L^2$-)Wasserstein distance between two probability measures $\mu,\nu\in\Prob(\domain)$ will be denoted by $\Wass^2$, and is defined as
\begin{align*}
\Wass^2(\mu,\nu)\coloneqq \inf_{\Pi\in U(\mu,\nu)}\int_{\domain^2}d^2(x,y)\d\Pi(x,y).
\end{align*}
\end{dfntn}

Note that an optimal coupling is guaranteed to exist, and hence the infimum is actually a minimum. Furthermore, restricting to the space $\ProbW(\domain)\times\ProbW(\domain)$ actually leads to a complete metric space, cf. \cite{Vi09}. With this in mind, we define a $\ProbW(\domain)$-valued curve $(\mu_t)_{t\in[0,1]}$ as absolutely continuous, if there exists $m\in L^1([0,1])$, so that $\Wass(\mu_t,\mu_s)\leq\int_s^tm(r)\d r$ for all $0\leq s\leq t\leq 1$.

Moreover, let $C_b^0(\domain)$ be the set of continuous, bounded functions on $\domain$. We then say that the sequence of measures $(\mu_k)_k$ converges narrowly to some $\mu\in\Prob(\domain)$, if
\begin{align*}
\int_\domain f\d\mu_k\rightarrow \int_\domain f\d\mu,
\end{align*}
for all $f\in C_b^0(\domain)$. This will be denoted by $\mu_k\rightharpoonup\mu$.

The concept of tightness of probability measures will play a key role in the sequel: A set $\mathcal{K}\subseteq\Prob(\domain)$ is said to be tight, if for any $\varepsilon>0$ there is a compact set $\domain_\varepsilon\subseteq\domain$, such that $\mu(\domain\setminus \domain_\varepsilon)\leq\varepsilon$ for all $\mu\in\mathcal{K}$. Prokhorov's theorem states that tightness of a set of measures is equivalent to relative compactness in the topology induced by the narrow convergence of measures, cf. \cite{Pr56}.

\subsection{Wasserstein spaces as a Riemannian Manifold}
In this section we consider the spline interpolation problem from a geometric perspective. To this end, we will rely on the formal definition of a Riemannian metric on $\ProbW(\R^d)$ given in \cite{Lo06} and chapter 8 of \cite{AmGi05}. We first introduce a characterization of absolutely continuous measure-valued curves $(\mu_t)_{t\in[0,1]}$. Indeed, absolute continuity of a curve $(\mu_t)_t$ is equivalent to the existence of a velocity field $v_t:\R^d\rightarrow\R^d$ for $t\in[0,1]$, satisfying certain estimates, and solving the continuity equation (CE):
\begin{align}\label{eq:cont_eq}
\partial_t\mu_t+\nabla\cdot(v_t\mu_t)=0 \ \ \ \text{in} \ \ (0,1)\times\R^d,
\end{align}
encoding the conservation of mass (see \cite{AmGi05}, Theorem 8.3.1. for a thorough proof). The above equation is to be understood in the sense of distributions. Moreover, due to the Benamou-Brenier formula (cf. \cite{BeBr00}, Proposition 1.1) one recovers the following definition of the Wasserstein distance in terms of the velocity field $(v_t)_t$:
\begin{align*}
\Wass^2(\mu_0,\mu_1)\coloneqq\inf_{(\mu,v) \in CE(\mu_0,\mu_1)} \int_0^1\int_{\R^d} \vert v_t\vert^2\d\mu_t\d t, 
\end{align*} 
where $CE(\overline\mu_0,\overline\mu_1)$ is the set of pairs $(\mu,v)$, such that $\mu=(\mu_t)_t$ is an absolutely continuous curve in $\ProbW(\R^d)$, and $v=(v_t)_t$ is a time-dependent vector field, such that it satisfies \eqref{eq:cont_eq} in the distributional sense, with $\mu_0=\overline\mu_0$ and $\mu_1=\overline\mu_1$. For a fixed curve $(\mu_t)_t$, the optimal velocity field $(v_t)_t$ of the above problem can be characterized as belonging to the set 
\begin{align*}
T_{\mu_t}\coloneqq\overline{\{ \nabla\varphi: \varphi\in C_c^\infty(\R^d) \}}^{L^2(\mu_t, \R^d)}
\end{align*}
for almost every $t\in[0,1]$ (cf. \cite{AmGi05}, Proposition 8.4.5), where $C_c^\infty(\R^d)$ is the set of all real-valued, smooth, compactly supported functions on $\R^d$, and the bar notation denotes the closure of a set with respect to the $L^2(\mu_t,\R^d)$ norm. This fact justifies the suggestive definition of the set $T_{\mu}$ as the tangent space of the Wasserstein space $\ProbW(\R^d)$ at the point $\mu$. The Riemannian metric on $\ProbW(\R^d)$ at $\mu$ is then simply given by the $L^2$ product
\begin{align*}
\langle v,w\rangle_{T_{\mu_t}}\coloneqq \int_{\R^d}\langle v,w\rangle\d\mu_t,
\end{align*}
where $\langle\cdot,\cdot\rangle$ on the right hand side represents the usual inner product on $\R^d$.
Then, the path energy $\pathenergy$ of the measure-valued curve $(\mu_t)_t$ can be expressed by 
\begin{align}\label{eq:path_energy}
\pathenergy((\mu_t)_t)=\inf_{v: (\mu,v) \in CE (\mu_0,\mu_1)} \int_0^1\int_{\R^d}\vert v_t\vert^2\d\mu_t\d t.
\end{align}
In their landmark paper \cite{BeBr00}, Benamou and Brenier showed that the functional being minimised in the last line is convex in the variables $\mu$ and $w=v\mu$. In \cite{dB63}, classical splines are defined as minimizers of the squared acceleration, integrated over time. The Riemannian counterpart to the acceleration of a particle is the covariant derivative of its velocity field $(v_t)_t$. To define this 
let us call a curve $(\mu_t)_t$ in $\ProbW(\R^d)$ \emph{regular}, if it is absolutely continuous and the optimal velocity vector field $(v_t)_t$ satisfying the continuity equation is Lipschitz in space and satisfies
\[\int_0^1 {\sf Lip}(v_t) \d t <\infty\;,\]
where ${\sf Lip}(v)$ denotes the Lipschitz constant of $v$. Then, by \cite{AmGi05}, Proposition 8.1.8, there exists a unique family of flow maps $\flowmap_s^t(\cdot):\R^d\rightarrow\R^d$ that satisfy
\begin{align}\label{eq:flow_map_eq}
\frac{d}{\d t}\flowmap^t_s(x)=v_t(\flowmap^t_s(x)), \ \ \ \flowmap_s^s(x)=x.
\end{align} 
We have that $\mu_t = (\flowmap_s^t)_\# \mu_s$ for all $s\leq t$.
The total derivative of an absolutely continuous vector field $(w_t)_t$ along a regular curve $(\mu_t)_t$ on $\ProbW(\R^d)$ is then defined for almost all $t\in(0,1)$ as
\begin{align*}
\frac{D}{\d t}w_t\coloneqq\lim_{h\rightarrow 0}\frac{w_{t+h}\circ \flowmap^{t+h}_t-w_t}{h},
\end{align*}
in the sense of $L^2(\mu_t)$. For a smooth vector field $(w_t)_t$ along a regular measure curve $(\mu_t)_t$, we can use \eqref{eq:flow_map_eq} to obtain explicitly 
\begin{align*}
\frac{D}{\d t}w_t=\partial_tw_t+\nabla w_t\cdot v_t.
\end{align*}
Finally, the covariant derivative can be given by projecting onto the tangent space
\begin{align}
\nabla_{v_t}w_t\coloneqq P_{\mu_t}(\partial_tw_t+\nabla w_t\cdot v_t),
\end{align}
where $P_{\mu}$ is the orthogonal projection in $L^2(\mu)$ onto the tangent space $T_{\mu}$. For a thorough derivation of the covariant derivative on $\ProbW(\domain)$, we refer to \cite{AmGi13}, chapter 6.
\section{Splines in Wasserstein Spaces}\label{sec:splines_wass_spaces}
\subsection{Definition of splines}
Based on the discussion in the previous section, for $v_t=\nabla\varphi_t\in T_{\mu_t}$ one may use $\nabla_{v_t}v_t$ as the acceleration of a regular measure-valued curve $\mu_t$. This leads to
\begin{align}\label{eq:cov_deriv_velocity}
\nabla_{v_t}v_t = P_{\mu_t}(\partial_tv_t+\nabla v_t\cdot v_t)=P_{\mu_t}(\dot{v}_t+\tfrac12\nabla \vert v_t\vert^2)=\dot{v}_t+\tfrac12\nabla \vert v_t\vert^2,
\end{align}
where the last equality holds due to $\dot v_t=\nabla\dot\varphi_t\in T_{\mu_t}$ and the second term already being in gradient-field form. This naturally leads to the following notion of a continuous-time spline energy functional. For a general curve $(\mu_t)_t:[0,1]\rightarrow\ProbW(\R^d)$ we set
\begin{align}\label{def:spline_energy_cont}
\mathcal{F}((\mu_t)_t) = \inf_{v}\int_0^1\int_{\mathbb{R}^d}\left\vert\dot{v}_t+\tfrac12 \nabla\vert v_t\vert^2\right\vert^2\d\mu_t\d t,
\end{align}
where the infimum is taken over sufficiently regular time-dependent vector fields $v=(v_t)_t$, such that $(\mu,v)\in CE(\mu_0,\mu_1)$, and $v_t\in T_{\mu_t}$ for all $t\in(0,1)$.
The spline interpolation problem in the Wasserstein space is then to find a curve $(\mu_t)_t:[0,1]\rightarrow\ProbW(\R^d)$ that minimizes the functional \eqref{def:spline_energy_cont}, subject to a set of $I>2$ point-wise interpolation constraints
\begin{align}\label{eq:interp_constraints}
\mu_{\overline{t}_i}=\overline{\mu}_i, \ \ \ i=1,\ldots, I,
\end{align}
for prescribed times $\overline{t}_i\in[0,1]$, $i=1,\ldots, I$, with $\overline{t}_1<\ldots<\overline{t}_I$ and $\overline{\mu}_i\in\ProbW(\R^d)$. As already discussed in \cite{dB63} for the Euclidean case, and \cite{HeRuWi18} for the Riemannian case, we may impose one of the following boundary conditions (b.c.):
\begin{align}
\text{natural b.c.: \ \ \ \ }&\text{no additional condition},\label{eq:naturalbc}\\
\text{Hermite b.c.: \ \ \ }&v_0=\overline{v}_0, v_1=\overline{v}_1 \text{ for given } \overline{v}_0\in T_{\mu_0} \text{ and } \overline{v}_1\in T_{\mu_1},\label{eq:hermitebc}\\
\text{periodic b.c.: \ \ \ }&\mu_0=\mu_1, v_0=v_1.\label{eq:periodicbc}
\end{align}
In the case of Hermite (also known as clamped) boundary conditions, we assume that $\overline{t}_1=0$ and $\overline{t}_I=1$, so that $\mu_0$ and $\mu_1$ are prescribed as well.

From a theoretical point of view, it will be advantageous to regularize the above spline energy by adding the path energy $\pathenergy$ multiplied by a regularization parameter $\delta>0$. Hence, we introduce the regularized spline energy functional 
\begin{align}\label{reg_spline_energy}
\splineenergy^\delta:=\splineenergy+\delta\pathenergy.
\end{align} 
This will ensure tightness of all probability measures with finite energy, and consequently existence in the time-discrete case.
\begin{dfntn}
For given times $\overline{t}_i\in[0,1]$ and prescribed probability distributions $\overline{\mu}_i\in\ProbW(\R^d)$, $i=1,\ldots,I$, we define a (regularized) spline interpolation $(\mu_t)_t$ as a minimizer of the spline energy functional \eqref{def:spline_energy_cont} (resp. \eqref{reg_spline_energy}) subject to \eqref{eq:interp_constraints} and at most one of the boundary conditions \eqref{eq:naturalbc}-\eqref{eq:periodicbc}. 
\end{dfntn}
\begin{xmpl} [Euclidean space]
The Wasserstein distance between two delta distributions located at $x$ and $y$ is equal to the Euclidean distance $\vert x-y\vert$, and the associated Wasserstein geodesic is given by the curve of delta distributions at the locations of the Euclidean geodesic interpolating the end points. We now briefly check whether our definition is also consistent with cubic splines in $\mathbb{R}^d$  when considering delta distributions.

Let $x:[0,1]\rightarrow\mathbb{R}^d$ be a twice-differentiable curve, and define the measure-valued curve $\mu_t=\delta_{x_t}$. Then, one checks that with the choice $v_t\equiv\dot{x}_t$, (CE) is satisfied in distributional sense:
\begin{align*}
\int_0^1\int_{\R^d}\left(\mu_t\partial_t\phi(t,x)+\mu_tv_t\nabla\phi(t,x)\right)\d t\d x &= \int_0^1\left(\partial_t\phi(t,x(t))+\dot x_t\cdot\nabla\phi(t,x(t))\right)\d t =0,
\end{align*}
for all $\phi\in C_c^\infty((0,1)\times\R^d)$. Moreover, as $v_t$ is constant in space, we have $Dv_t\equiv 0$. Due to \eqref{eq:cov_deriv_velocity}, we obtain $\nabla_{v_t}v_t=\ddot{x}_t$, so using \eqref{def:spline_energy_cont} one gets
\begin{align*}
\mathcal{F}((\mu_t)_t)=\int_0^1\int_{\R^d}\vert\ddot{x}_t\vert^2\d \mu_t\d t=\int_0^1\vert\ddot{x}_t\vert^2\d t,
\end{align*} 
for which the minimizer is given by the cubic spline subject to the interpolation constraints \cite{dB63}.
\end{xmpl}

The Wasserstein space $\ProbW(\R^d)$ is isometrically isomorphic to $\R^d\times \ProbW^0(\R^d)$, where the factor $\R^d$ represents the center of mass
and $\ProbW^0(\R^d)$ is the space of probability distributions centered around $0$. In this spirit the 
dynamic of spline paths can be split into the time evolution of the center the mass and the time evolution of the distribution around it, as we shall now demonstrate.

Let $(\mu,v)\in CE$ be a solution to the continuity equation, with $v=(v_t)_t$ being optimal. Hence, for all $t\in[0,1],$ $v_t$ is a gradient field, and in particular $Dv_t^T=Dv_t.$
Let $m_t\coloneqq\int x\d\mu_t(x)$ be the center of mass and let $\tilde\mu_t(\cdot )\coloneqq\mu_t(\cdot+m_t)$ be the re-centered distribution. Furthermore,
we define the re-centered velocity field 
$\tilde v_t(x)\coloneqq v_t(x+m_t)-\dot m_t$.
Then one easily checks that 
$(\tilde\mu,\tilde v)\in CE$.
Now, we first show a decoupling of the (first-order) action functional, i.e.
\begin{align*}
\int_0^1\int_{\R^d}\vert\tilde v_t\vert^2\d\tilde\mu_t\d t=\int_0^1\int_{\R^d}\vert v_t(\cdot+m_t)-\dot m_t\vert^2\d(\Ide-m_t)_\#\mu_t\d t=\int_0^1\int_{\R^d}\vert v_t-\dot m_t\vert^2\d\mu_t\d t\\
=\int_0^1\vert v_t\vert^2\d\mu_t\d t+\int_0^1\vert\dot m_t\vert^2\d t-2\int_0^1\int_{\R^d}\langle v_t,\dot m_t\rangle\d\mu_t\d t=\int_0^1\vert v_t\vert^2\d\mu_t\d t-\int_0^1\vert\dot m_t\vert^2\d t,
\end{align*}
where we used that by the continuity equation
\begin{equation}\label{eq:ce-mean}
\dot m_t = \frac{\d}{\d t}\int_{\R^d} x\d\mu_t(x) = \int_{\R^d} \nabla x \cdot v_t \d\mu_t(x) = \int_{\R^d} v_t\d\mu_t(x)\;.
\end{equation}
Next, we consider the decoupling of the (second-order) spline energy $\int_0^1\int_{\R^d}\vert \dot{v}_t+\frac{1}{2}\nabla\vert v_t\vert^2\vert^2\d\mu_t\d t$.
Taking into account  
\begin{align*}
&\dot{\tilde v}_t(x)=\partial_t(v_t(x+m_t)-\dot m_t)=\dot v_t(x+m_t)+(D v_t){(x+m_t)}\cdot\dot m_t-\ddot m_t,\\
&\nabla\vert\tilde v_t\vert^2=\nabla\vert\dot m_t\vert^2+\nabla\vert v_t(x+m_t)\vert^2-2\nabla\langle v_t(x+m_t),\dot m_t\rangle=2(D v_t){(x+m_t)}\cdot v_t(x+m_t)-2(D v_t){(x+m_t)}\dot m_t
\end{align*}
we obtain 
\begin{align*}
&\int_0^1\int_{\R^d}\left\vert \dot{\tilde v}_t+\frac{1}{2}\nabla\vert\tilde v_t\vert^2\right\vert^2\d\tilde\mu_t\d t\\
&=\int_0^1\int_{\R^d}\left\vert \dot v_t(x+m_t)+(D v_t){(x+m_t)}\cdot\dot m_t-\ddot m_t-(D v_t){(x+m_t)}\cdot\dot m_t+(D v_t){(x+m_t)}\cdot v_t(x+m_t)\right\vert^2\d\tilde\mu_t(x)\d t\\
&=\int_0^1\int_{\R^d}\left\vert \dot v_t(x+m_t)-\ddot m_t+(D v_t){(x+m_t)}\cdot v_t(x+m_t)\right\vert^2\d\tilde\mu_t(x)\d t\\
&=\int_0^1\int_{\R^d}\left\vert \dot v_t-\ddot m_t+D v_t(v_t)\right\vert^2\d\mu_t\d t\\
&=\int_0^1\int_{\R^d}\vert \dot v_t+D v_t(v_t)\vert^2\d\mu_t\d t+\int_0^1\vert\ddot m_t\vert^2\d t-2\int_0^1\int_{\R^d}\langle\ddot m_t,\dot v_t+D v_t(v_t)\rangle\d\mu_t\d t.
\end{align*}
Now, differentiating \eqref{eq:ce-mean} in time we achieve
\begin{align*}
\ddot m_t &= \frac{\d}{\d t}\int_{\R^d} v_t\d\mu_t
=
\int_{\R^d} \dot v_t +D v_t(v_t) \d\mu_t\;.
\end{align*}
Finally, plugging this back into the previous computation we get
\begin{align}\label{eq:spline-energy-decoupling}
\int_0^1\int_{\R^d}\vert \dot{\tilde v}_t+\frac{1}{2}\nabla\vert\tilde v_t\vert^2\vert^2 \d\tilde\mu_t\d t 
&=\int_0^1\int_{\R^d}\vert \dot{v}_t+\frac{1}{2}\nabla\vert v_t\vert^2\vert^2\d\mu_t\d t  
-\int_0^1\vert\ddot m_t\vert^2\d t 
\end{align}
This decoupling is advantageous for the numerical implementation. In fact, it leads to a reduced computing time (cf. Figure \ref{fig:decoupling}).

\subsection{Variational discretization of splines} \label{timediscrete}
The temporal discretization of (regularized) Wasserstein spline energies will be based on a variational problem.
To motivate the proposed discrete spline energy functional, let us consider the situation in Euclidean spaces, in which the velocity field $v$ of a smooth curve $x:[0,1]\rightarrow\R^d$ coincides with $\dot{x}$. By sampling this curve uniformly, i.e. taking $x_k\coloneqq x(t^K_k)$ for $t^K_k\coloneqq k/K$, $k=0,\ldots,K$, we are able to approximate the velocity at a time $t^K_k$ by finite differences, that is, $\dot{x}(t_k^K)\approx K(x_{k+1}-x_k)$. Therefore, we obtain
\begin{align*}
\vert\dot{x}(t_k^K)\vert^2\approx K^2\vert x_{k+1}-x_k\vert^2.
\end{align*}
Similarly, in Euclidean spaces the covariant derivative of the velocity field coincides with the acceleration $\ddot{x}$. We approximate this by central second order difference quotients, i.e. $\ddot{x}(t^K_k)\approx K^2(x_{k+1}-2x_k+x_{k-1}).$ Thus, defining $\Bary(x_{k+1},x_{k-1})\coloneqq\frac{x_{k+1}+x_{k-1}}{2}$ one obtains 
\begin{align*}
\vert\ddot{x}(t^K_k)\vert^2\approx4K^4\left\vert x_k-\frac{x_{k+1}+x_{k-1}}{2}\right\vert^2=4K^4\left\vert x_k-\Bary(x_{k+1},x_{k-1})\right\vert^2.
\end{align*}
A simple rectangular quadrature rule $\int_0^1f(t)\d t\approx K^{-1}\sum_{k=1}^{K-1}f(t^K_k)$ for $t^K_k\coloneqq\frac{k}{K}$ leads to the following approximations of the Euclidean velocity and acceleration functional, respectively:
\begin{align}\label{eq:eucl_path_nrg]}
\pathenergy(x)&=\int_0^1\vert\dot{x}_t\vert^2\d t\approx K\sum_{k=1}^{K}\left\vert x_{k+1}-x_k\right\vert^2,\\
\label{eq:eucl_spline_nrg]}
\splineenergy(x)&=\int_0^1\vert\ddot{x}_t\vert^2\d t\approx 4K^3\sum_{k=1}^{K-1}\left\vert x_k-\Bary(x_{k+1},x_{k-1})\right\vert^2.
\end{align}
Recall that the Euclidean barycenter is the solution to the following minimization problem:
\begin{align*}
\Bary(x, y)=\argmin_{z\in\R^d}\left(\vert x-z\vert^2+\vert y-z\vert^2\right),
\end{align*}
for some $x,y\in\R^d$. Hence, it is intuitive to replace the Euclidean $L^2$-norm with the Wasserstein distance, giving rise to the following discrete path energy 
\begin{align}\label{discrete_path_energy_def}
\PathEnergy^K(\bm\mu^K)\coloneqq K\sum_{k=0}^{K-1} \Wass^2(\mu^K_k,\mu^K_{k+1}),
\end{align}
for a $(K+1)$-tuple of probability measures $\bm{\mu}^K:=(\mu_0^K,\ldots,\mu_K^K)\in\ProbW(\domain)^{K+1}$.
Moreover, we will also give suitable definitions of a Wasserstein barycenter:
\begin{dfntn}
Let $\mu,\nu\in\Prob(\domain)$, and $t\in[0,1]$. The set of $t$-barycenters $\Bary^t(\mu,\nu)$ between $\mu$ and $\nu$
is the set of solutions of the following minimization problem
\begin{align} \label{bary_var_eq}
\argmin_{\rho\in\Prob(\domain)}\ (1-t)\Wass^2(\rho,\mu)+t\Wass^2(\rho,\nu).
\end{align}
For the sake of readability, we shall usually omit the $t$-index from both the notation and nomenclature when $t=\tfrac12$. 
\end{dfntn}
\begin{rmrk}\label{rmrk:bary}
If $\mu,\nu\in\ProbW(\domain)$, then we can guarantee the existence of a solution of \eqref{bary_var_eq} (cf. \cite{Lo06}). Indeed, let $\Pi\in U_o(\mu,\nu)$, and let $\pi^i$ be the projection operators onto the $i$-th coordinate. Then, 
\begin{align}
 \left((1-t) \pi_1+t\pi_2\right)_\#\Pi \in \Bary^t(\mu,\nu).
\end{align}
Let $\domain\subseteq\R^d$ and define $\ProbAC(\domain)\subset\ProbW(\domain)$ as the set of all absolutely continuous probability measures in $\ProbW(\domain)$ with respect to the Lebesgue measure on $\R^d$. If, in addition, at least one of $\mu$ or $\nu$ belong to the set $\ProbAC(\domain)$, then Brenier's theorem \cite{Br91} and McCann's interpolation \cite{Mc01} even guarantee uniqueness of the $t$-barycenter, given explicitly by
\begin{align}
\Bary^t(\mu,\nu)= \left\{((1-t)\Id+tT_\mu^\nu)_\#\mu\right\},
\end{align}
where $T_\mu^\nu$ is the optimal transport map from $\mu$ to $\nu$.
\end{rmrk}
In Wasserstein spaces, there is another related notion of barycenter, which will be called generalized Wasserstein barycenter:
\begin{dfntn}\label{def_gen_bary}
Let $\mu_1,\mu_2,\mu_3\in\Prob(\domain)$. Let now $\Pi$ be a three-measure coupling between them, i.e. $\Pi\in\Prob(\domain^3)$, and $\Pi(A\times\domain\times\domain)=\mu_1(A)$, $\Pi(\domain\times A\times\domain)=\mu_2(A)$, and $\Pi(\domain\times\domain\times A)=\mu_3(A)$ for all Borel sets $A\subseteq\domain$. If furthermore, we have that $(\pi^1,\pi^2)_\#\Pi\in U_o(\mu_1,\mu_2)$, and $(\pi^2,\pi^3)_\#\Pi\in U_o(\mu_2,\mu_3)$, we say $\Pi\in U_o(\mu_1,\mu_2,\mu_3)$. A measure $\mu$ is in the set of generalized (Wasserstein) $t$-barycenters
$\Bary^t_{\mu_2}(\mu_1,\mu_3)$ between $\mu_1$ and $\mu_3$ with base point $\mu_2$, 
if it is of the form  $\mu=((1-t) \pi^1+t \pi^3)_\#\Pi$ for a $\Pi\in U_o(\mu_1,\mu_2,\mu_3)$. 
When $t=\frac12$, we shall omit $t$ from the notation. 
\end{dfntn}
\begin{rmrk}\label{rmrk:gen_bary}
Similarly as above, if $\domain\subseteq\R^d$ and $\mu_2\in\ProbAC(\domain)$, then Brenier's theorem guarantees uniqueness of the generalized $t$-barycenter, given explicitly by
\begin{align}
\Bary^t_{\mu_2}(\mu_1,\mu_3)=\left\{((1-t)T_2^1+tT_2^3)_\#\mu_2\right\},
\end{align}
where $T_2^i$ is the optimal transport map from $\mu_2$ to $\mu_i$, $i=1,3$. Note that $(1-t)T_2^1+tT_2^3$ is again an optimal map (since it inherits the structure of being the gradient of a convex function from $T_2^i$).
\end{rmrk}
In analogy to equation \eqref{eq:eucl_spline_nrg]} we will define two notions of time discrete spline energies related to the different kinds of barycenters introduced above:
\begin{dfntn} [Discrete spline energy] Let $\bm{\mu}^K:=(\mu_0^K,\ldots,\mu_K^K)\in\Prob(\domain)^{K+1}$ be a $(K+1)$-tuple of probability measures. The discrete spline energy $\SplineEnergy^K$ of $\bm{\mu}^K$ is then defined as
\begin{align}\label{discrete_spline_energy}
\SplineEnergy^K(\bm\mu^K)\coloneqq \inf_{\tilde{\bm{\mu}}^K} 4K^3  
\sum_{k=1}^{K-1}  \Wass^2(\mu_k^K,\tilde \mu^K_k),
\end{align}
where the infimum is taken over all  $\tilde{\bm{\mu}}^K = (\tilde \mu^K_k)_{k=1,\ldots, K-1}$ with $\tilde \mu^K_k \in \Bary(\mu_{k-1}^K,\mu_{k+1}^K)$.
Similarly, one defines the generalized discrete spline energy $\SplineEnergyG^K$ of $\bm{\mu}^K$ as
\begin{align}\label{discrete_gen_spline_energy}
\SplineEnergyG^K(\bm\mu^K)\coloneqq\inf_{\tilde{\bm{\mu}}^K} 4K^3\sum_{k=1}^{K-1}\Wass^2(\mu_k^K,\tilde \mu^K_k),
\end{align}
where the infimum is taken over all  $\tilde{\bm{\mu}}^K = (\tilde \mu^K_k)_{k=1,\ldots, K-1}$ with $\tilde \mu^K_k \in \Bary_{\mu_k}(\mu_{k-1}^K,\mu_{k+1}^K)$.
The regularized discrete spline energies are given by
\begin{align}\label{regularized_discrete_spline_energy}
\SplineEnergy^{\delta,K}\coloneqq\SplineEnergy^K+\delta\PathEnergy^K, \ \ \  \SplineEnergyG^{\delta,K}\coloneqq\SplineEnergyG^K+\delta\PathEnergy^K
\end{align}
for $\delta>0$ (for $\delta=0$ we retrieve the non-regularized spline energy).
Computing a (regularized) time-discrete spline interpolation now consists in finding a tuple $\bm{\mu}^K=(\mu_0^K,\ldots,\mu_K^K)$ that minimizes the functional \eqref{regularized_discrete_spline_energy} in some sense to be defined, subject to a set of $I>2$ point-wise interpolation constraints
\begin{align}\label{eq:interp_constraints_d}
\mu_{K\overline{t}_i}^K=\overline{\mu}_i, \ \ \ i=1,\ldots, I,
\end{align}
for fixed prescribed times $\overline{t}_i\in[0,1]$, which fulfil $K\overline{t}_i\in\N_0$, with $\overline{t}_1<\ldots<\overline{t}_I$ and $\overline{\mu}_i\in\ProbW(\mathbb{R}^d)$ for $i=1,\ldots, I$. 
\end{dfntn}

The discrete counterparts of boundary conditions, one of which may be additionally imposed, can be written as follows:
\begin{align}
&\text{natural b.c.: \ \ \ \ no additional condition},\label{eq:naturalbc_d}\\
&\text{Hermite b.c.: \ \ \ }\mu^K_0=\overline{\mu}_0, \ \mu^K_1=\overline{\mu}_1, \ \mu^K_{K-1}=\overline{\mu}_{K-1}, \ \mu^K_K=\overline{\mu}_K,\label{eq:hermitebc_d}\\
&\text{periodic b.c.: \ \ \ }\mu_1^K=\mu_K^K,  \ \mu_0^K=\mu_{K-1}^K. \label{eq:periodicbc_d}
\end{align}
Now we are in position to define regularized time-discrete spline interpolations:
\begin{dfntn} [Regularized discrete spline interpolations] \label{def_spline_interp} 
For $2\leq I \leq K$, given data points $\overline{t}_i\in[0,1]$ fulfilling $K\overline t_i\in\N_0$, $\delta>0$ and fixed data $\overline{\mu}_i\in\Prob(\domain)$ for $i=1,\ldots,I$, we define the tuple $\bm\mu^K\in\ProbW(\domain)^{K+1}$ to be a regularized (generalized) discrete spline interpolation if it is a minimizer of the discrete spline energy functional $\SplineEnergy_{(G)}^{\delta,K}$ with $\delta > 0$ (cf.~\eqref{regularized_discrete_spline_energy}) 
that satisfy the interpolation constraints \eqref{eq:interp_constraints_d} and one of the boundary conditions \eqref{eq:naturalbc_d}-\eqref{eq:periodicbc_d}. 
\end{dfntn}
We will now show existence of a minimizer of the regularized spline energy functional introduced above, for all $\delta>0$. 
First, let us show a technical lemma:
\begin{lmm}
Let $\Omega$ be as in Subsection \ref{ssec:reviewOT} and let $(\mu_n)_n\subseteq\Prob(\domain)$ be tight, and $(\nu_n)_n\subseteq\Prob(\domain)$. If $\sup_{n}\Wass^2(\mu_n,\nu_n)\leq C<\infty$, then $(\nu_n)_n$ is also tight.
\end{lmm}
\begin{proof}
We will argue by contradiction: Assume that $(\nu_n)_n$ is not tight. Then, there is an $\varepsilon>0$, so that for all $R>0$ there is a $k=k(R)\in\N$ that fulfils $\nu_k(\Omega\setminus\overline{B_R(\omega)})>\varepsilon$ for some fixed $\omega\in\Omega$. 

Let $r>0$ be chosen so that $R>r$, and $\mu_n(\Omega\setminus\overline{B_r(\omega)})\leq\varepsilon/2$ for all $n\in\N$. This is possible due to the tightness of $(\mu_n)_n$. For any coupling $\Pi\in\Prob(\Omega^2)$ of $\mu_{k(R)}$ and $\nu_{k(R)}$ we have that $$\Pi(\lbrace(x,y):d^2(x,y)\geq(R-r)^2\rbrace)>\epsilon/2.$$ Hence, we obtain $$\Wass^2(\mu_k,\nu_k)>\frac{\varepsilon}{2}(R-r)^2.$$
Since $\varepsilon$ and $r$ are fixed, and $k$ only depends on $R$, we can choose $R$ big enough so that $\Wass^2(\mu_k,\nu_k)>C$, which leads to the desired contradiction. 
\end{proof}
\begin{thrm}
For all $\delta>0$, $K\in\N$, $2\leq I\leq K$, given times $\overline{t}_i\in[0,1]$ and prescribed probability measures $\overline{\mu}_i\in\ProbW(\R^d)$ for all $i=1,\ldots, I$, there exists a discrete regularized (generalized) spline interpolation in the sense of Definition \ref{def_spline_interp}.
\end{thrm}
\begin{proof}
Any choice of $\mu_k\in\Prob(\R^d)$ for $k=0,\ldots, K$ gives a finite regularized spline energy $\overline{F}\coloneqq\SplineEnergy^{\delta,K}((\mu_k)_k)$. Let $(\bm{\mu}^{(n)})_n$ be a minimizing sequence for $\SplineEnergy^{\delta,K}$ under the given constraints. In particular, $\sup_n\SplineEnergy^{\delta,K}(\bm{\mu}^{(n)})\leq\overline{F}$. Thus, $$\overline{F}\geq\sup_n\SplineEnergy^{\delta,K}(\bm{\mu}^{(n)})\geq\sup_n\delta\Wass^2(\mu_{k}^{(n)},\mu_{k+1}^{(n)}),$$ for any $k$. For $i=1,\ldots, I$, $\mu_{K\overline t_i}^{(n)}=\overline{\mu}_{K\overline t_i}$ for all $n$. Since any constant measure-valued sequence is tight, by the previous lemma the sequence $\left(\mu_{K\overline t_i+1}^{(n)}\right)_n$ is also tight. We can use the previous lemma multiple times and "propagate" tightness by induction. Next, by Prokhorov's theorem we can choose a subsequence, so that for all $k\in\lbrace 0,\ldots,K\rbrace$ the sequence $(\mu_{k}^{(n)})_n$ is narrowly convergent to some $\mu_k\in\Prob(\R^d)$. In fact, we have by the triangle inequality 
\begin{align*}
\Wass(\mu^{(n)}_k,\delta_{x_0})\leq\Wass(\mu^{(n)}_k,\mu^{(n)}_{k-1})+\Wass(\mu^{(n)}_{k-1},\delta_{x_0}),
\end{align*}
for a point $x_0$ in $\R^d$. If $k-1=K\overline{t}_i$ for some $i=1,\ldots, I$, then $\mu_{k-1}=\overline\mu_{i}\in\ProbW(\R^d)$ and hence, the second term on the right hand side is uniformly bounded in $n$. Since the first term on the right hand side is one term in the discrete path energy contained in $\SplineEnergy^{\delta,K}$, it is uniformly bounded in $n$ as well. We now use $\mu_k^{(n)}\rightharpoonup\mu_k$ and the lower semi-continuity of $\Wass$ under narrow convergence to show that
\begin{align*}
\Wass^2(\mu_k,\delta_{x_0})\leq\liminf_{n\rightarrow\infty}\Wass^2(\mu^{(n)}_k,\delta_{x_0})<\infty.
\end{align*}
Proceeding by induction, we obtain that $\mu_k\in\ProbW(\R^d)$ for all $k=0,\ldots,K$.
Let us now rewrite \eqref{discrete_spline_energy}:
\begin{align*}
\SplineEnergy^{K}(\bm\mu^{(n)})&=4K^3\sum_{k=1}^{K-1}\inf_{\tilde\mu_k^{(n)}\in\Bary(\mu_{k-1}^{(n)},\mu_{k+1}^{(n)})}\Wass^2(\mu_k^{(n)},\tilde\mu_k^{(n)})\\
&=4K^3\sum_{k=1}^{K-1}\inf_{\tilde\Pi^{(n)}_{k-1,k+1}\in U_o(\mu_{k-1}^{(n)},\mu_{k+1}^{(n)})}\Wass^2\left(\mu_k^{(n)},\left(\frac{1}{2} \pi^1+\frac{1}{2} \pi^2\right)_\#\tilde\Pi^{(n)}_{k-1,k+1}\right).
\end{align*}

Next, we denote the value of the inner infimum above $I^{(n)}_k$ and assume that  $\Pi_{k-1,k+1}^{(n)}\in U_{o}((\mu_{k-1}^{(n)},\mu_{k+1}^{(n)}))$ is chosen such that $\Wass^2\left(\mu_k^{(n)},\left(\frac{1}{2} \pi^1+\frac{1}{2} \pi^2\right)_\#\Pi^{(n)}_{k-1,k+1}\right) \leq I^{(n)}_k+1/n$. By the stability of optimal couplings  \cite[Prop.~7.1.3]{AmGi13}, $\Pi^{(n)}_{k-1,k+1}$ converges (up to a subsequence) to an optimal coupling $\Pi_{k-1,k+1}$ of $\mu_{k-1}$ and $\mu_{k+1}$.
This entails narrow convergence of the barycenter
$$\Bary(\mu_{k-1}^{(n)},\mu_{k+1}^{(n)})\ni\left(\frac{1}{2} \pi^1+\frac{1}{2} \pi^2\right)_\#\Pi^{(n)}_{k-1,k+1}\rightharpoonup \left(\frac{1}{2} \pi^1+\frac{1}{2} \pi^2\right)_\#\Pi_{k-1,k+1}=:\tilde\mu_k\in\Bary(\mu_{k-1},\mu_{k+1})$$
due to the continuity of the projections $\pi^i$ (for $i=1,2$), and Remark \ref{rmrk:bary}.

Finally, we use the lower semi-continuity of the Wasserstein distance under narrow convergence and the fact that the spline energy contains a minimization over the choice of barycenters, to obtain
$$\SplineEnergy^{\delta,K}(\bm\mu)\leq 4K^3\sum_{k=1}^{K-1}\Wass^2(\mu_k,\tilde \mu_k)+\delta\PathEnergy^K(\bm\mu)\leq\liminf_{n\rightarrow\infty}\SplineEnergy^{\delta,K}(\bm\mu^{(n)}),$$
where $\bm\mu\coloneqq(\mu_0,\ldots,\mu_K)$.
As the right-hand sequence was assumed to be a minimizing sequence, 
$\bm\mu$ is indeed a spline interpolation according to Definition \ref{def_spline_interp}. The proof of existence of generalized spline interpolations is by analogy. The only remarkable difference is to prove that generalized barycenters narrowly converge to a generalized barycenter, up to a subsequence. To see this, let $\Pi_{k-1,k,k+1}^{(n)}\in\Prob(\R^{3d})$ be a three-measure optimal transport plan between $\mu_{k-1}^{(n)}, \mu_k^{(n)}$ and $\mu_{k+1}^{(n)}$, i.e. $\Pi_{k-1,k,k+1}^{(n)}\in U_o(\mu_{k-1}^{(n)},\mu_k^{(n)},\mu_{k+1}^{(n)})$ (cf. Definition \ref{def_gen_bary}). Once again, as the marginals of $\Pi^{(n)}_{k-1,k,k+1}$ are tight, the sequence of optimal couplings $(\Pi^{(n)}_{k-1,k,k+1})_n$ is also tight, and due to the lower semicontinuity of $\Wass$, it narrowly converges (up to a subsequence) to an optimal coupling $\Pi_{k-1,k,k+1}$ of $\mu_{k-1},\mu_k$ and $\mu_{k+1}$. From this, narrow convergence (up to a subsequence) of the sequence of generalized barycenters $\left(\frac{1}{2} \pi^1+\frac{1}{2} \pi^3\right)_\#\Pi^{(n)}_{k-1,k,k+1}$ to the generalized barycenter $\left(\frac{1}{2} \pi^1+\frac{1}{2} \pi^3\right)_\#\Pi_{k-1,k,k+1}$ follows, again due to the continuity of the projections $\pi^i$ (for $i=1,3$), and Definition \ref{def_gen_bary}.
\end{proof}

\section{Gaussian E-splines}\label{sec:gaussian}
In this section we will explicitly derive the continuous spline energy for measure-valued curves restricted to the space of Gaussian distributions, i.e. minimizers of the spline energy among Gaussian curves, and show its consistency with the discrete spline energy notions we defined in the previous section. Let us first introduce some notation:
\begin{dfntn}\label{def:gauss_space}
Let $\BW$ be the space of symmetric, positive definite $d\times d$ matrices, and $\BWd\subset\BW$ the space of diagonal, positive definite $d\times d$ matrices. Then, one can identify the space of Gaussian probability measures with the set $\R^d\times\BW$ through the bijective map 
\begin{align*}
\Phi:\R^d\times\BW &\longrightarrow\Phi(\R^d\times\BW)\subset\ProbW(\R^d)\\
(m,\sigma) &\mapsto\mathcal{N}(m,\sigma^2),
\end{align*}
where $\mathcal{N}(m,\sigma^2)$ is the Gaussian probability measure with mean $m$ and standard deviation matrix $\sigma$, i.e. the absolutely continuous probability measure with respect to the Lebesgue measure $\mathcal{L}$ on $\R^d$ with density $\frac{\d \mathcal{N}(m,\sigma^2)}{\d\mathcal{L}}$ given by $(2\pi)^{-\frac{d}{2}}\det(\sigma)^{-1}e^{-\tfrac12 (x-m)^T\sigma^{-2}(x-m)}$.
Defining $\avg:\ProbW(\R^d)\rightarrow\R^d,$ $\mu\mapsto\int_{\R^d}x\d\mu(x),$ and $\cov:\ProbW(\R^d)\rightarrow\BW,$ $\mu\mapsto\int_{\R^d}(x-\avg(\mu))(x-\avg(\mu))^T\d\mu(x)$ as the mean and covariance matrix of a probability measure $\mu$, respectively, one can straightforwardly check that the inverse $\Phi^{-1}:\Phi(\R^d\times\BW)\longrightarrow\R^d\times\BW$ is explicitly given by $\mu\mapsto\left(\avg(\mu),\std(\mu)\right)$, where the standard deviation matrix $\std(\mu)$ is the unique element in $\BW$ with $\std^2(\mu)=\cov(\mu)$. 
\end{dfntn}
\subsection{The case of general Gaussian distributions}
In what follows, we explicitly compute the spline energy for curves in the space of Gaussian distributions. To this end, we will first list some facts about optimal transport in this restricted setting. Since the space of Gaussian distributions is contained in $\ProbAC(\R^d)$, we shall from now on abuse notation and denote with $\Bary(\mu,\nu)$ the unique element in the set of barycenters, rather than the set itself.

\begin{prpstn}\label{prop:gauss_facts}
Let $m_1,m_2\in\R^d$, and $\sigma_1,\sigma_2\in\BW$. Define $\mu_1=\mathcal{N}(m_1,\sigma_1^2)$ and $\mu_2=\mathcal{N}(m_2,\sigma_2^2)$. Then, the following statements hold:
\begin{enumerate}
\item The optimal transport map $T$ from $\mu_1$ to $\mu_2$ is given by $x\mapsto T(x)=m_2+\sigma_1^{-1}(\sigma_1\sigma_2^2\sigma_1)^{\tfrac12}\sigma_1^{-1}(x-m_1)$. If $\sigma_1$ and $\sigma_2$ are simultaneously diagonalizable, $T$ is simplified to $T(x)=m_2+\sigma_1^{-1}\sigma_2(x-m_1)$.
\item The squared $L^2-$Wasserstein distance between $\mu_1$ and $\mu_2$ is $\Wass^2(\mu_1,\mu_2)=\vert m_1-m_2\vert^2+B^2(\sigma_1,\sigma_2),$ where $B^2(\sigma_1,\sigma_2)\coloneqq\tr(\sigma_1^2+\sigma_2^2-2(\sigma_1\sigma_2^2\sigma_1)^{1/2})$ is the squared Bures-Wasserstein metric defined in \cite{Bu69}.  If $\sigma_1$ and $\sigma_2$ are simultaneously diagonalizable, $B^2$ is given by $B^2(\sigma_1,\sigma_2)=\Vert\sigma_1-\sigma_2\Vert_F^2$, where $\Vert A\Vert^2_F\coloneqq\tr(A^TA)$ is the Frobenius norm of a matrix $A\in\R^{d\times d}$.
\item For all $t\in[0,1]$, $\Bary^t(\mu_1,\mu_2)$ is a Gaussian distribution with 
\begin{align*}
\avg\left(\Bary^t(\mu_1,\mu_2)\right)&=(1-t)m_1+tm_2,\\
\std\left(\Bary^t(\mu_1,\mu_2)\right)&=\left[\left((1-t)\sigma_1+t\sigma_1^{-1}(\sigma_1\sigma_2^2\sigma_1)^{\tfrac12}\right)\left((1-t)\sigma_1+t\sigma_1^{-1}(\sigma_1\sigma_2^2\sigma_1)^{\tfrac12}\right)^T\right]^{\tfrac12}
\end{align*}
\item Let $m\in\R^d$, $\sigma\in\BW$, and $\mu\coloneqq\mathcal{N}(m,\sigma^2)$. For all $t\in[0,1]$, $\Bary^t_\mu(\mu_1,\mu_2)$ is a Gaussian distribution with 
\begin{align*}
\avg\left(\Bary^t_\mu(\mu_1,\mu_2)\right)&=(1-t)m_1+tm_2,\\
\std\left(\Bary^t_\mu(\mu_1,\mu_2)\right)&=\left[\left((1-t)\sigma^{-1}(\sigma\sigma_1^2\sigma)^{\tfrac12}+t\sigma^{-1}(\sigma\sigma_2^2\sigma)^{\tfrac12}\right)\left((1-t)\sigma^{-1}(\sigma\sigma_1^2\sigma)^{\tfrac12}+t\sigma^{-1}(\sigma\sigma_2^2\sigma)^{\tfrac12}\right)^T\right]^{\tfrac12}.
\end{align*}
\end{enumerate}
\end{prpstn}

\begin{proof}
\begin{enumerate}
\item: See \cite{PeCu19}, equation (2.40). 
\item: See \cite{PeCu19}, equations (2.41)-(2.42).
\item and (4)~: It is straightforward to prove that for  $a, b\in\R^d$, $A\in\R^{d\times d}$ and $\Sigma\in\BW$ then for the map $F:x\mapsto Ax+b$, it holds $F_\#\mathcal{N}(a,\Sigma)=\mathcal{N}(Aa+b,A\Sigma A^T)$. Plugging in the explicit expression for $T$ given in (1), and using Remarks \ref{rmrk:bary} and \ref{rmrk:gen_bary} respectively, one obtains the desired results. 
\end{enumerate}
\end{proof}
\begin{rmrk}
The above proposition implies that Wasserstein geodesics $(\mu_t)_{t\in[0,1]}$ between two Gaussian distributions are also Gaussian distributed for all $t\in[0,1]$. However, at this point we are not able to prove an analogous statement for Wasserstein splines. 
\end{rmrk}

\begin{prpstn} [Consistency]\label{prop:consistency}
Let $(m_t, \sigma_t)_t$ be a curve in $C^3([0,1],\R^d\times\BW)$, and let $(\mu_t)_t\coloneqq\mathcal{N}(m_t,\sigma_t^2)$ be the respective Gaussian-valued curve. Moreover, for $k=0,\ldots, K$, define $\mu_k^K\coloneqq\mu_{t_k^K}$, with $t_k^K\coloneqq k/K$. Then, we have 
\begin{align}\label{eq:consistency_path}
\pathenergy((\mu_t)_t)&=\int_0^1\left\Vert\sigma_t^{-1}\left.\frac{\d}{\d h}\right\vert_{h=0}(\sigma_t\sigma_{t+h}^2\sigma_t)^{\tfrac12}\right\Vert_F^2\d t+\int_0^1\vert\dot{m}_t\vert^2\d t=\PathEnergy^K((\mu_0^K,\ldots,\mu_K^K))+\bigO(K^{-1}),\\
\label{eq:consistency_spline}
\splineenergy((\mu_t)_t)&=\int_0^1\left\Vert\sigma_t^{-1}\left.\frac{\d^2}{\d h^2}\right\vert_{h=0}(\sigma_t\sigma_{t+h}^2\sigma_t)^{\tfrac12}\right\Vert_F^2\d t+\int_0^1\vert\ddot{m}_t\vert^2\d t=\SplineEnergyG^K((\mu_0^K,\ldots,\mu_K^K))+\bigO(K^{-1}),
\end{align}
where the implicit constant in $\bigO(K^{-1})$ is independent of $K$. 
\end{prpstn}

\begin{proof}
Recall from Proposition \ref{prop:gauss_facts} that the optimal transport map $T_t^s$ from $\mu_t$ to $\mu_{s}$ is given by 
$T_t^{s}(x)=A_t^{s}(x-m_t)+m_{s}$  where $A_t^{s}\coloneqq\sigma_t^{-1}\left(\sigma_t\sigma_{s}^2\sigma_t\right)^{\tfrac12} \sigma_t^{-1}$.
Further we note that the optimal velocity field $v_t$ in the continuity equation solved by $\mu_t$ is given for almost all $t\in(0,1)$ by (see \cite[eq.~(8.4.8)]{AmGi05})
\[
v_t(x)= \partial_s\big|_{s=t} T_t^s(x) = \partial_s\big|_{s=t}A_t^s(x-m_t) + \dot m_t\;.
\]
By the assumptions on the curve $(\mu_t,\sigma_t)_t$, the matrix-valued function $A_t^s$ is continuously-differentiable in $t$ and $s$ with all derivatives up to order 3 uniformly bounded for $s,t\in[0,1]$. Hence also $T_t^s(x)$ is continuously differentiable in $t$ and $s$ and all derivatives up to order three can be bounded by $C\|x\|$ for a constant $C$ independent of $s,t\in [0,1]$ and $x\in \R^d$. In particular, $v_t(x)$ is continuously differentiable in $t$ with derivatives up to order two bounded by $C\|x\|$.

Since $\frac{T_t^{t+h}+T_t^{t-h}}{2}$ is the optimal map between $\Bary_{\mu_t}(\mu_{t-h},\mu_{t+h})$ and $\mu_t$, cf. Remark \ref{rmrk:gen_bary}, we have that 
\begin{equation}\label{eq:one-term-spline}
4\frac{\Wass^2(\mu_t,\Bary_{\mu_t}(\mu_{t-h},\mu_{t+h}))}{h^4} = 
\int_{\R^d}\frac{\vert T^{t+h}_t-2\cdot\Ide+T^{t-h}_t\vert^2}{h^4}\d \mu_t\;.
\end{equation}
By Taylor expansion, we have
\begin{align*}
T_t^{t+h}&= \Ide - h \partial_s\big|_{s=t+h}T_s^{t+h}+\frac12 h^2 \partial_s^2\big|_{s=t+h}T_s^{t+h} -\frac16 h^3 \partial_s^3\big|_{s=u}T_s^{t+h} \;,\\
T_t^{t-h}&= \Ide - h \partial_s\big|_{s=t}T_t^{s}+\frac12 h^2 \partial_s^2\big|_{s=t}T_t^{s} -\frac16 h^3 \partial_s^3\big|_{s=v}T_t^{s}\;, 
\end{align*}
for some $u\in (t,t+h)$ and $v\in (t-h,t)$.

Taking into account the fact that $T_t^s\circ T_s^t=\Ide$ and taking first and second derivatives in $s$ of this identity at $s=t$ we readily deduce that 
\[\partial_s\big|_{s=t}T_s^t= - \partial_s\big|_{s=t} T_t^s=-v_t\;,\]
and 
\begin{align*}
\partial_s^2\big|_{s=t}T_s^t &= - \partial_s^2\big|_{s=t}T_t^s -2(D\partial_s\big|_{s=t}T_t^s\big)\partial_s\big|_{s=t}T_s^t\;.
\end{align*}
The last term evaluates to
\[
-2D\big(\partial_s\big|_{s=t}T_t^s\big)\partial_s\big|_{s=t}T_s^t = 2\big(D v_t\big)v_t = \nabla|v_t|^2\;.
\]
Collecting these observations, we obtain
\begin{align}\nonumber
\frac{T^{t+h}_t-2\Ide+T^{t-h}_t}{h^2}
&= 
\frac{1}{h}\Big[-\partial_s\big|_{s=t+h}T_s^{t+h}-\partial_s\big|_{s=t}T_t^{s}\Big]
+\frac{1}{2}\Big[\partial_s^2\big|_{s=t+h}T_s^{t+h} + \partial_s^2\big|_{s=t}T_t^{s}\Big] 
+O(h)\\\nonumber
&=
\frac{1}{h}(v_{t+h}-v_t) 
+\frac{1}{2}\Big[-\partial_s^2\big|_{s=t+h}T_{t+h}^s + \partial_s^2\big|_{s=t}T_t^{s} + \nabla|v_{t+h}|^2\Big] 
+O(h)\\\label{eq:2nd-diff-T}
&=
 \dot v_t + \frac{1}{2} \nabla |v_{t}|^2 +O(h)\;,
\end{align}
where the terms $O(h)$ are bounded by $Ch\|x\|$ for a constant $C$ independent of $t, h$ and $x$.
Hence,
\[
4\frac{\Wass^2(\mu_t,\Bary_{\mu_t}(\mu_{t-h},\mu_{t+h}))}{h^4} = \int_{\R^d}\big|\dot v_t+\frac{1}{2}\nabla|v_t|^2\big|^2\d\mu_t + \bigO(h)\;,
\]
with $\bigO(h)$ bounded by $Ch$ for a uniform constant $C$. Finally, recall the rectangular quadrature rule $\int_0^1f(t)\d t= K^{-1}\sum_{k=1}^{K-1}f(t^K_k)+\bigO(K^{-1})$ with $t_k^K\coloneqq k/K$ for a Lipschitz function $f$, where the implicit constant in the $\bigO(K^{-1})$ term depends only on ${\sf Lip}(f)$. Setting $h=K^{-1}$ and defining 
$\mu_k^K\coloneqq\mu_{t_k^K}$, we obtain for a uniform constant $C$:
\[
\left\vert\int_0^1\int_{\R^d}\vert \dot v_t+\frac{1}{2}\nabla\vert v_t\vert^2 \vert^2 \d\mu_t\d t-\SplineEnergy^K_G((\mu_0^K,\ldots,\mu_K^K))\right\vert
\leq C K^{-1}\;.
\]
The continuous spline energy on Gaussian distributions is given by
\begin{align*}
\int_0^1\int_{\R^d}\vert \dot v_t+\frac{1}{2}\nabla|v_t|^2\vert^2\d\mu_t\d t&=\int_0^1\int_{\R^d}\left\vert\left.\frac{\d^2}{\d h^2}\right\vert_{h=0}\left(T_t^{t+h}\right)\right\vert^2\d \mu_t\d t\\
&=\int_0^1\tr\left(\sigma_t^{-1}\left(\left.\frac{\d^2}{\d h^2}\right\vert_{h=0}(\sigma_t\sigma_{t+h}^2\sigma_t)^{\tfrac12}\right)^2\sigma_t^{-1}\right)\d t
+\int_0^1\vert\ddot{m}_t\vert^2\d t\\
&=\int_0^1\left\Vert\sigma_t^{-1}\left.\frac{\d^2}{\d h^2}\right\vert_{h=0}(\sigma_t\sigma_{t+h}^2\sigma_t)^{\tfrac12}\right\Vert_F^2\d t+\int_0^1\vert\ddot{m}_t\vert^2\d t\;,
\end{align*}
where we have used in the first step the expansion \eqref{eq:2nd-diff-T} and in the second step the decoupling of the energies from \eqref{eq:spline-energy-decoupling}.
The first equality in \eqref{eq:consistency_path} can be proven once again by using \eqref{eq:spline-energy-decoupling} and the explicit form of $v_t$. Finally, the corresponding estimates for \eqref{eq:consistency_path} can be proven similarly by Taylor approximation of the optimal map to first order. Namely, in place of \eqref{eq:one-term-spline} one uses that 
\[\frac{1}{h^2}\Wass^2(\mu_t,\mu_{t+h})=\int_{\R^d}\frac{|T_t^{t+h}-\Ide|^2}{h^2}\d\mu_t\;,\]
as well as $T_t^{t+h}=\Ide+h v_t +\bigO(h^2)$ with $\bigO(h^2)$ being bounded by $C\|x\|$ with a uniform constant $C$, so that
\[
\left\vert\int_0^1\int_{\R^d}\vert v_t \vert^2 \d\mu_t\d t-\PathEnergy^K((\mu_0^K,\ldots,\mu_K^K))\right\vert
\leq C K^{-1}\;.
\]
\end{proof}
\begin{rmrk}
We expect the previous consistency result for the generalised discrete spline energy to hold for general curves $(\mu_t)_t$ with suffuciently regular densities and velocity fields. In fact, we note that the argument relies on the Gaussian structure essentially only for the explicit error estimates in the Taylor expansion of the optimal maps. In particular, we expect the identity 
$\nabla_{v_t}v_t = \frac{d^2}{dh^2}|_{h=0}T_t^{t+h}$ to hold true for general sufficiently regular curves. However, obtaining a general consistency result for the discrete spline energy with the true barycenter seems much more delicate.
\end{rmrk}
To show the consistency of the proposed non-generalized discrete spline energy functional $\SplineEnergy^K((\mu_0^K,\ldots,\mu_K^K))$, we shall need the following lemmata (which are restricted to the case $d=2$), which relate the barycenter with the generalized barycenter:
\begin{lmm}\label{LemmaStdBar}
Let $m_1,m_2\in\R^2$ and $\sigma_1,\sigma_2\in\BWt$, such that $\Vert\sigma_1-\sigma_2\Vert_F\leq 2h\leq 2$. Then, one obtains
\begin{align*}
\left\Vert\std\left(\Bary(\mu_1,\mu_2)\right)-\frac{\sigma_1+\sigma_2}{2}\right\Vert_F\leq Ch^2,
\end{align*}
where $\mu_1=\mathcal{N}(m_1,\sigma_1^2)$, $\mu_2=\mathcal{N}(m_2,\sigma_2^2)$ and the constant $C$ only depends on 
$\max\{\lambda_{\max}(\sigma_1),\lambda_{\max}(\sigma_2)\}$ 
and $\min\{\lambda_{\min}(\sigma_1),\lambda_{\min}(\sigma_2)\}$, where $\lambda_{\max}(A)$ and $\lambda_{\min}(A)$ denote the largest and smallest eigenvalue of a symmetric matrix $A$, respectively. 
\end{lmm}
\begin{proof}
Recall that in $2$D, the positive definite square root of a positive definite matrix $\sigma=\begin{pmatrix}
a& c\\
c& b
\end{pmatrix}$ is given by the following explicit formula:
\begin{align*}
\sigma^{\tfrac12}=(\tr(\sigma)+2\sqrt{\det(\sigma)})^{-\tfrac12}(\sigma+\sqrt{\det(\sigma)}\Id).
\end{align*}
Indeed, 
\begin{align*}
(\tr(\sigma)+2\sqrt{\det(\sigma)})^{-1}(\sigma+\sqrt{\det(\sigma)}\Id)^2&=(a+b+2\sqrt{\det(\sigma)})^{-1}(\sigma^2+2\sigma\sqrt{\det(\sigma)}+\det(\sigma)\Id)\\
&=(a+b+2\sqrt{\det(\sigma)})^{-1}(\sigma+2\sqrt{\det(\sigma)}\Id+\det(\sigma)\sigma^{-1})\sigma\\
&=(a+b+2\sqrt{\det(\sigma)})^{-1}\left(\begin{pmatrix}a& c\\ c& b\end{pmatrix}+2\sqrt{\det(\sigma)}\Id+\begin{pmatrix}b& -c\\ -c& a\end{pmatrix}\right)\sigma\\
&=(a+b+2\sqrt{\det(\sigma)})^{-1}(a+b+2\sqrt{\det(\sigma)})\sigma=\sigma.
\end{align*}
Let us now define $\sigma\coloneqq\frac{\sigma_1+\sigma_2}{2}$, and $\sigma'\coloneqq\frac{\sigma_2-\sigma_1}{2h}$ and $\sigma_t=\sigma+t\sigma'$, for $t\in[-h,h]$. Moreover, we define $\mu_t\coloneqq\mathcal{N}(0,\sigma_t^2)$ and $F:\R\rightarrow\BWt;  t\mapsto F(t)=\std(\Bary(\mu_{-t},\mu_t))$, which implies $\std\left(\Bary(\mu_1,\mu_2)\right)=F(h)$. Since $F$ is actually smooth, we have by Taylor's theorem
\begin{align*}
F(h)=F(0)+\dot F(0)h+\tfrac12\ddot F(s)h^2,
\end{align*}
for some $s\in[0,h]$. It is straightforward to see that $F(0)=\std(\mu_0)=\frac{\sigma_1+\sigma_2}{2}$. Next, we prove that $\dot F(0)=0$. To this end, we show that if $\sigma:[0,1]\rightarrow\BWt; t\mapsto\sigma_t$ is $C^1$, with $\dot{\sigma}_t=0$ for some $t\in[0,1]$, then we have $\frac{\d}{\d t}\left[\sigma^{\tfrac12}_t\right]=0$. This is easily verified using the matrix square root formula above:
\begin{align*}
\frac{\d}{\d t}\left[\sigma^{\tfrac12}_t\right]=&-\tfrac12 (\tr(\sigma_t)+2\sqrt{\det(\sigma_t)})^{-\tfrac32}(\sigma_t+\sqrt{\det(\sigma_t)}\Id)(\tr(\dot{\sigma}_t)+\sqrt{\det(\sigma_t)}\tr(\dot{\sigma}_t\sigma^{-1}_t))\\
&+(\tr(\sigma_t)+2\sqrt{\det(\sigma_t)})^{-\tfrac12}(\dot{\sigma}_t+\tfrac12\sqrt{\det(\sigma_t)}\tr(\dot{\sigma}_t\sigma^{-1}_t)\Id)=0.
\end{align*}
Hence, it will be sufficient to prove $\left.\frac{\d}{\d t}\right\vert_{t=0}F^2(t)=0$. This is obviously true since $F^2$ is symmetric with respect to $t=0$. 
\begin{align*}
\ddot F(s)=\frac{\d^2}{\d s^2}(F^2)^{\tfrac12}(s)=\frac{\d^2}{\d s^2}\left(\tr(F^2(s))+2\sqrt{\det(F^2(s))}\right)^{-\tfrac12}(F^2(s)+\sqrt{\det(F^2(s))}\Id).
\end{align*}
Recall that 
\begin{align*}
F^2(s)=\tfrac14\left(\sigma_{-s}^2+\sigma_s^2+\sigma_s^{-1}(\sigma_s\sigma_{-s}^2\sigma_s)^{\tfrac12}\sigma_s+\sigma_s(\sigma_s\sigma_{-s}^2\sigma_s)^{\tfrac12}\sigma_s^{-1}\right).
\end{align*}
Due to the formula of the matrix square root for $2\times 2$ matrices, in order to find upper bounds of $\Vert\ddot F(s)\Vert_F$ it is enough to prove upper and (positive) lower bounds of $\Vert\sigma_s\Vert_F, \Vert\sigma_s^{-1}\Vert_F$, $\tr(\sigma_s)$, $\tr(\sigma_s^{-1})$, and $\det(\sigma_s)$ which are independent of $s\in(-h,h)$.
For a matrix $A\in\BWt$, we have the following sequence of inequalities:
\begin{align*}
\frac{1}{\sqrt{2}}\tr(A)=\frac{1}{\sqrt{2}}\tr(A\cdot\Id)\leq\frac{1}{\sqrt{2}}\Vert A\Vert_F\Vert\Id\Vert_F=\Vert A\Vert_F=\sqrt{\tr(A^TA)}=\sqrt{\tr(A^2)}=\sqrt{\lambda_1^2+\lambda_2^2}\leq\lambda_1+\lambda_2=\tr(A).
\end{align*}
Hence, it will suffice to prove upper and (positive) lower bounds of $\Vert\sigma_s\Vert_F, \Vert\sigma_s^{-1}\Vert_F$ and $\det(\sigma_s)$, independent of $s\in(-h,h)$. Indeed, note that for $s\in(-h,h)$, we have that
\begin{align*}
\sigma_s=\left(1-\frac{s+h}{2h}\right)\sigma_1+\frac{s+h}{2h}\sigma_2,
\end{align*}
is the sum of positive definite, symmetric matrices. By Weyl's inequality, we obtain that the smallest eigenvalue of $\sigma_s$, denoted as $\lambda_{\min}(\sigma_s)$ is bounded from below by $\left(1-\frac{s+h}{2h}\right)\lambda_{\min}(\sigma_1)+\frac{s+h}{2h}\lambda_{\min}(\sigma_2)$ which is bounded from below by $\xi\coloneqq\min(\lambda_{\min}(\sigma_1),\lambda_{\min}(\sigma_2))$, and is in particular independent of $s$. Similarly, one obtains $\lambda_{\max}(\sigma_s)\leq\zeta,$ where $\zeta\coloneqq\max(\lambda_{\max}(\sigma_1),\lambda_{\max}(\sigma_2))$. Hence, one finally obtains $\sqrt{2}\xi\leq\Vert\sigma_s\Vert\leq 2\zeta, \sqrt{2}\zeta^{-1}\leq\Vert\sigma_s^{-1}\Vert\leq 2\xi^{-1}$, $\xi^2\leq\det(\sigma_s)\leq\zeta^2$ for all $s\in[-h,h]$.
\end{proof}
\begin{rmrk}
This can be considered as the analogue (with a slight improvement on the exponent) of \cite{RuWi15}, Lemma 5.7.
\end{rmrk}
\begin{lmm}
Let $m_i\in\R^2$ and $\sigma_i\in\BWt$ for $i=1,2,3$ such that $\Vert\sigma_i-\sigma_2\Vert_F\leq 2h$ for $i=1,3$, and $\left\Vert\sigma_2-\frac{\sigma_3+\sigma_1}{2}\right\Vert_F\leq h^2$. Then, we can verify the estimates
\begin{align}
\Wass^2(\mu_2,\Bary_{\mu_2}(\mu_1,\mu_3))=\Wass^2(\mu_2,\Bary(\mu_1,\mu_3))+\mathcal{O}(h^5),
\end{align}
and
\begin{align}\label{eq:polar}
\Wass^2(\mu_2,\Bary_{\mu_2}(\mu_1,\mu_3))=\tfrac12 \Wass^2(\mu_1,\mu_2)+\tfrac12 \Wass^2(\mu_3,\mu_2)-\tfrac14 \Wass^2(\mu_1,\mu_3)+\mathcal{O}(h^5),
\end{align}
where $\mu_i\coloneqq\mathcal{N}(m_i,\sigma_i^2)$ for $i=1,2,3$.
\end{lmm}
This lemma can be proven fully analogously to the previous one. However, as one needs to expand the terms up to the fifth order, the computations become extremely lengthy, so we will leave out the explicit computations and simply give a sketch of the proof:

Let $(\sigma(t))_{t\in[-h,h]}$ be the uniquely determined second-order $\BWt$-valued curve, such that $\sigma(-h)=\sigma_1$, $\sigma(0)=\sigma_2$ and $\sigma(h)=\sigma_3$, and let $(\mu_t)_{t\in[-h,h]}$ be the respective measure-valued curve.  Then, define
\begin{align*}
F_1(t)&\coloneqq\Wass^2(\mu_2,\Bary_{\mu_2}(\mu_{-t},\mu_t)),\\
F_2(t)&\coloneqq\Wass^2(\mu_2,\Bary(\mu_{-t},\mu_t)),\\
F_3(t)&\coloneqq\tfrac12 \Wass^2(\mu_{-h},\mu_2)+\tfrac12 \Wass^2(\mu_h,\mu_2)-\tfrac14 \Wass^2(\mu_{-h},\mu_h).
\end{align*}
Then, one has $F_1(h)=\Wass^2(\mu_2,\Bary_{\mu_2}(\mu_{1},\mu_3))$, $F_2(h)=\Wass^2(\mu_2,\Bary(\mu_{1},\mu_3))$, and $F_3(h)= \tfrac12 \Wass^2(\mu_{1},\mu_2)+\tfrac12 \Wass^2(\mu_3,\mu_2)-\tfrac14 \Wass^2(\mu_{1},\mu_3).$ Now, we can explicitly compute any derivatives of the $F_i$, and expand them at $t=0$ up to the fifth order, i.e. 
\begin{align*}
F_i(h)=\sum_{k=0}^4\frac{1}{k!} F_i^{(k)}(0)h^k+\frac{1}{5!}F_i^{(5)}(s_i)h^5,
\end{align*}
for some $s_i\in(0,h).$ Now, one checks that all derivatives up to the third order vanish for $i=1,2,3$. The zeroth-order derivative vanishing is trivial, while the first and third order derivatives vanish at $0$ due to the symmetry of the derivatives of lesser order. For the fourth order derivatives, one checks that they all coincide for $i=1,2,3$. Now, it remains to show that the fifth-order derivative can be bounded by a constant independent of $h$. To this end, one uses Lemma \ref{LemmaStdBar} and the same estimation strategy as inside the proof of 
this lemma.
\begin{rmrk}
Let $t\mapsto(m_t,\sigma_t)$ for $t\in[0,1]$ be a curve in $\R^d\times\BW$, such that $(\sigma_t)_{t\in[0,1]}$ is simultaneously diagonalizable. After choosing a common diagonal basis, one may without loss of generality regard $(\sigma_t)_{t\in[0,1]}$ as a curve in $\R^d\times\BWd$ instead.
\end{rmrk}
\subsection{The case of Gaussian distributions with diagonal covariance matrices}
\begin{dfntn}
Let $\Phi$ be the map from Definition \ref{def:gauss_space}, i.e. $\Phi(m,\sigma)=\mathcal{N}(m,\sigma^2)$ for $(m,\sigma)\in\R^d\times\BW$. Then, we define $\ProbWGd\coloneqq\Phi(\R^d\times\BWd)$ as the space of non-degenerate Gaussian distributions with diagonal covariance matrices. 
\end{dfntn}
\begin{crllr}
Let $(m_t, \sigma_t)_t$ be a curve in $C^3([0,1],\R^d\times\BWd)$, and let $(\mu_t)_t\coloneqq\mathcal{N}(m_t,\sigma_t^2)$ be the respective $\ProbWGd$-valued curve. Then, we have 
\begin{align}\label{eq:diag_path_energy}
\pathenergy((\mu_t)_t)&=\int_0^1\left\Vert\dot{\sigma}_t\right\Vert_F^2\d t+\int_0^1\vert\dot{m}_t\vert^2\d t,\\
\label{eq:diag_spline_energy}\splineenergy((\mu_t)_t)&=\int_0^1\left\Vert\ddot{\sigma}_t\right\Vert_F^2\d t+\int_0^1\vert\ddot{m}_t\vert^2\d t,
\end{align}
where $\splineenergy$ is the spline energy on the space of diagonal Gaussian distributions $\Phi(\R^d\times\BWd)\subset\ProbW(\R^d)$. 
\end{crllr}
\begin{proof}
Using eq. \eqref{eq:consistency_path} and \eqref{eq:consistency_spline}, and assuming $\sigma_t\in\BWd$ for all $t\in[0,1]$, we have
\begin{align*}
\sigma_t^{-1}\left.\frac{\d}{\d h}\right\vert_{h=0}(\sigma_t\sigma_{t+h}^2\sigma_t)^{\tfrac12}=\sigma_t^{-1}\left.\frac{\d}{\d h}\right\vert_{h=0}\sigma_{t+h}\sigma_t=\left.\frac{\d}{\d h}\right\vert_{h=0}\sigma_{t+h}=\dot{\sigma}_t,\\
\sigma_t^{-1}\left.\frac{\d^2}{\d h^2}\right\vert_{h=0}(\sigma_t\sigma_{t+h}^2\sigma_t)^{\tfrac12}=\sigma_t^{-1}\left.\frac{\d^2}{\d h^2}\right\vert_{h=0}\sigma_{t+h}\sigma_t=\left.\frac{\d^2}{\d h^2}\right\vert_{h=0}\sigma_{t+h}=\ddot{\sigma}_t,
\end{align*}
which proves the claim. 
\end{proof}
Alternatively, one can "brute-force" this:
\begin{xmpl}
Let $U\subseteq\R^n$, $V\subseteq\R^m$ be open subsets. Then, $H^k(U,V)\coloneqq W^{k,2}(U,V)$ denotes the Sobolev space of functions $f:U\rightarrow V$, such that $f$ and its weak derivatives up to order $k$ have finite $L^2$-norm. Let $t\mapsto(m_t,\sigma_t)$ for $t\in(0,1)$ be a curve in $H^2((0,1),\R^d\times\BWd)$. Then, by abusing notation we can define $G_t\coloneqq(2\pi)^{-\frac{d}{2}}\det(\sigma_t)^{-1}e^{-\tfrac12 (x-m_t)^T\sigma_t^{-2}(x-m_t)}$, i.e. the Lebesgue density function of $\mu_t\coloneqq\mathcal{N}(m_t,\sigma_t^2)$. We have 
\begin{align*}
\partial_tG_t&=\left[-\tr(\sigma_t^{-1}\dot{\sigma}_t)+\langle \dot{m}_t, \sigma_t^{-2}(x-m_t)\rangle+\langle x-m_t,\sigma_t^{-3}\dot{\sigma}_t(x-m_t)\rangle\right]G_t,\\
\nabla G_t&=\left[-\sigma_t^{-2}(x-m_t)\right]G_t.
\end{align*}
Take $\varphi_t(x) := \langle x,\dot{m}_t\rangle+\tfrac12\langle x-m_t,\dot{\sigma}_t\sigma_t^{-1}(x-m_t)\rangle$. Then, we obtain
\begin{align}\label{eq:velocity_gauss}
\nabla\varphi_t &= \dot{m}_t+\dot{\sigma}_t\sigma_t^{-1}(x-m_t),\\
\Delta\varphi_t &= \tr(\dot{\sigma}_t\sigma_t^{-1}).\nonumber
\end{align}
Thus, the pair $(\mu_t, v_t)$ with $v_t=\nabla\varphi_t$ satisfies (CE), i.e.
\begin{align*}
\partial_t\mu_t+\nabla\cdot(\nabla\varphi_t\mu_t)=\partial_t\mu_t+\Delta\varphi_t\mu_t+\nabla\varphi_t\cdot\nabla\mu_t\equiv 0.
\end{align*}
Moreover, we have that
\begin{align*}
\nabla\dot{\varphi}_t &= \ddot{m}_t+\ddot{\sigma}_t\sigma_t^{-1}(x-m_t)-\dot{\sigma}_t\sigma^{-1}_t\dot{m}_t-\dot{\sigma}_t^{2}\sigma^{-2}_t(x-m_t),\\
\tfrac12\nabla\vert\nabla\varphi_t\vert^2&=\nabla^2\varphi_t\nabla\varphi_t=\dot{\sigma}_t\sigma_t^{-1}\left(\dot{m}_t+\dot{\sigma}_t\sigma_t^{-1}(x-m_t)\right),
\end{align*}
and hence we finally compute
\begin{align*}
\nabla\left(\dot{\varphi}_t+\tfrac12\vert\nabla\varphi_t\vert^2\right)=
& \ \ddot{m}_t+\ddot{\sigma}_t\sigma_t^{-1}(x-m_t)-\dot{\sigma}_t\sigma^{-1}_t\dot{m}_t-\dot{\sigma}_t^{2}\sigma^{-2}_t(x-m_t)\\
&+\dot{\sigma}_t\sigma_t^{-1}\left(\dot{m}_t+\dot{\sigma}_t\sigma_t^{-1}(x-m_t)\right)=\ddot{m}_t+\ddot{\sigma}_t\sigma_t^{-1}(x-m_t).
\end{align*}
The continuous spline energy is then given by
\begin{align}\label{eq:spline_gauss_general}
\mathcal{F}(\mu)=\int_0^1\!\int_{\mathbb{R}^N}\vert\ddot{m}_t+\ddot{\sigma}_t\sigma_t^{-1}(x-m_t)\vert^2\d \mu_t\d t=\int_0^1\!\int_{\mathbb{R}^N}\vert\ddot{m}_t\vert^2+\vert\ddot{\sigma}_t\sigma_t^{-1}(x-m_t)\vert^2\d \mu_t\d t=\int_0^1\vert\ddot{m}_t\vert^2+\tr(\ddot{\sigma}_t^2)\d t,
\end{align}
where in the second equality we used the fact that the $\mu_t$-integral of an antisymmetric function (with respect to $m_t$) vanishes, and $\int_{\R^d}\vert A(x-m_t)\vert^2\d \mu_t=\tr(A\sigma_t^2A^T)$ in the last equality (cf. proof of Proposition \ref{prop:consistency}). 
One can simplify the above expression even further:
\begin{align}\label{eq:spline_gauss}
\mathcal{F}(\mu)=\int_0^1\vert\ddot{m}_t\vert^2\d t+\sum_{j=1}^d\int_0^1\vert\ddot{\lambda}_t^j\vert^2\d t,
\end{align}
where $(\lambda_t^j)_{j=1,\ldots,d}$ are the eigenvalues of $\sigma_t$. In view of \eqref{eq:spline_gauss}, one might be tempted to assert that a spline interpolation on the space of Gaussian distributions with simultaneously diagonalizable covariances can be obtained by spline interpolating each eigenvalue independently (after choosing a fixed common eigenbasis). However, this assumption ignores the restriction that $\sigma_t$ is required to be positive definite. Indeed, the spline interpolation of some given Gaussian key-frames amounts to solving the following minimization problem
\begin{align*}
\inf_{m_t\in\R^d,\lambda_t^j>0} \int_0^1\vert\ddot{m}_t\vert^2\d t+\sum_{j=1}^d\int_0^1\vert\ddot{\lambda}_t^j\vert^2\d t,
\end{align*}
together with some point-wise evaluation constraints. If the above minimization problem has a solution, then the spline interpolation results from a classical cubic spline interpolation of $m_t$ and $\lambda_t^j$. The positivity constraint, however, implies that whenever the interpolating eigenvalue-spline becomes negative, it can not coincide with the Wasserstein E-spline.
\end{xmpl}
The above equation \eqref{eq:spline_gauss_general} allows us to canonically identify any $\ProbWGd$-valued functional $\splineenergy$ with a $(\R^d\times\BWd)$-valued functional $\splineenergyt$ via $$\splineenergyt((m_t,\sigma_t)_t)\coloneqq\splineenergy((\mathcal{N}(m_t,\sigma_t^2))_t).$$ \begin{lmm}\label{lmm:continuity_prop}
The regularized spline energy $\splineenergyt^{\delta}$  is lower semi-continuous under weak and continuous under strong convergence in $H^2((0,1),\R^d\times\BWd)$. 
\end{lmm}
\begin{proof}
First, let us note that by \eqref{eq:diag_path_energy} and \eqref{eq:diag_spline_energy}, for any curve $(m_t, \sigma_t)_{t\in[0,1]}$ the path energy $\pathenergyt$ and the spline energy $\splineenergyt$ coincide with the squared semi-norms $\vert\cdot\vert_{H^1}^2$ and $\vert \cdot\vert^2_{H^2}$, respectively, both of which are weakly lower semi-continuous on $H^2$. Thus, $$\liminf_{n\rightarrow\infty}\splineenergyt^\delta((m_t^{(n)},\sigma_t^{(n)})_t)\geq\splineenergyt^\delta((m_t,\sigma_t)_t),$$ for a $H^2((0,1);\R^d\times\BWd)$-weakly convergent sequence $(m_t^{(n)},\sigma_t^{(n)})_t\rightharpoonup(m_t,\sigma_t)_t$. Moreover, if the sequence converges strongly we also obtain $\liminf_{n\rightarrow\infty}\splineenergyt^\delta((m_t^{(n)},\sigma_t^{(n)})_t)=\splineenergyt^\delta((m_t,\sigma_t)_t)$, since strong convergence implies convergence in norm.
\end{proof}
\section{Convergence of discrete Gaussian E-splines}\label{sec:mosco_convergence}
In this section we will discuss the convergence of discrete Gaussian spline curves to continuous Gaussian E-splines. For the sake of presentation, we will at first only consider centered Gaussian curves, i.e. $m_t=0$ for all $t\in[0,1]$. As the energy of the mean and standard deviation matrices decouples (cf. equations \eqref{eq:spline-energy-decoupling} and \eqref{eq:diag_spline_energy}), this is a very natural approach: for the general non-centered case, one can directly apply the results from \cite{HeRuWi18} to obtain Mosco convergence (cf. \cite{Mo69}) of the mean term of the functionals. As the space of standard deviation matrices $\BWd$ has a non-trivial boundary the results of \cite{HeRuWi18} do not apply to the standard deviation matrix term. Nevertheless we follow the general procedure of the proof in this paper.

In the sequel, we will focus on natural boundary conditions with a comment on the periodic case below. We will now use a suitable interpolation to identify discrete curves with continuous ones to be able to rewrite the discrete energy as a functional on time-continuous curves. 
As in \cite{HeRuWi18} and \cite{JuRaRu23}, this will be done via cubic Hermite interpolation at time interval midpoints. 
For a tuple $\bm\sigma^K\coloneqq(\sigma_0,\ldots,\sigma_K)\in(\BWd)^{K+1}$, we define the temporal extension $\eta_{\bm\sigma^K}$ of $\bm\sigma^K$ as
\begin{align*}
\eta_{\bm\sigma^K}(t)\coloneqq\begin{cases}\sigma_0+(\sigma_1-\sigma_0)Kt, \ \ \ t\in[0,t^K_{1/2}],\\
\frac{\sigma_{k-1}+\sigma_k}{2}+(\sigma_k-\sigma_{k-1})K(t-t^K_{k-1/2})+(\sigma_{k+1}-2\sigma_k+\sigma_{k-1})K^{2}\frac{(t-t^K_{k-1/2})^2}{2}, \ \ \  t\in[t^K_{k-1/2},t^K_{k+1/2}],\\
\sigma_{K-1}+(\sigma_K-\sigma_{K-1})K(t-t^K_{K-1}), \ \ \ t\in[t^K_{K-1/2},1],
\end{cases}
\end{align*}
where $t^K_{k+1/2}\coloneqq\frac{k+1/2}{K}$ for $k=0,\ldots,K-1$.
For periodic boundary conditions, we can neglect the definition on the starting and final half-intervals, identifying $[t^K_{K-1/2},t^K_{K+1/2}]$ with $[0,t^K_{1/2}]\cup[t^K_{K-1/2},1]$. 

Next, we recall  the following convergence of a piecewise cubic Hermite 
interpolation from 
\cite[Lemma 4.3]{HeRuWi18}.
\begin{lmm}[Strong convergence of piecewise cubic Hermite curves to smooth Gaussian curves]\label{lmm:strong_conv_pw_bezier} 
Let $\sigma=(\sigma(t))_{t\in[0,1]}$ be a $C^3$ curve in $\BWd$, and define $\bm\sigma^K\coloneqq(\sigma_j^K)_{j=0,\ldots,K}=(\sigma(j/K))_{j=0,\ldots,K}$, i.e. $\bm\sigma^K$ is an equidistant sampling of the continuous curve $\sigma$ with $K+1$ samples. Then, $\eta_{\bm\sigma^K}$ converges strongly in $H^2([0,1],\BWd)$ to $\sigma$ for $K\rightarrow\infty$. 
\end{lmm}

The next two lemmas compare the discrete path and spline energy with the corresponding 
continuous counterpart evaluated on the piecewise cubic Hermite  interpolation. To this end, we define the hat operator for discrete functionals: for a functional $\SplineEnergy^K$ on $(\ProbWGd)^{K+1}$ we define $\SplineEnergyt^K$ on $(\R^d\times\BWd)^{K+1}$ via $$\SplineEnergyt^K((m_k,\sigma^2_k)_{k=0,\ldots,K})\coloneqq\SplineEnergy^K((\mathcal{N}(m_k,\sigma^2_k))_{k=0,\ldots,K}).$$
\begin{lmm}[Path energy estimate]\label{lmm:path_nrg_estimate}
Let $\sigma=(\sigma(t))_{t\in[0,1]}\in H^2((0,1),\BWd)$, and define $\bm\sigma^K$ as above. Then, for $K$ big enough, we have
$
\vert\pathenergyt[\eta_{\bm\sigma^K}]-\PathEnergyt^K[\bm\sigma^K]\vert\leq CK^{-1},
$
where $\pathenergy$ and $\PathEnergy^K$ have been defined in \eqref{eq:path_energy} and \eqref{discrete_path_energy_def}, respectively, and the constant $C$ depends only on the curve $\sigma$.
\end{lmm}
\begin{proof}
By the definition of ${\eta}_{\bm\sigma^K}(t)$, and using \eqref{eq:diag_path_energy} and Proposition \ref{prop:gauss_facts} (2) for the expressions of $\pathenergyt$ and $\PathEnergyt$, respectively, we obtain
\begin{align*}
\vert\pathenergyt[\eta_{\bm\sigma^K}]-\PathEnergyt^K[\bm\sigma^K]\vert=&\ \left\vert K\sum_{k=1}^{K-1}\vert\sigma^K_k-\sigma^K_{k-1}\vert^2+\sum_{k=1}^{K-1}\int_{t^K_{k-1/2}}^{t^K_{k+1/2}}\vert\sigma^K_{k+1}-2\sigma^K_k+\sigma^K_{k-1}\vert^2K^4(t-t^K_{k-1/2})^2\d t\right.\\
&\left.\ \ +2\sum_{k=1}^{K-1}\int_{t^K_{k-1/2}}^{t^K_{k+1/2}}(\sigma^K_k-\sigma^K_{k-1})(\sigma^K_{k+1}-2\sigma^K_k+\sigma^K_{k-1})K^3(t-t^K_{k-1/2})\d t-K\sum_{k=1}^{K}\vert\sigma^K_k-\sigma^K_{k-1}\vert^2\right\vert\\
&\leq K\sum_{k=1}^{K-1}\vert\sigma^K_k-\sigma^K_{k-1}\vert\vert\sigma^K_{k+1}-2\sigma^K_k+\sigma^K_{k-1}\vert+\frac{K}{3}\sum_{k=1}^{K-1}\vert\sigma^K_{k+1}-2\sigma^K_k+\sigma^K_{k-1}\vert^2\\
&\leq K\left(\sum_{k=1}^{K-1}\vert\sigma^K_k-\sigma^K_{k-1}\vert^2\right)^{\tfrac12}\left(\sum_{k=1}^{K-1}\vert\sigma^K_{k+1}-2\sigma^K_k+\sigma^K_{k-1}\vert^2\right)^{\tfrac12}+\frac{K}{3}\sum_{k=1}^{K-1}\vert\sigma^K_{k+1}-2\sigma^K_k+\sigma^K_{k-1}\vert^2\\
&\leq C'KK^{-\tfrac12}\vert\sigma\vert_{H^1}K^{-\tfrac32}\vert\sigma\vert_{H^2}+C''K^{-2}\vert\sigma\vert^2_{H^2}\leq CK^{-1},
\end{align*}
where we used $K\sum_{k=1}^{K-1}\vert\sigma^K_k-\sigma^K_{k-1}\vert^2\leq C^*\vert\sigma\vert^2_{H^1}$, and $K^3\sum_{k=1}^{K-1}\vert\sigma^K_{k+1}-2\sigma^K_{k}+\sigma^K_{k-1}\vert^2\leq C''\vert\sigma\vert^2_{H^2}$ (cf. proof of \cite[Lemma 4.2]{HeRuWi18}). The final inequality holds for $K$ chosen big enough.
\end{proof}
\begin{lmm}[Spline energy estimate]\label{lmm:spline_nrg_estimate}
Let $\sigma=(\sigma(t))_{t\in[0,1]}\in H^2((0,1),\BWd)$, and define $\bm\sigma^K$ as above. Then, we have
$\splineenergyt[\eta_{\bm\sigma^K}]=\SplineEnergyt^K[\bm\sigma^K]$,
where $\splineenergy$ and $\SplineEnergy^K$ have been defined in \eqref{def:spline_energy_cont} and \eqref{discrete_spline_energy}, respectively.
\end{lmm}
\begin{proof}
Recall that $\ddot{\eta}_{\bm\sigma^K}(t)=(\sigma^K_{k+1}-2\sigma^K_k+\sigma^K_{k-1})K^{2}$ for $t\in[t^K_{k-1/2},t^K_{k+1/2}]$, $k=1,\ldots,K-1$, and $0$ otherwise.
We then obtain by using \eqref{eq:diag_spline_energy} and Proposition \ref{prop:gauss_facts}
\begin{align*}
\splineenergyt[\eta_{\bm\sigma^K}]-\SplineEnergyt^K[\bm\sigma^K]=4K^3\sum_{k=1}^{K-1}\left\vert\sigma^K_k-\frac{\sigma^K_{k+1}+\sigma^K_{k-1}}{2}\right\vert^2-\SplineEnergyt^K[\bm\sigma^K] =0.
\end{align*}
\end{proof}
We are now in the position to prove the convergence of the discrete spline functional to the continuous one. To this end, we introduce two indicator functions to filter the constraints.
For the continuous problem
$\mathcal{I}[\sigma]=0$ if $\mu_t\coloneqq\mathcal{N}(0,\sigma_t^2)$ satisfies the evaluation constraints \eqref{eq:interp_constraints} (with $\overline\mu_i\coloneqq\mathcal{N}(0,\overline\sigma_i^2)$ for given interpolation constraints $(\overline t_i,\overline\sigma_i)_{i=1,\ldots,I}$) as well as the corresponding boundary condition from \eqref{eq:naturalbc}-\eqref{eq:periodicbc} and  $\infty$ else. Furthermore, for the discrete problem
$\mathcal{I}^K[\sigma]=0$ if $\sigma=\eta_{(\sigma_0,\ldots,\sigma_K)}$ for some $(\sigma_0,\ldots,\sigma_K)\in (\BWd)^{K+1}$, where $\mathcal{N}(0, \sigma_i^2)_{i=0,\ldots,K}$ 
satisfies the evaluation constraints \eqref{eq:interp_constraints_d} 
as well as the corresponding discrete boundary condition from \eqref{eq:naturalbc_d}-\eqref{eq:periodicbc_d}, 
and $\infty$ else.

Regarding the compatibility (cf. equation \eqref{eq:interp_constraints_d}) of the given interpolation times $\overline t_i$ and the number $K+1$ of points along a discrete curve, we shall in the following and without explicit mention always interpret $K\rightarrow\infty$ as a sequence of natural numbers approaching infinity, such that $K\overline t_i\in\N_0$ for $i=1,\ldots,I$.

Before stating the main result of this section, let us recall the definition of Mosco convergence on metric vector spaces. 
\begin{dfntn}
A sequence of functionals  $\splineenergy^K:\Sigma\rightarrow\overline{\R}$ on a metric vector space $\Sigma$ is said to converge to the functional $\splineenergy:\Sigma\rightarrow\overline{\R}$ in the sense of Mosco, if the following conditions hold:
\begin{itemize}
\item \textit{weak liminf inequality}: For every sequence $(\sigma^K)_{K\in\N}\subset\Sigma$, such that $\sigma^K\rightharpoonup\sigma$, it holds $$\splineenergy(\sigma)\leq\liminf_{K\rightarrow\infty}\splineenergy^K(\sigma^K).$$
\item \textit{strong limsup inequality}: For every $\sigma\in \Sigma$ there is a sequence $(\sigma^K)_{K\in\N}$ with $\sigma^K\rightarrow\sigma$, such that $$\splineenergy(\sigma)\geq\limsup \splineenergy^K(\sigma^K).$$
\end{itemize}
\end{dfntn}
Then, we finally obtain the following theorem:
\begin{thrm}[Mosco convergence and convergence of discrete minimizers]
Let $\pathenergyt^K[\sigma]$ be given by $\PathEnergyt[\bm\sigma^K]$ if $\sigma=\eta_{\bm\sigma^K}$, and $\infty$ otherwise. Similarly, define $\splineenergyt^K[\sigma]$ be given by $\SplineEnergyt[\bm\sigma^K]$ if $\sigma=\eta_{\bm\sigma^K}$, and $\infty$ otherwise. Then, set $\splineenergyt^{\delta,K}\coloneqq\splineenergyt+\delta\pathenergyt$. With respect to the weak topology in $H^2((0,1);\BWd)$ we have $\lim_{K\rightarrow\infty}\splineenergyt^{\delta,K}+\mathcal{I}^K=\splineenergyt^\delta+\mathcal{I}$ in the sense of Mosco, for $\delta>0$. Moreover, any sequence $(\sigma^K)_K$ with $\splineenergyt^{\delta,K}[\sigma^K]+\mathcal{I}^K[\sigma^K]$ uniformly bounded contains a subsequence that converges weakly in $H^2((0,1);\BWd)$. As a consequence, any sequence of minimizers of $\splineenergyt^{\delta,K}+\mathcal{I}^K$ contains a subsequence converging weakly to a minimizer of $\splineenergyt^{\delta}+\mathcal{I}$.
\end{thrm}
\begin{proof}
We have to show the \textit{weak liminf} and the \textit{strong limsup} inequalities defining Mosco convergence. 
Concerning the weak $\liminf$-inequality, we need to show that for every sequence $(\sigma^K)_{K\in\N}\subset H^2((0,1);\BWd)$, such that $\sigma^K\rightharpoonup\sigma$, it holds that $\liminf_{K\rightarrow\infty}\splineenergyt^{\delta,K}[\sigma^K]+\mathcal{I}^K[\sigma^K]\geq\splineenergyt[\sigma]+\mathcal{I}[\sigma]$. Let $\sigma^K\rightharpoonup\sigma$ in $H^2((0,1);\BWd)$. Upon taking a subsequence, we may replace the $\liminf$ by an actual $\lim$ and may assume without loss of generality $\splineenergyt^{\delta,K}[\sigma^K]+\mathcal{I}^K[\sigma^K]\leq C$ for some constant $C<\infty$. Thus, we have $\sigma^K=\eta_{(\sigma^K_0,,\ldots,\sigma^K_K)}$ for some $(\sigma^K_0,\ldots,\sigma^K_K)\in (\BWd)^{K+1}$ and this estimate implies
\begin{align*}
d_K\coloneqq\max_{k\in\lbrace1,\ldots,K\rbrace}\vert\sigma^K_k-\sigma^K_{k-1}\vert\leq\sqrt{\sum_{k=1}^K\vert\sigma^K_k-\sigma^K_{k-1}\vert^2}=\sqrt{\pathenergyt^K[\sigma^K]/K}\leq\sqrt{\frac{C}{\delta K}},
\end{align*}
which converges to zero as $K\rightarrow\infty$. Next, we show that $\mathcal{I}[\sigma]=0$. It is straightforward to see that 
$\eta_{\bm\sigma^K}(t)$ is in the convex hull of $\sigma_{k-1}$, $\sigma_{k}$, and $\sigma_{k+1}$
for $t\in[t^K_{k-1/2},t^K_{k+1/2}]$.
Thus, the evaluation constraint is satisfied in the limit. To conclude, by the weak lower semi-continuity of $\splineenergyt^\delta$ due to Lemma \ref{lmm:continuity_prop}, and by Lemmas \ref{lmm:path_nrg_estimate} and \ref{lmm:spline_nrg_estimate} we obtain
\begin{align*}
\splineenergyt^\delta[\sigma]+\mathcal{I}[\sigma]=\splineenergyt^\delta[\sigma]\leq\liminf_{K\rightarrow\infty}\splineenergyt^\delta[\sigma^K]\leq\liminf_{K\rightarrow\infty}\splineenergyt^{\delta,K}[\sigma^K]+\delta\tfrac{C}{K}\vert\sigma^K\vert_{H^1}\vert\sigma^K\vert_{H^2}\leq\liminf_{K\rightarrow\infty}\left(\splineenergyt^{\delta,K}[\sigma^K]+\mathcal{I}^K[\sigma^K]\right),
\end{align*}
where we used the uniform boundedness of $\vert\sigma^K\vert_{H^1}$ and $\vert\sigma^K\vert_{H^2}$ due to the weak convergence of $\sigma^K$.\\

Concerning the strong $\limsup$ inequality, we need to show that for every $\sigma\in  H^2((0,1);\BWd)$ there is a sequence $(\sigma^K)_{K\in\N}\subset H^2((0,1);\BWd)$ with $\sigma^K\rightarrow\sigma$, such that $\splineenergyt^\delta[\sigma]+\mathcal{I}[\sigma]\geq\limsup_{K\rightarrow\infty}\splineenergyt^{\delta,K}[\sigma^K]+\mathcal{I}^K[\sigma^K]$. Let $\sigma\in C^3([0,1],\BWd)$ with finite energy (in particular, $\mathcal{I}(\sigma)=0$), and choose $\sigma^K=\eta_{\bm\sigma^K}$ as the recovery sequence. By definition, we have $\mathcal{I}^K[\sigma^K]=0$. As $K\rightarrow\infty$, we have $d_K\coloneqq\max_{\lbrace 1,\ldots,K\rbrace}\vert\sigma(t_k^K)-\sigma(t^K_{k-1})\vert\rightarrow 0$, as well as $\sigma^K\rightarrow\sigma$ strongly in $H^2$ by Lemma \ref{lmm:strong_conv_pw_bezier}. Thus, by the strong $H^2$-continuity of $\splineenergyt^\delta$ from Lemma \ref{lmm:continuity_prop} and by Lemma \ref{lmm:spline_nrg_estimate} we have
\begin{align*}
\splineenergyt^\delta[\sigma]+\mathcal{I}[\sigma]&
=\splineenergyt^\delta[\sigma]=\lim_{K\rightarrow\infty}\splineenergyt^\delta[\eta_\sigma^K]\\
&\geq\limsup_{K\rightarrow\infty}\splineenergyt^{\delta,K}[\sigma^K]-\delta CK^{-1}\vert\sigma^K\vert_{H^1}\vert\sigma^K\vert_{H^2}=\limsup_{K\rightarrow\infty}\left(\splineenergyt^{\delta,K}[\eta_\sigma^K]+\mathcal{I}^K[\sigma^K]\right),
\end{align*}
where we again used the uniform boundedness of $\vert\sigma^K\vert_{H^1}$ and $\vert\sigma^K\vert_{H^2}$, now due to the strong convergence of $\sigma^K$. Thus, we obtain $\limsup\splineenergyt^{\delta,K}+\mathcal{I}^K\leq\splineenergyt^\delta+\mathcal{I}$ on $C^3([0,1],\BWd)$. By a density argument ($C^3$ functions fulfilling interpolation constraints are dense in the space of $H^2$ functions satisfying the same interpolation constraints, cf. \cite[Lemma 4.6]{HeRuWi18}) and the strong $H^2$ continuity of $\splineenergyt^\delta$, we obtain
\begin{align*}
\splineenergyt^\delta[\sigma]+\mathcal{I}[\sigma]\geq\limsup_{K\rightarrow\infty}\splineenergyt^{\delta,K}[\sigma^K]+\mathcal{I}^K[\sigma^K].
\end{align*}
To show the convergence of discrete minimizers it remains to establish equicoercivity. Here we will follow the same strategy as in the proof of \cite[Theorem 4.9]{HeRuWi18}. Let $(\sigma^K)_K$ be a sequence with $\splineenergyt^{\sigma,K}[\sigma^K]+\mathcal{I}^K[\sigma^K]$ uniformly bounded. As before, we can assume without loss of generality that $\sigma^K=\eta_{(\sigma^K_0,\ldots,\sigma^K_K)}$ for some $(\sigma^K_0,\ldots,\sigma^K_K)\in (\BWd)^{K+1}$. 
Following the proof of Lemma \ref{lmm:path_nrg_estimate} 
and also recalling the estimate $K\sum_{k=1}^{K-1}\vert\sigma^K_k-\sigma^K_{k-1}\vert^2\leq C^*\vert\sigma\vert^2_{H^1}$
we obtain a uniform bound of the $H^1$-seminorm $\vert\sigma^K\vert_{H^1}$.
By Poincar\'e's inequality, one even obtains uniform boundedness of the norm $\Vert\sigma^K\Vert_{H^{1}}$. It remains to show the uniform boundedness of the $H^2$-seminorm $\vert\sigma^K\vert_{H^2}$. Indeed, we obtain the estimate
\begin{align*}
\vert\sigma^K\vert^2_{H^{2}}&=\vert\eta_{(\sigma^K_0,\ldots,\sigma^K_K)}\vert^2_{H^{2}}=4K^3\sum_{k=1}^{K-1}\left\vert\sigma_k^K-\frac{\sigma_{k+1}^K+\sigma_{k-1}^K}{2}\right\vert^2=4K^3\sum_{k=1}^{K-1}\Wass^2[\sigma_k^K,\Bary(\sigma_{k+1}^K,\sigma_{k-1}^K)]\\
&=\SplineEnergyt^K[\sigma_0^K,\dots,\sigma_K^K]\leq\splineenergyt^{\delta,K}[\sigma^K].
\end{align*}
The statement about the convergence of minimizers is now a standard consequence of the Mosco convergence from the previous theorem, cf. \cite{Br14}.
\end{proof}
\section{Fully discrete Wasserstein splines and numerical results}\label{sec:fully_discrete}
\subsection{Algorithmic foundations}\label{sec:algo}
To implement Wasserstein splines numerically, we have to further discretize the time-discrete spline energy in space.
With the application to images in mind, we consider $\Omega\coloneqq[0,1]^2$ ($d=2$) and an 'Eulerian' discretization of probability measures: Let $\mu\in\mathcal{P}(\Omega)$. Here, the image intensities on each colour channel are encoded as Lebesgue densities. We can obtain the discretized version of a density $\mu$ by first defining the computational mesh
\[
\discreteDomain=\left\{\tfrac{0}{M-1},\tfrac{1}{M-1},\ldots,\tfrac{M-1}{M-1}\right\}\times\left\{\tfrac{0}{N-1},\tfrac{1}{N-1},\ldots,\tfrac{N-1}{N-1}\right\}\quad 
\text{for }M,N\geq 3.
\]
Next, we integrate the mass of $\mu$ on each cell $\Omega^{kl}\coloneqq[\frac{k}{M-1},\frac{k+1}{M-1}]\times[\frac{l}{N-1},\frac{l+1}{N-1}]$, and define the weight
\begin{align*}
\omega_{kl}=\int_{\Omega^{kl}}d\mu,
\end{align*}
obtaining the spatially discrete measure $\mu^D[\omega]\coloneqq\sum_{k, l}\omega_{kl}\delta_{kl},$
where $\omega\coloneqq(\omega_{kl})_{k,l}\in\Sigma_{MN}\coloneqq\{\omega\in\R^{MN}_+:\sum_{m,n}\omega_{mn}=1\}$, and $\delta_{kl}$ is defined as the delta distribution located at the center of the cell (pixel) $\Omega^{kl}$. Alternatively, for other applications the 'Lagrangian' discretization might be more useful: In this case, one considers $\domain=\R^d$, and independently samples a measure $\mu\in\ProbW(\domain)$ a total of $L$ times. One then defines the spatially discrete measure $\mu^D[x]\coloneqq\frac{1}{L}\sum_{l=1}^L\delta_{x_l}$, where $x\coloneqq(x_l)_l\in\R^{dL}$. Since obtaining the exact Wasserstein distance between two discrete measures with $L=MN$ atoms comes with a cost of $\mathcal{O}(M^3N^3)$, we approximate the Wasserstein distance between two discrete measures $\mu$ and $\nu$ by the entropy-regularized Wasserstein distance $\WassEps$ introduced originally in \cite{Cu13}, with regularization parameter $\epsilon>0$. The loss $\WassEps$ can be very efficiently computed in an auto-differentiable manner, i.e. the gradients of $\WassEps(\cdot,\cdot)$ with respect to both the weights $\omega$ and locations $x$ are obtained as a by-product of the evaluation of this function (cf. \cite{CuDo14}) with state-of-the-art implementations of the Sinkhorn algorithm, such as in \cite{Sc19} and \cite{ChFe21}. 
Correspondingly, we take into account the entropy-regularized approximation $\BaryEps_{(\cdot)}$ (cf. \cite{BeCa14}) of the (generalized) barycenter, which, once again, can be efficiently computed in an auto-differentiable fashion. The entropy-relaxed, regularized spline objective functional will then look as follows:
\begin{align}\label{eq:fully_discrete_spline_funct}
\SplineEnergy^{\delta,K,\epsilon}_{(G)}(\bm\mu^D[\theta^0,\ldots,\theta^K])&\coloneqq 4K^3\sum_{k=1}^{K-1}\WassEps^2(\mu_k^D[\theta^k],\BaryEps_{(\mu_k^D[\theta^k])}(\mu^D_{k+1}[\theta^{k+1}],\mu^D_{k-1}[\theta^{k-1}]))\\\nonumber
&+\delta K\sum_{k=0}^{K-1}\WassEps^2(\mu^D_k[\theta^k],\mu^D_{k+1}[\theta^{k+1}]),
\end{align}
where we omitted the $K$ super-index. For a fixed computational domain $\domain_{MN}$, temporal resolution $K$, entropy-regularization $\epsilon>0$, regularizer $\delta>0$, interpolation constraints and a chosen boundary condition, our aim in the first discretization variant is to minimize the previous functional with respect to the weights $\theta^k=\omega^k\in\Sigma_{MN}$ of the discrete probability measures $\mu_k^D=\sum_{y\in\domain_{MN}}\omega^k_y \delta_y$ for indices $k$ that are not fixed by the interpolation conditions. The explicit minimization of functional \eqref{eq:fully_discrete_spline_funct} as a function of weights is performed by Algorithm \ref{algo:optimization}. For the second variant, one can instead fix a number of samples/locations $L$, and straightforwardly minimize the above functional over the positions $\theta^k=x^k\in\R^{dL}$ of atoms of the discrete probability measures $\mu_k^D\coloneqq\frac{1}{L}\sum_{l=1}^L\delta_{x^k_l}$. This can be implemented completely analogously to Algorithm \ref{algo:optimization}.
\begin{algorithm}
\SetInd{1ex}{1ex}
$t=0$\;
\For{$k=0$ \KwTo $K \operatorname{and} k \operatorname{not\ fixed}$}{
$\tilde\omega^k=\hat\omega^k=\Id_{MN}/MN$\;
}
\While{$\operatorname{not\ converged}$}{
$\beta=(t+1)/2$\;
\For{$k=0$ \KwTo $K \operatorname{and} k \operatorname{not\ fixed}$}{
\tcc{update weights (Nesterov's accelerated gradient update)}
$\omega^k=(1-\beta^{-1})\hat\omega^k+\beta^{-1}\tilde\omega^k$\;
\tcc{compute gradient (Sinkhorn algorithm)}
$\operatorname{grad}_k=\nabla_{\omega^k}\SplineEnergy^{\delta,K,\epsilon}_{(G)}(\bm\mu^D[\omega^0,\ldots,\omega^K])$\;
\tcc{update weights}
$\tilde\omega^k=\tilde\omega^k\odot e^{-t\beta\operatorname{grad}_k}$\;
$\tilde\omega^k=\tilde\omega^k/(\tilde\omega^k)^T\Id_{MN}$\;
$\hat\omega^k=(1-\beta^{-1})\hat\omega^k+\beta^{-1}\tilde\omega^k$\;
}
$t=t+1$\;
}
\caption{Algorithm for minimizing $\SplineEnergy_{(G)}^{\delta,K,\epsilon}$ as a function of weights $\omega^0,\ldots,\omega^K\in\Sigma_{MN}$. The product $\odot$ and the exponential function $e$ act component-wise on vectors, and $\Id_{MN}=(1,\ldots,1)\in\R^{MN}$.}
\label{algo:optimization}
\end{algorithm}
\subsection{Numerical results}\label{sec:results}
In what follows, we investigate and discuss qualitative
properties of the spline interpolation in the space of probability distributions, being aware that the superior temporal smoothness of this interpolation is difficult to show with series of still images. 

Figure~\ref{fig:GaussianRestricted} shows discrete piecewise geodesic and discrete spline interpolations of 
three two-dimensional Gaussian distributions $\overline\mu_i=\mathcal{N}(m_i,\sigma_i^2)$  for $i=1,2,3$ at the prescribed times $\overline t_1=0, \overline t_2=\tfrac12, \overline t_3=1$, where the interpolation is computed as a minimizer of $\SplineEnergy^K$ over all Gaussian parameters $(m_k,\sigma_k)\in\R^d\times\BW$ for $k=0,\ldots,K$. For the discrete splines the center of masses of spline interpolation correspond almost perfectly to the cubic 
spline interpolation of the center of masses of the key frames. 
The third row shows the spline interpolation result for the same key frames. 
This time, we instead optimize functional \eqref{eq:fully_discrete_spline_funct} over all weights $\theta^k=\omega^k\in\Sigma_{MN}$ for $k=0,\ldots,K$ and $M=N=128$. In particular, the solutions need not to be Gaussian distributions. The fourth row shows the difference between the first and second row, i.e. between the piecewise geodesic and spline interpolations. 

The next example in Figure~\ref{fig:Circle_Annulus} investigates the interpolation of three key frames with constant density on an annulus for the first and
constant density on a disk for the second and third (at times $\overline t_1=0, \overline t_2=\tfrac12, \overline t_3=1$). In case of the piecewise geodesic interpolation one observes a 
decreasing density on the closing annulus in between the first two key frames and obviously constant interpolation in between the second and third key frames.
In case of the spline interpolation ($\delta=0$) the annulus also closes between the first and second key frames
but shows a strong overshooting at the center between the second and third key frames. 

In Figure~\ref{fig:circle_square} a thin annulus shaped distribution and two times an equal square shaped frame are taken into account as key frame distributions (at times $\overline t_1=0, \overline t_2=\tfrac12, \overline t_3=1$).
Different from Figure 3 in \cite{JuRaRu23} in the case of spline interpolations in the metamorphosis model, one does not observe strongly inward pointing edges between the equal square shaped frames. Instead strong overshooting effects are visible at the corners of the squares.  

In Figure~\ref{fig:Two_gaussians} the key frames consist of pairs of Gaussians with 
constant mass and constant variance, which are far apart for the first and fourth key frame and close by for the second and third key frame (at times $\overline t_1=0, \overline t_2=\tfrac13, \overline t_3=\tfrac23, \overline t_4=1$). Piecewise geodesic interpolation leads to piecewise linear trajectories of the center of masses, whereas trajectories are curved in the spline case with a merger of the two 
bumps in between the second and third key frame.

In Figure~\ref{fig:Gaussian_break}, three key frames represent a single Gaussian, a pair of vertically displaced Gaussian of half the mass, 
and the vertically displaced configuration rotated by $-\tfrac\pi4$ (at times $\overline t_1=0, \overline t_2=\tfrac12, \overline t_3=1$). The piecewise geodesic shows the splitting of mass and approximately straight line trajectories
between the pairs of key frames. For the spline interpolation one observes an overshooting with a positive rotation angle in between the first and the second key frame.

In Figure~\ref{fig:decoupling}, we leverage the results of equation \eqref{eq:spline-energy-decoupling}. An implementation of the 
decoupling leads to a significant decrease of the computing time and the number of iterations with no apparent loss of detail.
\begin{figure*}[htb]
\resizebox{\linewidth}{!}{
\tikzstyle{frame} = [line width=1.8pt, draw=red,inner sep=0.01em]
\begin{tikzpicture}
\begin{scope}[scale=1.4]
\begin{scope}
\edef\currentCnt{0}
\foreach \i in  {0,1,2,3,4,5,6,7,8} {
\pgfmathparse{int(\i*2)}
\xdef\ii{\pgfmathresult}
\ifthenelse{0 = \i \OR 8 = \i \OR 4 = \i}{
\node[frame,anchor=south west] at (1.*\currentCnt+0.08,0) {\includegraphics[width=0.15\linewidth]{geod__\ii_cropped.png}};}
{
\node[anchor=south west] at (1.*\currentCnt,0-0.08) {\includegraphics[width=0.15\linewidth]{geod__\ii_cropped.png}};
}
\pgfmathparse{\currentCnt+1.9}
\xdef\currentCnt{\pgfmathresult}
}
\end{scope}

\begin{scope}[shift={(0,-1.9)}]
\edef\currentCnt{0}
\foreach \i in  {0,1,2,3,4,5,6,7,8} {
\pgfmathparse{int(\i*2)}
\xdef\ii{\pgfmathresult}
\ifthenelse{0 = \i \OR 8 = \i \OR 4 = \i}{
\node[frame,anchor=south west] at (1.*\currentCnt+0.08,0) {\includegraphics[width=0.15\linewidth]{spline__\ii_cropped.png}};}
{
\node[anchor=south west] at (1.*\currentCnt,0-0.08) {\includegraphics[width=0.15\linewidth]{spline__\ii_cropped.png}};
}
\pgfmathparse{\currentCnt+1.9}
\xdef\currentCnt{\pgfmathresult}
}
\edef\currentCnt{0}
\node[rotate=90] at (-0.2,0.9) {\Large{spline}};
\node[rotate=90] at (-0.2,2.8) {\Large{p.w. geodesic}};
\node[rotate=90] at (-0.2,-2.7) {\Large{difference}};
\node[anchor=south west] at (17.1,0.)
{\includegraphics[width=0.3\linewidth]{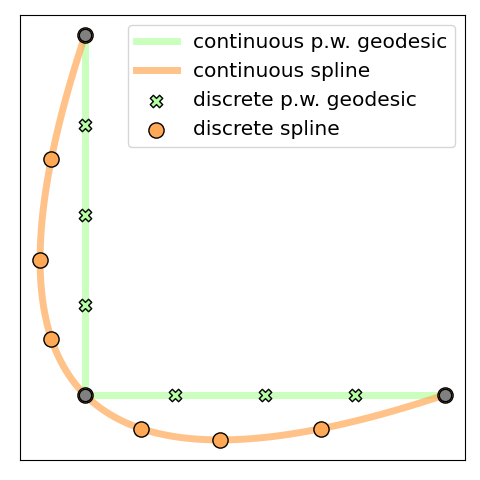}};
\end{scope}

\begin{scope}[shift={(0,-3.8)},scale=1.0]
\edef\currentCnt{0}
\foreach \i in  {0,1,2,3,4,5,6,7,8} {
\pgfmathparse{int(\i*2)}
\xdef\ii{\pgfmathresult}
\ifthenelse{0 = \i \OR 8 = \i \OR 4 = \i}{
\node[frame,anchor=south west] at (\currentCnt+0.08,0) {\includegraphics[width=0.15\linewidth]{u_\i_wd_cropped.png}};
}
{
\node[anchor=south west] at (\currentCnt,0-0.08) {\includegraphics[width=0.15\linewidth]{u_\i_wd_cropped.png}};
}

\pgfmathparse{\currentCnt+1.9}
\xdef\currentCnt{\pgfmathresult}
}
\edef\currentCnt{0}

\node[rotate=90] at (-0.2,0.9) {\Large{spline}};
\end{scope}
\begin{scope}[shift={(0,-5.7)},scale=1.0]
\edef\currentCnt{0}
\foreach \i in  {0,1,2,3,4,5,6,7,8} {
\pgfmathparse{int(\i*2)}
\xdef\ii{\pgfmathresult}
\ifthenelse{0 = \i \OR 8 = \i \OR 4 = \i}{
\node[frame,anchor=south west] at (1.*\currentCnt+0.08,0) {\includegraphics[width=0.15\linewidth]{diff_\ii.png}};}
{
\node[anchor=south west] at (1.*\currentCnt,0-0.08) {\includegraphics[width=0.15\linewidth]{diff_\ii.png}};
}
\pgfmathparse{\currentCnt+1.9}
\xdef\currentCnt{\pgfmathresult}
}
\edef\currentCnt{0}

\node[anchor=south west] at (17.1,0.)
{\includegraphics[width=0.30\linewidth]{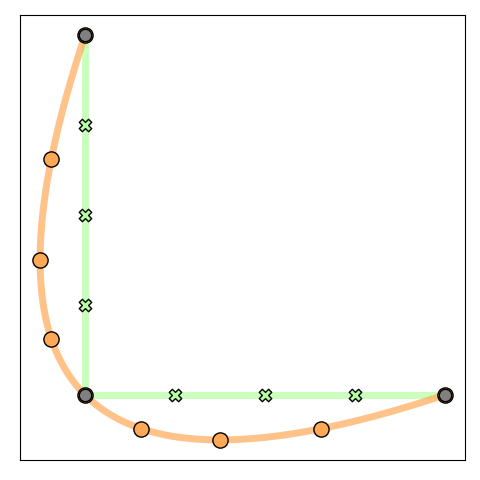}};
\end{scope}
\begin{scope}
\edef\currentCnt{0}
\end{scope}
\end{scope}
\end{tikzpicture}
}
\caption{First two rows: A comparison of discrete piecewise geodesic interpolation (first row) and discrete spline interpolation (second row) for $\delta=0$ is shown for key frame distributions framed in red. The optimization was done on $\ProbWGd$. Third row: Same as second row, except the optimization was performed in the full space $\ProbW(\R^d)$. Bottom row: Difference between spline and piece-wise geodesic interpolations. Top right: Plot of 
the center of masses as a polygonal curve in $\R^2$. Bottom right: Plot of the standard deviations as a polygonal curve in $\R^2$.}
\label{fig:GaussianRestricted}
\end{figure*}
\begin{figure*}[htb]
\resizebox{\linewidth}{!}{
\tikzstyle{frame} = [line width=1.8pt, draw=red,inner sep=0.01em]
\begin{tikzpicture}
\begin{scope}[scale=1.4]
\begin{scope}
\edef\currentCnt{0}
\foreach \i in  {0,1,2,3,4,5,6,7,8} {
\ifthenelse{0 = \i \OR 8 = \i \OR 4 = \i}{
\node[frame,anchor=south west] at (\currentCnt+0.08,0) {\includegraphics[width=0.15\linewidth]{u_\i_cag.png}};}
{
\node[anchor=south west] at (\currentCnt,0-0.08) {\includegraphics[width=0.15\linewidth]{u_\i_cag.png}};
}
\pgfmathparse{\currentCnt+1.9}
\xdef\currentCnt{\pgfmathresult}
}
\end{scope}

\begin{scope}[shift={(0,-1.9)}]
\edef\currentCnt{0}
\foreach \i in  {0,1,2,3,4,5,6,7,8} {
\ifthenelse{0 = \i \OR 8 = \i \OR 4 = \i}{
\node[frame,anchor=south west] at (\currentCnt+0.08,0.) {\includegraphics[width=0.15\linewidth]{u_\i_cas.png}};
}
{
\node[anchor=south west] at (\currentCnt,-0.08) {\includegraphics[width=0.15\linewidth]{u_\i_cas.png}};
}

\pgfmathparse{\currentCnt+1.9}
\xdef\currentCnt{\pgfmathresult}
}
\edef\currentCnt{0}

\node[rotate=90] at (-0.2,0.9) {\Large{spline}};
\node[rotate=90] at (-0.2,2.8) {\Large{p.w. geodesic}};
\end{scope}

\begin{scope}[shift={(0,-3.05)}]
\edef\currentCnt{0}
\node[anchor=south west] at (0.5,-1.47) {\includegraphics[width=1.28\linewidth]{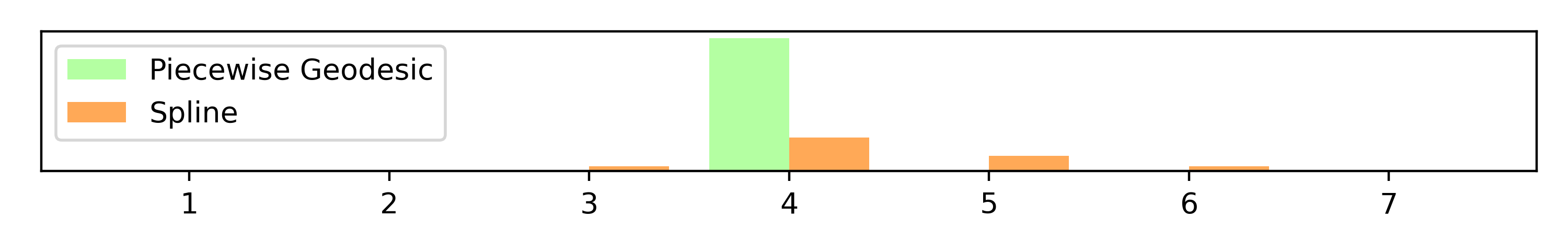}};

\end{scope}

\begin{scope}
\edef\currentCnt{0}
\end{scope}
\end{scope}
\end{tikzpicture}
}
\caption{Discrete piecewise geodesic interpolation (top) and discrete spline interpolation for $\delta=0$ (middle)
of three key frames with constant density on an annulus for the first and
constant density on a disk for the second and third (framed in red). Bottom: Contribution of each time-step $k=1,\ldots,K-1$ to the spline energy, i.e. $\Wass^2(\mu_k,\Bary(\mu_{k-1},\mu_{k+1}))$ for the spline interpolation (orange) and piecewise geodesic interpolation (green).}
\label{fig:Circle_Annulus}
\end{figure*}
\begin{figure*}[htb]
\resizebox{\linewidth}{!}{
\tikzstyle{frame} = [line width=1.8pt, draw=red,inner sep=0.01em]
\begin{tikzpicture}
\begin{scope}[scale=1.4]
\begin{scope}
\edef\currentCnt{0}
\foreach \i in  {0,1,2,3,4,5,6,7,8} {
\ifthenelse{0 = \i \OR 8 = \i \OR 4 = \i}{
\node[frame,anchor=south west] at (\currentCnt+0.08,0) {\includegraphics[width=0.15\linewidth]{u_\i_csg.png}};}
{
\node[anchor=south west] at (\currentCnt,0-0.08) {\includegraphics[width=0.15\linewidth]{u_\i_csg.png}};
}
\pgfmathparse{\currentCnt+1.9}
\xdef\currentCnt{\pgfmathresult}
}
\end{scope}

\begin{scope}[shift={(0,-1.9)}]
\edef\currentCnt{0}
\foreach \i in  {0,1,2,3,4,5,6,7,8} {
\ifthenelse{0 = \i \OR 8 = \i \OR 4 = \i}{
\node[frame,anchor=south west] at (\currentCnt+0.08,0.) {\includegraphics[width=0.15\linewidth]{u_\i_css.png}};
}
{
\node[anchor=south west] at (\currentCnt,0.-0.08) {\includegraphics[width=0.15\linewidth]{u_\i_css.png}};
}

\pgfmathparse{\currentCnt+1.9}
\xdef\currentCnt{\pgfmathresult}
}
\edef\currentCnt{0}

\node[rotate=90] at (-0.2,0.9) {\large{spline}};
\node[rotate=90] at (-0.2,2.8) {\large{p.w. geodesic}};
\node[rotate=90] at (-0.2,-0.85) {\large{difference}};
\end{scope}
\begin{scope}[shift={(0,-3.85)}]
\edef\currentCnt{0}
\foreach \i in  {0,1,2,3,4,5,6,7,8} {
\ifthenelse{\NOT 0 = \i }{
\node[anchor=south west] at (\currentCnt,0.) {\includegraphics[width=0.15\linewidth]{u_\i_diff.png}};
}
{
\node[anchor=south west] at (\currentCnt,0.) {\includegraphics[width=0.15\linewidth]{u_\i_diff.png}};
}

\pgfmathparse{\currentCnt+1.9}
\xdef\currentCnt{\pgfmathresult}
}
\end{scope}
\begin{scope}
\edef\currentCnt{0}
\end{scope}
\end{scope}
\end{tikzpicture}
}
\caption{Piecewise geodesic (top) and spline interpolation (middle) are shown for key frames (framed in red) consisting of a thin annulus-shaped distribution and two equal thin square-shaped distributions, using the color map ${0}\hspace{1mm}$\protect\resizebox{.08\linewidth}{!}{\protect\includegraphics{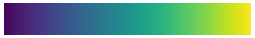}}$\hspace{1mm}4\mathrm{e}{-4}$. Bottom: Difference between the spline and piecewise geodesic interpolations, using the color map ${-6}\mathrm{e}{-5}\hspace{1mm}$\protect\resizebox{.08\linewidth}{!}{\protect\includegraphics{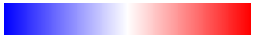}}$\hspace{1mm}6\mathrm{e}{-5}$.}
\label{fig:circle_square}
\end{figure*}
\begin{figure*}[htb]
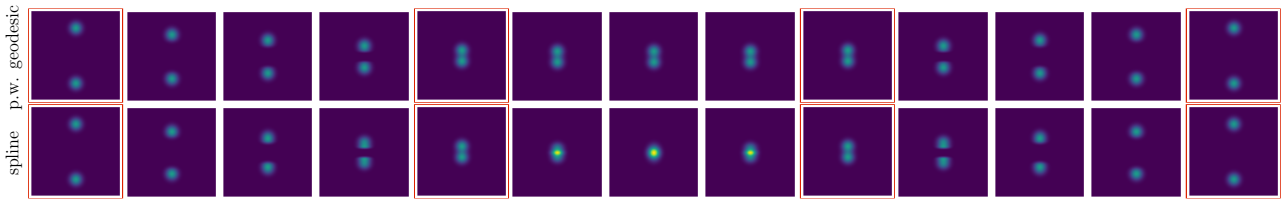

\resizebox{\linewidth}{!}{
\tikzstyle{frame} = [line width=1.8pt, draw=red,inner sep=0.01em]
\begin{tikzpicture}
\begin{scope}[scale=1.4]
\begin{scope}
\edef\currentCnt{0}
\foreach \i in  {0,1,2,3,4,5,6,7,8,9,10,11,12} {
\ifthenelse{0 = \i \OR 8 = \i \OR 4 = \i \OR 12 = \i}{
\node[frame,anchor=south west] at (\currentCnt+0.08,0) {\includegraphics[width=0.15\linewidth]{u_\i_wdg.png}};}
{
\node[anchor=south west] at (\currentCnt,0-0.08) {\includegraphics[width=0.15\linewidth]{u_\i_wdg.png}};
}
\pgfmathparse{\currentCnt+1.9}
\xdef\currentCnt{\pgfmathresult}
}
\end{scope}

\begin{scope}[shift={(0,-1.9)}]
\edef\currentCnt{0}
\foreach \i in  {0,1,2,3,4,5,6,7,8,9,10,11,12} {
\ifthenelse{0 = \i \OR 8 = \i \OR 4 = \i \OR 12 = \i}{
\node[frame,anchor=south west] at (\currentCnt+0.08,0.) {\includegraphics[width=0.15\linewidth]{u_\i_wd.png}};
}
{
\node[anchor=south west] at (\currentCnt,0.-0.08) {\includegraphics[width=0.15\linewidth]{u_\i_wd.png}};
}

\pgfmathparse{\currentCnt+1.9}
\xdef\currentCnt{\pgfmathresult}
}
\edef\currentCnt{0}

\node[rotate=90] at (-0.2,0.9) {\Large{spline}};
\node[rotate=90] at (-0.2,2.8) {\Large{p.w. geodesic}};
\end{scope}

\begin{scope}
\edef\currentCnt{0}
\end{scope}
\end{scope}
\end{tikzpicture}
}
\caption{The key frames represent two Gaussians that are far apart from each other (first and fourth key frames) and close to each other (second and third key frames). Piecewise geodesic (top) and spline (bottom) interpolations are shown.}
\label{fig:Two_gaussians}
\end{figure*}
\begin{figure*}[htb]
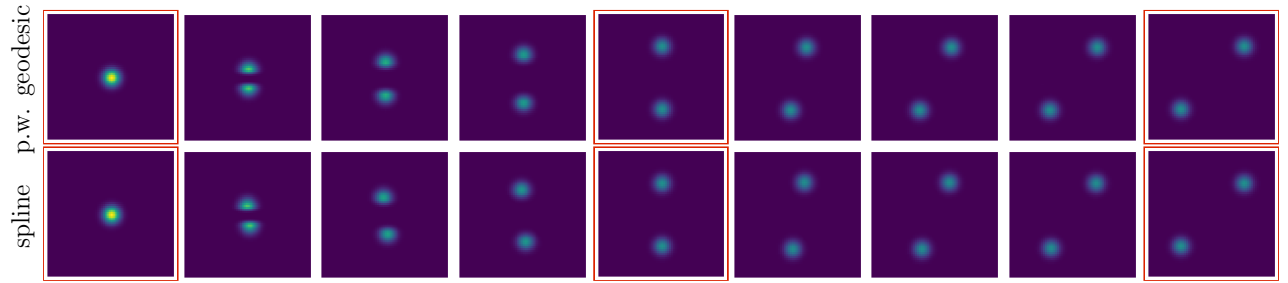

\resizebox{\linewidth}{!}{
\tikzstyle{frame} = [line width=1.8pt, draw=red,inner sep=0.01em]
\begin{tikzpicture}
\begin{scope}[scale=1.4]
\begin{scope}
\edef\currentCnt{0}
\foreach \i in  {0,1,2,3,4,5,6,7,8} {
\ifthenelse{0 = \i \OR 8 = \i \OR 4 = \i}{
\node[frame,anchor=south west] at (\currentCnt+0.08,0) {\includegraphics[width=0.15\linewidth]{u_\i_gsg.png}};}
{
\node[anchor=south west] at (\currentCnt,0-0.08) {\includegraphics[width=0.15\linewidth]{u_\i_gsg.png}};
}
\pgfmathparse{\currentCnt+1.9}
\xdef\currentCnt{\pgfmathresult}
}
\end{scope}

\begin{scope}[shift={(0,-1.9)}]
\edef\currentCnt{0}
\foreach \i in  {0,1,2,3,4,5,6,7,8} {
\ifthenelse{0 = \i \OR 8 = \i \OR 4 = \i}{
\node[frame,anchor=south west] at (\currentCnt+0.08,0.) {\includegraphics[width=0.15\linewidth]{u_\i_gss.png}};
}
{
\node[anchor=south west] at (\currentCnt,0.-0.08) {\includegraphics[width=0.15\linewidth]{u_\i_gss.png}};
}

\pgfmathparse{\currentCnt+1.9}
\xdef\currentCnt{\pgfmathresult}
}
\edef\currentCnt{0}

\node[rotate=90] at (-0.2,2.8) {\Large{p.w. geodesic}};
\node[rotate=90] at (-0.2,0.9) {\Large{spline}};
\end{scope}

\begin{scope}
\edef\currentCnt{0}
\end{scope}
\end{scope}
\end{tikzpicture}
}
\caption{From left to right the key frames represent a single Gaussian, a pair of vertically displaced Gaussian of half the mass, and the vertically displaced configuration rotated by $-\tfrac\pi4$. Discrete piecewise geodesic (top) and spline (bottom) interpolations are shown.}
\label{fig:Gaussian_break}
\end{figure*}
\begin{figure*}[htb]
\resizebox{\linewidth}{!}{
\tikzstyle{frame} = [line width=1.8pt, draw=red,inner sep=0.01em]
\begin{tikzpicture}
\begin{scope}[scale=1.]
\begin{scope}
\node[anchor=south west] at (0.,0.) {\includegraphics[width=1.\linewidth]{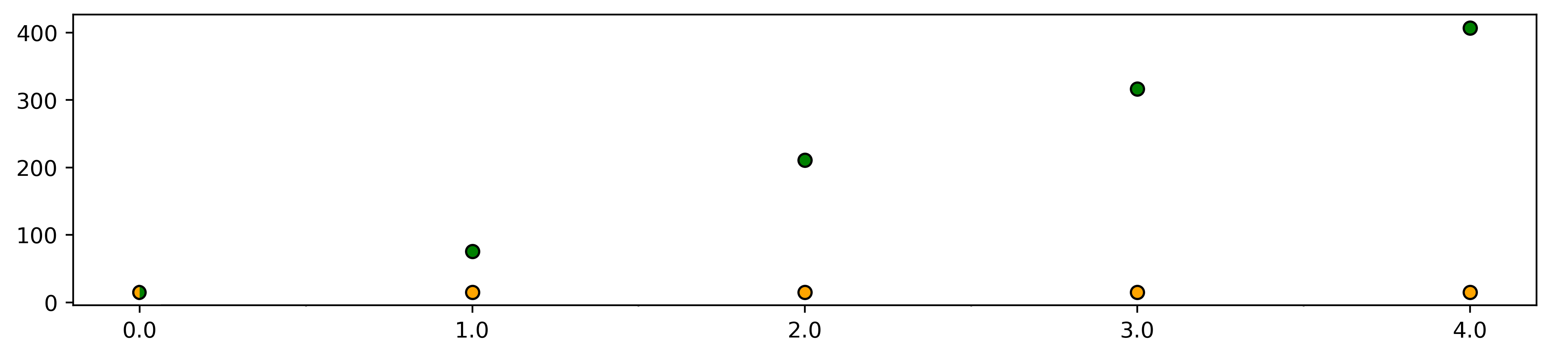}};
\node[rotate=0] at (0.,2.1) {$t$};
\node[rotate=0] at (17., 0.3) {$x$};
\end{scope}

\end{scope}
\end{tikzpicture}
}
\caption{Time $t$ (in seconds) until convergence of the fully discrete spline interpolation problem is reached, for a series of five interpolation problems $P(x)$ depending on parameter $x\in\{0,1,2,3,4\}$ ($x$-axis). The interpolation problem $P(x)$ is defined as follows: The prescribed times are $\overline t_0=0$, $\overline t_1=0.5$ and $\overline t_2=1$, and the prescribed probability measures are given by $\overline\mu_0=\mathcal{N}((0,0),\sigma_0^2)$, $\overline\mu_1=\mathcal{N}((x,x),\sigma_1^2)$, and $\overline\mu_2=\mathcal{N}((0,0),\sigma_2^2)$, for diagonal standard deviation matrices $\sigma_0=\diag(1,2)$, $\sigma_1=\diag(1,1)$ and $\sigma_2=\diag(2,1)$. Dots denote the computation time of the algorithm solving problem $P(x)$, both with implementation of the decoupling of the means as described by equation \eqref{eq:spline-energy-decoupling} (orange), and without it (green).}
\label{fig:decoupling}
\end{figure*}
\begin{figure*}[htb]
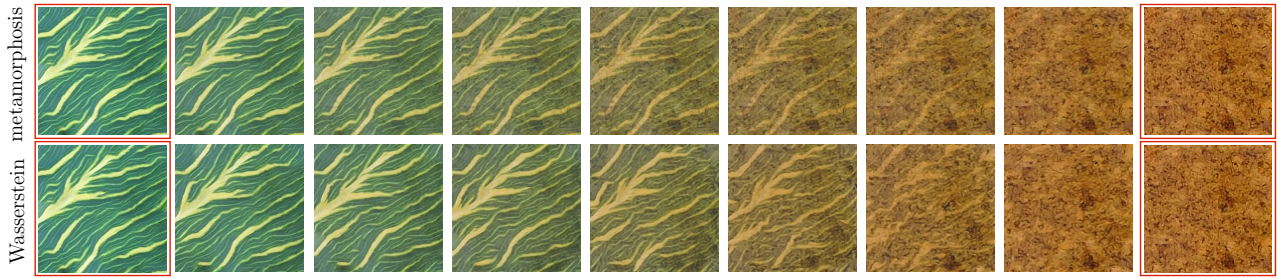

\resizebox{\linewidth}{!}{
\tikzstyle{frame} = [line width=0.8pt, draw=red,inner sep=0.2em]
\begin{tikzpicture}
\begin{scope}[scale=1.45]
\begin{scope}
\edef\currentCnt{0}
\foreach \i in  {0,1,2,3,4,5,6,7,8} {
\ifthenelse{0 = \i \OR 8 = \i}{
\node[frame,anchor=south west] at (\currentCnt+0.04,0) {\includegraphics[width=0.15\linewidth]{u_\i.png}};}
{
\node[anchor=south west] at (\currentCnt,-0.022) 
{\includegraphics[width=0.15\linewidth]{u_\i.png}};
}
\pgfmathparse{\currentCnt+1.9}
\xdef\currentCnt{\pgfmathresult}
}
\end{scope}

\begin{scope}[shift={(0,-1.9)}]
\edef\currentCnt{0}
\foreach \i in  {0,1,2,3,4,5,6,7,8} {
\ifthenelse{0 = \i \OR 8 = \i}{
\node[frame,anchor=south west] at (\currentCnt+0.04,0) {\includegraphics[width=0.15\linewidth]{u_\i.png}};}
{
\node[anchor=south west] at (\currentCnt,-0.022) 
{\includegraphics[width=0.15\linewidth]{synthesized_lc\i.png}};
}
\pgfmathparse{\currentCnt+1.9}
\xdef\currentCnt{\pgfmathresult}
}
\node[rotate=90] at (-0.2,2.8) {\large{metamorphosis}};
\node[rotate=90] at (-0.2,0.9) {\large{Wasserstein}};
\end{scope}
\begin{scope}[shift={(0,-3.8)}]
\end{scope}

\begin{scope}
\edef\currentCnt{0}
\end{scope}
\end{scope}
\end{tikzpicture}
}
\caption{Time discrete metamorphosis (top) and Wasserstein spline interpolations (bottom) with framed prescribed images/feature distributions for $K=16$. For the metamorphosis spline the key frames are chosen identical to the synthesized texture of the Wasserstein splines. Due to symmetry, only the first half of the interpolations is shown.}
\label{fig:texture}
\end{figure*}

\begin{figure*}[t]
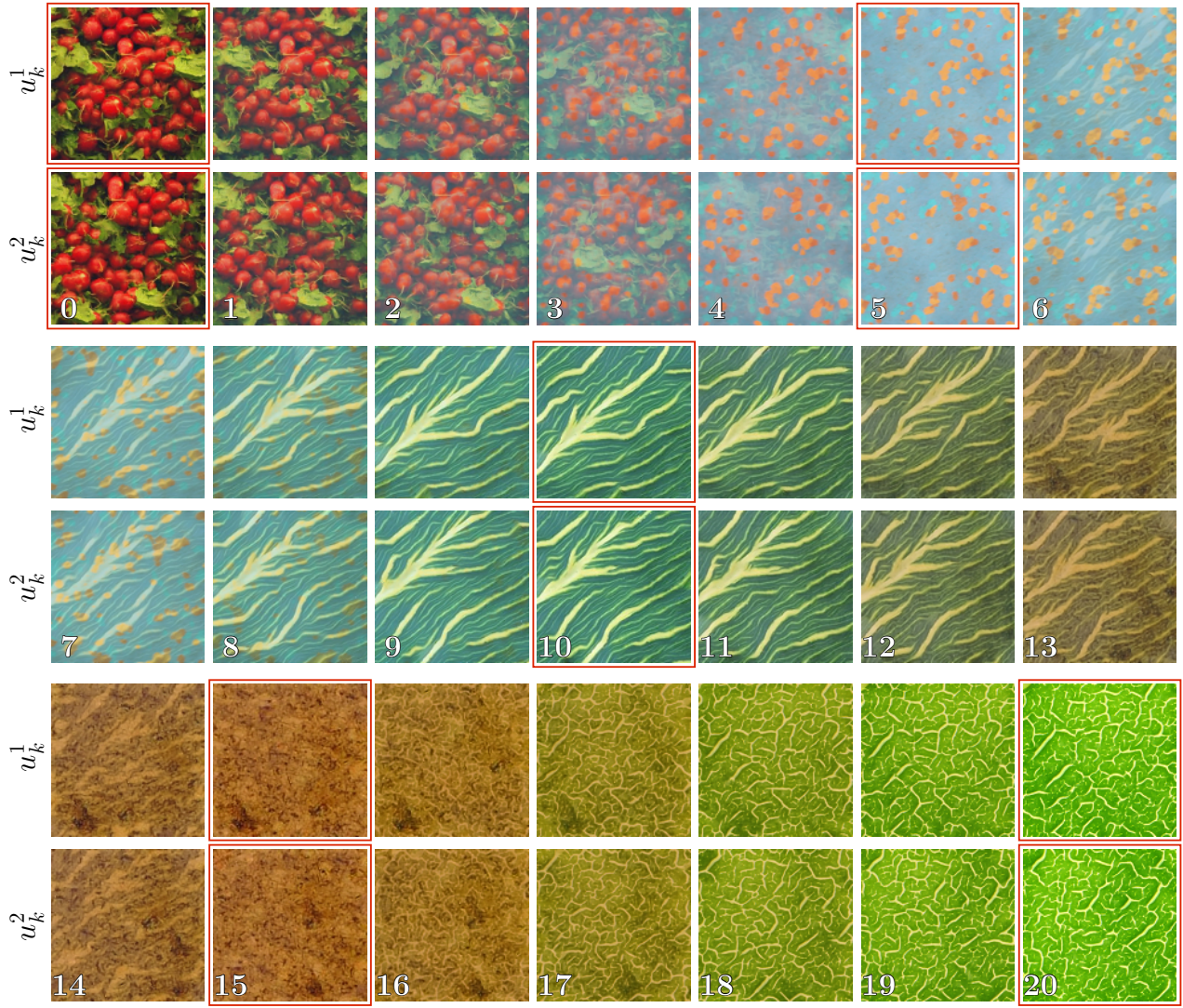

\resizebox{\linewidth}{!}{
\tikzstyle{frame} = [line width=0.8pt, draw=red,inner sep=0.2em]
\begin{tikzpicture}
\begin{scope}[scale=1.4]
\edef\currentCnt{0}
\foreach \i in  {0,1,2,3,4,5,6} {
\ifthenelse{0 = \i \OR 5 = \i}{
\node[frame,anchor=south west] at (\currentCnt+0.03,0) {\includegraphics[width=0.1478\linewidth]{synthesized_\i.png}};}
{
\node[anchor=south west] at (\currentCnt,0-0.03) {\includegraphics[width=0.1478\linewidth]{synthesized_\i.png}};
}
\ifthenelse{0 = \i \OR 5 = \i}{
\node[frame,anchor=south west] at (\currentCnt+0.03,-1.95) {\includegraphics[width=0.1478\linewidth]{synthesized2_\i.png}};}
{
\node[anchor=south west] at (\currentCnt,-1.95-0.03) {\includegraphics[width=0.1478\linewidth]{synthesized2_\i.png}};
}
\node at (\currentCnt+0.25,-1.7) {\color{white}{\contour{black}{ \Large{$\mathbf{\i}$}}}};
\pgfmathparse{\currentCnt+1.9}
\xdef\currentCnt{\pgfmathresult}
}
\node[rotate=90] at (-0.2,1) {\Large{$u^1_k$}};
\node[rotate=90] at (-0.2,-1.0) {\Large{$u^2_k$}};

\begin{scope}[shift={(0,-4.0)}]
\edef\currentCnt{0}
\foreach \j in  {7,8,9,10,11,12,13} {
\pgfmathparse{int(\j-7)}
\edef\i{\pgfmathresult}
\ifthenelse{10 = \j}{
\node[frame,anchor=south west] at (\currentCnt+0.03,0) {\includegraphics[width=0.1478\linewidth]{synthesized_\j.png}};}
{
\node[anchor=south west] at (\currentCnt,0-0.03) {\includegraphics[width=0.1478\linewidth]{synthesized_\j.png}};
}
\ifthenelse{10 = \j}{
\node[frame,anchor=south west] at (\currentCnt+0.03,-1.95) {\includegraphics[width=0.1478\linewidth]{synthesized2_\j.png}};}
{
\node[anchor=south west] at (\currentCnt,-1.95-0.03) {\includegraphics[width=0.1478\linewidth]{synthesized2_\j.png}};
}
\node at (\currentCnt+0.25,-1.7) {\color{white}{\contour{black}{ \Large{$\mathbf{\j}$}}}};
\pgfmathparse{\currentCnt+1.9}
\xdef\currentCnt{\pgfmathresult}
}
\node[rotate=90] at (-0.2,1) {\Large{$u^1_k$}};
\node[rotate=90] at (-0.2,-1.0) {\Large{$u^2_k$}};
\end{scope}

\begin{scope}[shift={(0,-8.0)}]
\edef\currentCnt{0}
\foreach \j in  {14,15,16,17,18,19,20} {
\pgfmathparse{int(\j-14)}
\edef\i{\pgfmathresult}
\ifthenelse{20 = \j \OR 15 = \j}{
\node[frame,anchor=south west] at (\currentCnt+0.03,0) {\includegraphics[width=0.1478\linewidth]{synthesized_\j.png}};}
{
\node[anchor=south west] at (\currentCnt,0-0.03) {\includegraphics[width=0.1478\linewidth]{synthesized_\j.png}};
}
\ifthenelse{20 = \j \OR 15 = \j}{
\node[frame,anchor=south west] at (\currentCnt+0.03,-1.95) {\includegraphics[width=0.1478\linewidth]{synthesized2_\j.png}};}
{
\node[anchor=south west] at (\currentCnt,-1.95-0.03) {\includegraphics[width=0.1478\linewidth]{synthesized2_\j.png}};
}
\node at (\currentCnt+0.25,-1.7) {\color{white}{\contour{black}{ \Large{$\mathbf{\j}$}}}};
\pgfmathparse{\currentCnt+1.9}
\xdef\currentCnt{\pgfmathresult}
}
\node[rotate=90] at (-0.2,1) {\Large{$u^1_k$}};
\node[rotate=90] at (-0.2,-1.0) {\Large{$u^2_k$}};
\end{scope}
\end{scope}
\end{tikzpicture}
}
\caption{Two realizations of a texture spline for different starting latent space samples, top and bottom of each panel respectively with parameters $K=20$, $\delta=0.01$. Let us remark that not only the actual spline interpolated textures but also the key frame textures 
differ as they are all different samples of the underlying spline probability distributions $\nu^K_k$.
}
\label{fig:texture_3}
\end{figure*}

\begin{figure*}[t]
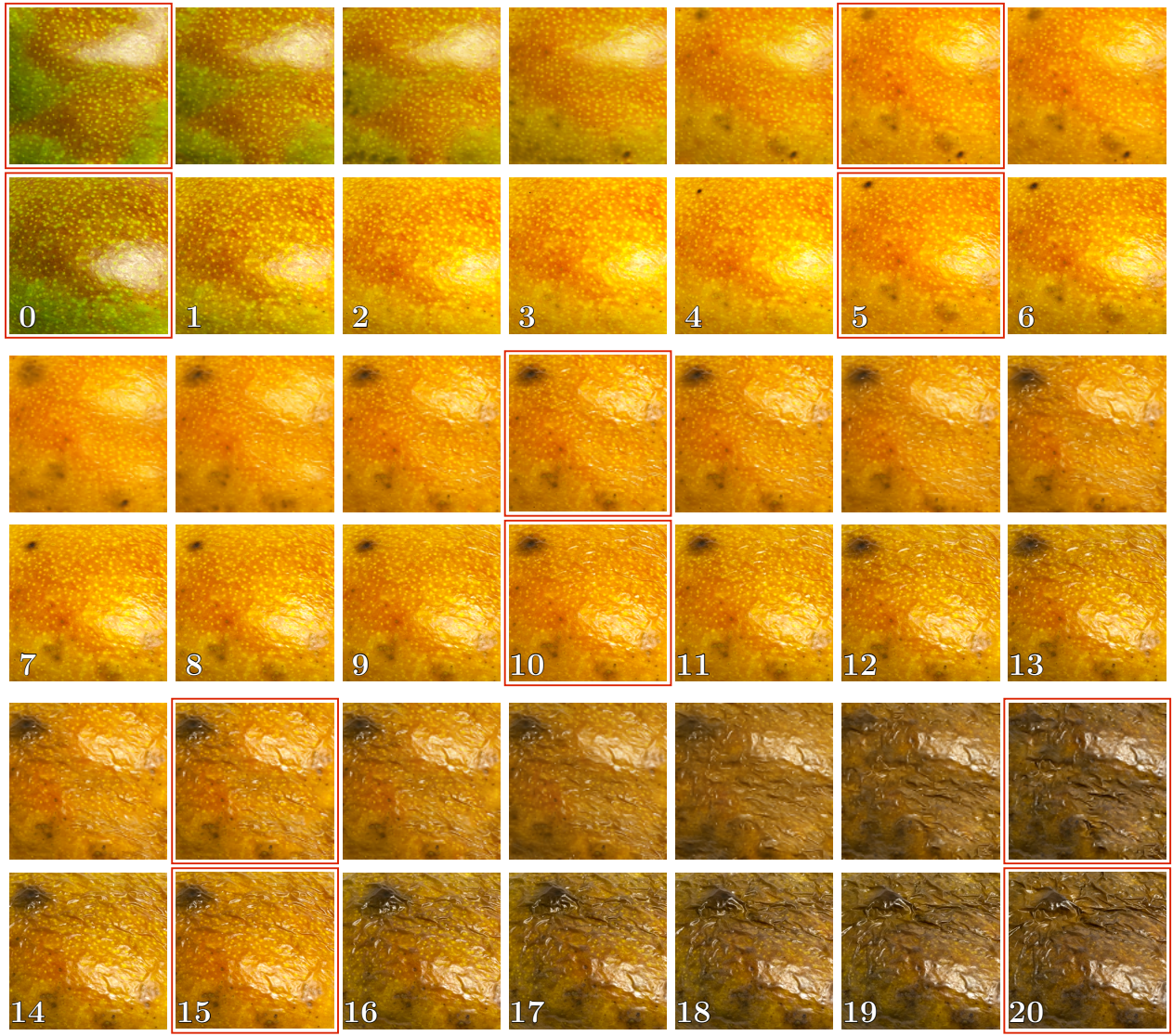

\resizebox{\linewidth}{!}{
\tikzstyle{frame} = [line width=0.8pt, draw=red,inner sep=0.2em]
\begin{tikzpicture}
\begin{scope}[scale=1.4]
\edef\currentCnt{0}
\foreach \i in  {0,1,2,3,4,5,6} {
\ifthenelse{0 = \i \OR 5 = \i}{
\node[frame,anchor=south west] at (\currentCnt+0.03,0) {\includegraphics[width=0.1478\linewidth]{mangoes5_\i.png}};}
{
\node[anchor=south west] at (\currentCnt,0-0.03) {\includegraphics[width=0.1478\linewidth]{mangoes5_\i.png}};
}
\ifthenelse{0 = \i \OR 5 = \i}{
\node[frame,anchor=south west] at (\currentCnt+0.03,-1.95) {\includegraphics[width=0.1478\linewidth]{mango\i.png}};}
{
\node[anchor=south west] at (\currentCnt,-1.95-0.03) {\includegraphics[width=0.1478\linewidth]{mango\i.png}};
}
\node at (\currentCnt+0.25,-1.7) {\color{white}{\contour{black}{ \Large{$\mathbf{\i}$}}}};
\pgfmathparse{\currentCnt+1.9}
\xdef\currentCnt{\pgfmathresult}
}
\begin{scope}[shift={(0,-4.)}]
\edef\currentCnt{0}
\foreach \j in  {7,8,9,10,11,12,13} {
\pgfmathparse{int(\j-7)}
\edef\i{\pgfmathresult}
\ifthenelse{10 = \j}{
\node[frame,anchor=south west] at (\currentCnt+0.03,0) {\includegraphics[width=0.1478\linewidth]{mangoes5_\j.png}};}
{
\node[anchor=south west] at (\currentCnt,0-0.03) {\includegraphics[width=0.1478\linewidth]{mangoes5_\j.png}};
}
\ifthenelse{10 = \j}{
\node[frame,anchor=south west] at (\currentCnt+0.03,-1.95) {\includegraphics[width=0.1478\linewidth]{mango\j.png}};}
{
\node[anchor=south west] at (\currentCnt,-1.95-0.03) {\includegraphics[width=0.1478\linewidth]{mango\j.png}};
}
\node at (\currentCnt+0.25,-1.7) {\color{white}{\contour{black}{ \Large{$\mathbf{\j}$}}}};
\pgfmathparse{\currentCnt+1.9}
\xdef\currentCnt{\pgfmathresult}
}
\end{scope}

\begin{scope}[shift={(0,-8.)}]
\edef\currentCnt{0}
\foreach \j in  {14,15,16,17,18,19,20} {
\pgfmathparse{int(\j-14)}
\edef\i{\pgfmathresult}
\ifthenelse{20 = \j \OR 15 = \j}{
\node[frame,anchor=south west] at (\currentCnt+0.03,0) {\includegraphics[width=0.1478\linewidth]{mangoes5_\j.png}};}
{
\node[anchor=south west] at (\currentCnt,0-0.03) {\includegraphics[width=0.1478\linewidth]{mangoes5_\j.png}};
}
\ifthenelse{20 = \j \OR 15 = \j}{
\node[frame,anchor=south west] at (\currentCnt+0.03,-1.95) {\includegraphics[width=0.1478\linewidth]{mango\j.png}};}
{
\node[anchor=south west] at (\currentCnt,-1.95-0.03) {\includegraphics[width=0.1478\linewidth]{mango\j.png}};
}
\node at (\currentCnt+0.25,-1.7) {\color{white}{\contour{black}{ \Large{$\mathbf{\j}$}}}};
\pgfmathparse{\currentCnt+1.9}
\xdef\currentCnt{\pgfmathresult}
}

\end{scope}
\end{scope}
\end{tikzpicture}
}
\caption{Top: A realization of a texture spline with parameters $K=20$, $\delta=0.01$. Bottom: Textures sampled at the interpolated times of the actual video from which the texture constraints were extracted from. 
As in Figure~\ref{fig:texture_3} 
the key frame textures from the spline interpolated path differ from the true images at the corresponding times
as they are samples of the underlying spline probability distributions $\nu^K_k$.}
\label{fig:texture_2}
\end{figure*}

\section{Generative texture synthesis based on spline interpolation of feature distributions}\label{sec:textures}

The flexibility of our model will be tested in this section to generate spline interpolations in the space of textures.
Recently, Houdard \etal \cite{Ho20} proposed GOTEX, a generative model for texture synthesis from a single sample image.
There, the parameters of the generator are chosen such that the distribution of features extracted from the generated textures  
is close in Wasserstein distance to the corresponding empirical feature distribution for the given sample image. In what follows, we shall outline how we leverage our spline interpolation model within the GOTEX framework. 
To this end, we proceed as follows:
\begin{itemize}
\item[-] First, for a vector of feature maps, we compute empirical feature distributions $\overline\nu_i \in \Prob(\R^d)$ 
for all input images $\overline{u}_i$ at times $t_i$ for $i=1,\ldots,I$ and some $d\in \N$.
\item[-] Next, we use the discrete spline approach presented in the preceeding sections to compute a 
discrete spline interpolation $(\nu_k^K)_{k=0,\ldots,K}$ for prescribed distributions $\overline{\nu}_i$ at times $k_i = \overline{t}_i K$ for 
$k_i \in \{0,\ldots, K\}$.
\item[-] Finally, in a post processing, we train a generative texture model to obtain image representations $u_k = g_{\theta_k}(z)$ for the computed  probability distributions $\nu_k^K$, where $\theta_k$ is a set of optimal parameters of a generative neural network $g_{\theta_k}$ applied to a sample $z$ of a regular distribution. 
\end{itemize}
\paragraph{\textit{Extracting empirical feature distributions.}}
Let $I\geq 2$ be fixed, and consider a vector $F=(F_m)_{m=1,\ldots, M}$ of $d$-dimensional 
local feature maps $F_m:\R^N\rightarrow\R^d$ defined on images with $N$ pixels. Each component $F_m$ operates on small pixel neighbourhoods (patches).
Given the images $\overline{u}_i$, $i=1,\ldots,I$ to be spline interpolated the associated empirical feature distributions are 
$$\overline\nu_i\coloneqq  \tfrac1M \sum_{m=1,\ldots, M} \delta_{F_m[\overline u_i]}$$ 
in $\Prob(\R^d)$ for $i=1,\ldots,I$, with $\delta_x$ being the Dirac measure at $x$ in $\R^d$.

\paragraph{\textit{Computing discrete splines in the space of feature distributions.}}
Given the set of feature distributions $\overline\nu_i$ with $i=1,\ldots,I,$ obtained from the first step with associated interpolation times $0\leq\overline t_0<\ldots,\overline t_I\leq 1$ and some $K\in \N$, we compute a discrete spline interpolation $(\nu_k^K)_{k=0,\ldots,K}$ of the prescribed feature distributions $\overline{\nu}_i$ at times $\overline{t}_i$ in $\ProbW(\R^d)$ as described in the previous sections. Here, we constrain the discrete spline to lie in the space of feature distributions, i.e. each distribution $\nu_k^K$ must be represented as the sum of $M$ delta distributions in $\R^d$ with equal weights. To this end, we minimize the fully discrete spline energy functional \eqref{eq:fully_discrete_spline_funct} with respect to the locations $x_m^k$, where $$\nu_k^K=\frac{1}{M}\sum_{m=1}^M\delta_{x_m^k},$$ and keep the distributions 
$\overline{\nu}_i$ at times $k_i = \overline{t}_i K$ for $k_i \in \{0,\ldots, K\}$ fixed.

\paragraph{\textit{Retrieving image representations via a generative texture model based on neural networks.}}
Different samples of a synthesized texture for a given feature distribution $\nu_k^K$ (obtained in the previous step) are regarded as samples of a probability distribution, which is defined as the push-forward of a fixed distribution $\zeta$ defined on a latent space $\mathcal{Z}$, with a generator $g_\theta:\mathcal{Z} \to \R^N$ with parameter vector $\theta$ in a set of admissible parameters $\Theta$. Typically, one may assume $\zeta$ to be the uniform distribution on the space $\mathcal{Z}=[0,1]^N$. Now, one assumes $g$ to be a feed-forward neural network that has been pre-trained on a number of textures.
Hence, one is looking for an optimal parameter vectors $\theta_k \in \Theta$, which minimizes the Wasserstein distance $\Wass(\mu_{\theta_k},\nu_k^K)$ of the resulting feature distribution 
$$\mu_{\theta_k} \coloneqq \tfrac1M \sum_{m=1,\ldots, M} (F_m\circ g_{\theta_k})_\#\zeta$$ 
from the given feature distribution $\nu_k^K$ of the discrete spline. 

The minimization of $\Wass(\mu_{\theta_k},\nu_k^K)$ with respect to the parameter vector $\theta_k$ can be numerically realized via a stochastic gradient descent approach. After obtaining the optimal parameter vectors $\theta_k$ for $k=0,\ldots,K$, one then samples $z\sim\zeta$, and generates the resulting texture spline interpolation $(u_k)_{k=0,\ldots,K}$ as a set of images $$u_k\coloneqq g_{\theta_k}(z)$$
for $k=0,\ldots,K$. 
For different samples $z$ one obtains different images $(u_k)_{k=0,\ldots,K}$ representing the spline interpolation of the textures.

The multi-scale architecture of the generator network $g$ is made up of chains of convolutional, non-linear activation and upsampling layers that take a noise sample $z$ as an input and terminate by producing the final image, cf. \cite{Ul16}. Each convolution block in the generator network contains three convolutional layers followed by a non-linear ReLU activation layer. The convolutional layers contain $3\times 3, 3\times 3$ and $1\times 1$ filters, respectively. Next, nearest-neighbour interpolation is used in the upsampling layers to obtain a tensor with the desired full resolution. For the last step, this tensor is mapped to an RGB-image by a batch of $1\times 1$ filters.

\paragraph{\textit{Numerical results.}}
In Figure~\ref{fig:texture} we compare our spline interpolation method with the alternative metamorphosis spline interpolation (middle) as described in \cite{JuRaRu23} with the same prescribed feature distributions/frames in each case. The key frames consist of a close-up of a leaf (at times $\overline t_1=0, \overline t_3=1$), and a close-up of a cork (at $t_2=\tfrac12$). Clearly, the leaf creases inherited from the first and last key frames are simply blended out in the first method. On the other hand, our approach ensures that the features are interpolated smoothly: Both the boundaries and the surface area covered by the creases change smoothly over time. 

Fig.~\ref{fig:texture_3} shows discrete texture curves resulting 
from a discrete spline interpolation $(\nu_k^K)_k$ for $K=20$ in the space of feature distributions as described above. The texture samples which have been synthesized from the prescribed feature distributions have been framed in red, and the prescribed times are given by $\overline t_i=i/4$, $i=0,\ldots,4$.  
Two different texture realizations $u^1_k\coloneqq g_{\theta_k}(z_1)$, $u^2_k\coloneqq g_{\theta_k}(z_2)$ for $z_1,\, z_2 \sim \zeta$ 
are shown on top of each other for a normalized random distribution $\zeta$ and 
$\theta_k$ minimizing the entropy regularized Wasserstein distance between $\mu_{\theta_k}$ and $\nu_k^K$. The weights of the neural network used to generate these textures are kept unchanged between both interpolations (and along each interpolation between different time frames). Hence, even though for a fixed time step $k$ the spatial arrangement of the texture pattern varies substantially between both samples, it becomes apparent that the texture characteristics 
described by the distribution of features $\nu^K_k$ coincide and vary smoothly along the curve. 

Fig.~\ref{fig:texture_2} serves as a benchmark for our spline interpolation model on how well it can predict the texture patterns in comparison to the ground truth. Therein, the prescribed feature distributions are extracted from equally spaced still frames of a video showing the life cycle of the surface patch of a mango. As the mango peel goes from green to ripe and eventually rots away, not only the colors but the texture of the peel changes significantly. In the frame of generative texture synthesis our method (top rows on each panel) matches both structure and coloring of the actual textures (bottom row on each panel) at corresponding times.
\bibliographystyle{abbrv}     
\bibliography{bibliography}

\end{document}